\documentclass[10pt,a4paper]{amsart}
\usepackage{amsthm,amsfonts,amsmath,amssymb,latexsym,amscd}
\usepackage{epsfig,graphics,color}
\usepackage[matrix,arrow,curve]{xy}
\usepackage{hyperref}
\usepackage{mathrsfs}
\usepackage{verbatim}
\usepackage{bm}
\usepackage{enumerate}
\usepackage{tikz-cd}
\usepackage{mathtools}
\title{Chern-Simons theory and string topology}
\author{Kai Cieliebak}
\address{Universit\"at Augsburg, Universit\"atsstrasse 14, 86159 Augsburg, Germany}
\email{kai.cieliebak@math.uni-augsburg.de}
\author{Evgeny Volkov}
\address{Universitat Polit\`ecnica de Catalunya,
Av.~Dr.~Mara\~n\'on 44--50, 08028 Barcelona, Spain}
\email{evgeny.volkov@upc.edu} 
\theoremstyle{plain}
\newtheorem{theorem}{Theorem}[section]
\newtheorem{thm}[theorem]{Theorem}
\newtheorem{corollary}[theorem]{Corollary}
\newtheorem{cor}[theorem]{Corollary}
\newtheorem{proposition}[theorem]{Proposition}
\newtheorem{prop}[theorem]{Proposition}
\newtheorem{lemma}[theorem]{Lemma}
\newtheorem{lem}[theorem]{Lemma}

\theoremstyle{remark}

\newtheorem{remark}[theorem]{Remark}
\newtheorem{rem}[theorem]{Remark}

\newtheorem{definition}{Definition}

\newcommand{\Id}{{{\mathchoice {\rm 1\mskip-4mu l} {\rm 1\mskip-4mu l}
{\rm 1\mskip-4.5mu l} {\rm 1\mskip-5mu l}}}}
\newcommand{\id}{{\rm id}}
\newcommand{\ddt}{\frac{d}{dt}}

\newcommand{\ol}{\overline}

\newcommand{\p}{\partial}

\newcommand{\om}{\omega}
\newcommand{\Om}{\Omega}
\newcommand{\eps}{\varepsilon}
\newcommand{\into}{\hookrightarrow}
\newcommand{\la}{\langle}
\newcommand{\ra}{\rangle}
\newcommand{\wt}{\widetilde}
\newcommand{\wh}{\widehat}
\newcommand{\onto}{\twoheadrightarrow}
\newcommand{\N}{{\mathbb{N}}}
\newcommand{\Z}{{\mathbb{Z}}}
\newcommand{\R}{{\mathbb{R}}}
\newcommand{\C}{{\mathbb{C}}}

\newcommand{\F}{\mathbb{F}}

\newcommand{\bbb}{{\bf b}}

\newcommand{\m}{{\bf m}}
\newcommand{\n}{{\bf n}}
\newcommand{\G}{{\bf G}}

\newcommand{\kkk}{{\bf k}}
\newcommand{\bs}{{\bf s}}

\newcommand{\Aut}{{\rm Aut}}

\newcommand{\im}{{\rm im}\,}        

\newcommand{\Ad}{{\rm Ad}}

\newcommand{\conn}{{\rm conn}}

\newcommand{\inn}{{\rm int}}

\newcommand{\can}{{\rm can}}
\newcommand{\reg}{{\rm reg}}

\newcommand{\IBL}{{\rm IBL}}
\newcommand{\dIBL}{{\rm dIBL}}
\newcommand{\ext}{{\rm ext}}

\newcommand{\cross}{{\rm cross}}

\newcommand{\Flag}{{\rm Flag}}

\newcommand{\Ver}{{\rm Vert}}
\newcommand{\Edge}{{\rm Edge}}
\newcommand{\Leaf}{{\rm Leaf}}
\newcommand{\Hom}{{\rm Hom}}

\newcommand{\main}{{\rm main}}
\newcommand{\hidden}{{\rm hidden}}
\newcommand{\Tr}{{\rm Tr}}
\newcommand{\DD}{\mathcal{D}}

\newcommand{\CC}{\mathcal{C}}

\newcommand{\FF}{\mathcal{F}}

\newcommand{\HH}{\mathcal{H}}
\newcommand{\XX}{\mathcal{X}}
\renewcommand{\AA}{\mathcal{A}}

\newcommand{\YY}{\mathcal{Y}}
\newcommand{\PP}{\mathcal{P}}
\newcommand{\RR}{\mathcal{R}}

\newcommand{\GG}{\mathcal{G}}
\newcommand{\WW}{\mathcal{W}}

\newcommand{\cyc}{{\rm cyc}}
\newcommand{\g}{{\mathfrak g}}         

\newcommand{\fp}{{\mathfrak p}}
\newcommand{\fq}{{\mathfrak q}}
\newcommand{\ff}{{\mathfrak f}}
\newcommand{\fm}{{\mathfrak m}}

\newcommand{\HHH}{\mathbb{H}}
\newcommand{\W}{\mathcal{W}}
\newcommand{\D}{\mathcal{D}}

\newcommand{\K}{\mathbb{K}}

\newcommand{\Bl}{{\rm Bl}_+}

\begin{document}

\maketitle

\begin{abstract}
We construct chain-level $S^1$-equivariant string topology for each
simply connected closed manifold. This amounts to constructing a
Maurer-Cartan element for the canonical involutive Lie bialgebra (IBL)
structure on the dual cyclic bar complex of its de Rham cohomology
that is unique up to $\IBL_\infty$ gauge equivalence. The construction 
involves integrals over configuration spaces associated to trivalent
ribbon graphs, which can be seen as a version of perturbative
Chern-Simons theory in arbitrary dimension.  
\end{abstract}
%
\tableofcontents

\section{Introduction}\label{sec:intro}

Over the past 20 years there has been mounting evidence for the
equivalence of the following theories associated to a closed manifold
$M$: 
\begin{enumerate}[(A)]
\item string topology on the loop space of $M$;
\item holomorphic curve theory on the cotangent bundle of $M$;
\item Chern-Simons theory on $M$.
\end{enumerate}
All these theories come in various flavours: String topology can
refer to the based loop space, the free loop space, and/or the
$S^1$-equivariant theory on the free loop space. Holomorphic curves
can be compact with punctures, asymptotic markers, and/or Lagrangian
boundary conditions. Chern-Simons theory can be perturbative or
non-perturbative, and it involves a gauge group as well as harmonic
insertions and/or Wilson loops. Moreover, all the theories can be
considered at chain level or at the level of homology. 

Evidence for $(A)\Longleftrightarrow(B)$ arises for example from Viterbo's
isomorphism between symplectic homology (or wrapped Floer homology) of
the cotangent bundle and singular homology of the free (or based) loop
space~\cite{Viterbo-cotangent,AS,AS-corrigendum,AS2,SW,Kragh,Abouzaid-cotangent}; 
the conjectural equivalence between symplectic field theory on the
cotangent bundle and $S^1$-equivariant string topology on the free
loop space~\cite{Fukaya-Lag,Cieliebak-Latschev};
and the expression of knot contact homology and other holomorphic
curve invariants of knots in terms of open string topology~\cite{CELN,Ekholm-Ng-Shende}.

Evidence for $(B)\Longleftrightarrow(C)$ appears in the physics
literature on string field theory, see e.g.~\cite{Witten95}. However,
to our knowledge, these physical arguments have not yet been made
mathematically rigorous. 

It is the goal of this paper to directly establish an instance of the
equivalence $(A)\Longleftrightarrow(C)$. Our motivation comes from
symplectic field theory. The algebraic structure encoded by moduli
spaces of punctured holomorphic curves of arbitrary genus in the
cotangent bundle $T^*M$ was first described in~\cite{EGH} in the
formalism of Weyl algebras. We will use its reformulation
in~\cite{Cieliebak-Fukaya-Latschev} in terms of {\em
  $\IBL_\infty$-algebras} (involutive Lie bialgebras up 
to infinite homotopies, see~\S\ref{sec:IBLinfty}). According
to~\cite{Cieliebak-Latschev,Cieliebak-Fukaya-Latschev}, the same structure should appear on
chain-level $S^1$-equivariant string topology. 

One approach to chain-level string topology builds on the de Rham
chains introduced by Fukaya. This approach has been successfully
implemented by Irie~\cite{Irie} in the non-equivariant case, where it
leads to the structure of an algebra over the framed little disk operad
on a suitable chain model for the free loop space. 

We focus in this paper on the $S^1$-equivariant case and use a
different approach proposed in~\cite{Cieliebak-Fukaya-Latschev}.
For a closed oriented $n$-manifold $M$, let $\Om^*(M)$ denote the de
Rham complex of $M$ and $H_{dR}^*(M)$ its  
cohomology. According to~\cite{Cieliebak-Fukaya-Latschev}
(see also~\S\ref{sec:IBLinfty}), the degree shifted dual cyclic bar complex
$B^{\text{\rm cyc}*}H^*_{dR}(M)[2-n]$ carries a canonical $\IBL$ structure. 
Our first result
is

\begin{thm}\label{thm:existence-intro}
For any closed oriented $n$-manifold $M$, there exists a Maurer-Cartan
element on $B^{\text{\rm cyc}*}H^*_{dR}(M)[2-n]$
whose twisted homology equals the Connes cyclic cohomology of $\Om^*(M)$.  
\end{thm}

The construction of this Maurer-Cartan element depends on the choice
of a complement $\HH$ of $\im d$
in $\ker d$ and a special Hodge propagator (see~\S\ref{ss:blowupprop}).
Our second result concerns the independence of these choices.

\begin{thm}\label{thm:uniqueness-intro}
The Maurer-Cartan element in Theorem~\ref{thm:existence-intro} is
independent of choices up to $\IBL_\infty$ gauge equivalence. In
particular, the associated twisted $\IBL_\infty$ structure is
independent of choices up to $\IBL_\infty$ homotopy equivalence.
\end{thm}

\begin{remark}
A result very similar to Theorem~\ref{thm:existence-intro} has been
obtained previously 
and independently by Naef and Willwacher in~\cite{Naef-Willwacher}. 
Their approach builds for simply connected $M$ on a finite
dimensional Poincar\'e duality model for $\Om^*(M)$ provided by
Lambrechts and Stanley~\cite{Lambrechts-Stanley}, and for non-simply
connected $M$ on a dgca model for the configuration space of points on
$M$ constructed by Campos and Willwacher~\cite{Campos-Willwacher}. 
By contrast, our approach is more analytic, working directly with
configuration space integrals in the smooth setting and thereby
establishing connections to perturbative Chern-Simons theory as
discussed below. 
The article~\cite{Naef-Willwacher} does not contain an
invariance statement as in Theorem~\ref{thm:uniqueness-intro}.
It is shown in~\cite{Cieliebak-Hajek-Volkov} that, under some
additional hypotheses, the $\IBL_\infty$ structures 
in~\cite{Naef-Willwacher} and in Theorem~\ref{thm:existence-intro} are
homotopy equivalent. 
\end{remark}

\begin{rem}
Uniqueness of the Maurer-Cartan element up to $\IBL_\infty$ gauge
equivalence is the strongest possible kind of invariance. We do not
know, however, whether this is actually stronger than uniqueness of
the twisted $\IBL_\infty$ structure up to $\IBL_\infty$ homotopy equivalence.
\end{rem}

{\bf Relation to string topology. }
Let $M$ be a connected oriented closed manifold,
let $\Lambda=C^\infty(S^1,M)$ denote the free loop space,
$\Lambda_0\subset\Lambda$ the space of constant loops, and 
$q_0\in M\cong \Lambda_0$ a basepoint.
It is proved in~\cite{Cieliebak-Volkov-cyc} that Chen's iterated
integrals~\cite{Chen73,Chen77} induce a map 
\begin{equation}\label{eq:chen}
  \bar J_{\lambda}:H^{S^1}_*(\Lambda,q_0)
  \longrightarrow\overline{HC}^*_\lambda(\Om^*(M))
\end{equation}
between the reduced $S^1$-equivariant homology of $\Lambda$ and the
reduced Connes cyclic homology of $\Om^*(M)$,  
which is an isomorphism for $M$ simply connected. 

On the other hand, Chas and Sullivan have defined in~\cite{Chas-Sullivan04} 
a string bracket and cobracket giving an $\IBL$ structure on
$H_*^{S^1}(\Lambda,\Lambda_0)$.
It is proved in~\cite{Cieliebak-Volkov-stringtop} that this
canonically extends to an $\IBL$ structure on
$H^{S^1}_*(\Lambda,q_0)$, which is intertwined under the 
Chen map~\eqref{eq:chen} with the $\IBL$ structure on
$\overline{HC}^*_\lambda(\Om^*(M))$ induced from the twisted
$\IBL_\infty$ structure in Theorem~\ref{thm:existence-intro}. 
See also~\cite{Naef-Willwacher}.
Hence, in the simply connected case, the twisted $\IBL_\infty$ structure
in Theorem~\ref{thm:existence-intro} is indeed ``chain level string topology''
in the sense that its induced structure on homology agrees with the
one from string topology. 
Note that Theorems~\ref{thm:existence-intro} and~\ref{thm:uniqueness-intro} 
provide a canonical twisted $\IBL_\infty$-structure also in the
non-simply connected case; it would be interesting to explore its
relation to string topology and its implications for manifold topology.


\begin{rem}
The question on which version of loop space homology string topology
operations can be defined is surprisingly subtle. In the
non-equivariant case, $H_*(\Lambda,\chi(M)\cdot q_0)$ turns out to be the
unique space on which both the loop product and coproduct can be
defined~\cite{Cieliebak-Hingston-Oancea-PD}. 
By contrast, in the $S^1$-equivariant case, the string bracket and
cobracket can be defined on any space 
between $H_*^{S^1}(\Lambda,\chi(M)\cdot q_0)$ and
$H_*^{S^1}(\Lambda,\Lambda_0)$; see~\cite{Cieliebak-Volkov-stringtop}.
\end{rem}

{\bf Relation to Chern-Simons theory. }
The proofs of the above theorems involve constructions similar to
those in perturbative Chern-Simons theory, which we will now describe. 

To prove Theorem~\ref{thm:existence-intro}, we pick a 
complement $\HH$ of $\im d$ in $\ker d$ and a special Hodge
propagator (see~\S\ref{ss:blowupprop}).  
The integral kernel of the Hodge propagator is a smooth $(n-1)$-form
$G$ on $(M\times M)\setminus\Delta_2$ which blows up along the diagonal
$\Delta_2\subset M\times M$.
We call $\HH\subset\Om^*(M)$ the space of {\em harmonic forms}. 
Consider now a trivalent ribbon graph $\Gamma$, with associated ribbon
surface $\Sigma_\Gamma$ of genus $g\geq 0$ with $\ell\geq 1$ boundary components,
such that $\Gamma$ has $s_b\geq 1$ leaves on the $b$-th boundary loop
for $b=1,\dots,\ell$. Given an ordered collection $\alpha=\{\alpha^b_j\}$ of
harmonic forms $\alpha^b_j\in\HH$ for $b=1,\dots,\ell$ and
$j=1,\dots,s_b$ we define a real number
\begin{equation}\label{eq:MC-intro}
  \fm_\Gamma(\alpha) := \pm 
  \int_{M^k\setminus\Delta}G^e\wedge\prod_{b,j}\alpha^b_j.
\end{equation}
Here $k$ and $e$ are the numbers of vertices and edges of
$\Gamma$, respectively, $M^k\setminus\Delta$ is
the configuration space of $k$ distinct points on $M$, and we
assign factors of $G$ to the edges and $\alpha^b_j$ to the leaves. 
See Figure~\ref{fig:insertforms}.
\begin{figure}
\begin{center}
\includegraphics[width=\textwidth]{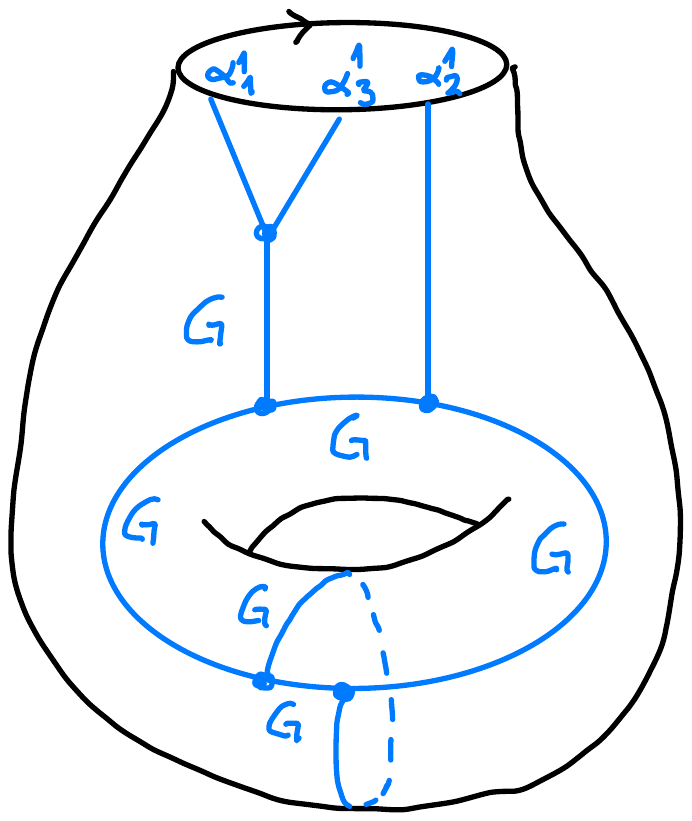}
\vspace{-1.5cm}
\caption{Definition of the configuration space integral $\fm_\Gamma(\alpha)$}
\label{fig:insertforms} 
\end{center}
\end{figure}
Then the sums over isomorphism classes of trivalent ribbon graphs
$$
  \fm_{g,\ell} := \frac{1}{\ell!} \sum_{\Gamma\in\RR_{\ell,g}^3}\fm_{\Gamma}\in 
   B^{\cyc *}\HH[3-n]^{\otimes\ell}
$$
define the Maurer-Cartan element $\m=\{\fm_{g,\ell}\}$. 
To make this argument rigorous, we need to resolve the following
technical difficulties.

{\em Existence of the integrals~\eqref{eq:MC-intro}. } 
For this, we choose $G$ so that it extends smoothly to the oriented
real blow-up of $M\times M$ along the diagonal, and we
replace $M^k\setminus\Delta$ by a suitable compactified configuration
space $\CC_\Gamma$ on which the form $G^e\wedge\prod_{b,j}\alpha^b_j$
becomes continuous and thus integrable.
Our compactification is closely related to the ones by Bott and
Taubes~\cite{Bott-Taubes} and by Fulton and
MacPherson~\cite{Fulton-MacPherson}. 

{\em Signs in~\eqref{eq:MC-intro}. }
Finding the correct signs in~\eqref{eq:MC-intro} turned out to be the
hardest part of this project. We will define the signs and prove
their basic invariance properties in~\S\ref{sec:opforms},
and carry out the proof in reasonable detail in~\S\ref{sec:proofMC}, 
while for some lengthy sign computations (only relevant for even $n$)
we refer to~\cite{Volkov-thesis}.

{\em Proof of the Maurer-Cartan equation. }
Although the compactified configuration space $\CC_\Gamma$ is not a
manifold with corners, Stokes' theorem still holds on it
(see~\S\ref{sec:blow-up} and~\cite{Cieliebak-Volkov-stringtop}).   
By a duality argument, this yields the Maurer-Cartan equation, except
for possible error terms coming from additional boundary components
(``hidden faces'') of $\CC_\Gamma$. We prove
in~\cite{Cieliebak-Volkov-stringtop} (and recall here
in~\S\ref{ss:CGamma}) that the integrals over hidden faces vanish, by a symmetry
argument similar to the one in~\cite[Lemma 2.2]{Kontsevich-feynman}. 

The proof of Theorem~\ref{thm:uniqueness-intro} uses some ideas from
symplectic field theory. In~\S\ref{ss:gaugegen}, we derive from the
differential Weyl algebra formalism in~\cite{EGH} a criterion for
gauge equivalence of two Maurer-Cartan elements $\fm_0,\fm_1$. This
criterion (Lemma~\ref{lem:gaugeIBLinfty}) involves a smooth path
$(\fm_t)_{t\in[0,1]}$ of Maurer-Cartan elements connecting $\fm_0,\fm_1$
and a smooth path of ``primitives'' $(\bbb_t)_{t\in[0,1]}$. 
In~\S\ref{ss:gaugegen}, we use suitable integrals over configuration
spaces from~\S\ref{sec:opforms} to define the paths $\fm_t,\bbb_t$ and
prove that they define a gauge equivalence.  

\smallskip

The integrals over configuration spaces in~\eqref{eq:MC-intro} are of
the same type as those appearing in the perturbative expansion of the
Chern-Simons path integral associated to a closed oriented
$3$-manifold~\cite{Witten-jones,Bar-Natan95,Sawon,Axelrod-Singer-II,Cattaneo-Mnev}.
In view of this, O.\,Gwilliam has suggested that our theory should
correspond to the large $N$ limit of $U(N)$ Chern-Simons theory as in
the article by Ginot, Gwilliam, Hamilton
and Zeinalian~\cite{Ginot-Gwilliam-Hamilton-Zeinalian}.
We discuss the relation of our configuration space integrals to physical
Chern-Simons theory in more detail in~\S\ref{sec:physics}.

\bigskip

  









\centerline{\bf Acknowledgements}

This project greatly benefited from discussions with many people
including
Alberto Cattaneo,
Tobias Ekholm,
Kenji Fukaya,
Owen Gwilliam,
Pavel H\'ajek,
Branislav Jur\v{c}o,
Janko Latschev,
Pavel Mn\"{e}v,
Ivo Sachs,
Vivek Shende,
and Bruno Vallette.

\section{Ribbon graphs}\label{sec:graphs}

In this section we introduce the notions of ribbon graphs and their
(extended) labellings as they are needed in this paper. Moreover, 
we discuss the disjoint union, cutting, and gluing of ribbon graphs,
as well as the duality operation on trivalent ribbon graphs. 

\subsection{Basic definitions}\label{ss:ribbonbasicdef}

In this paper, by a {\em ribbon graph} (or simply a {\em graph}) we
mean a finite 
graph $\Gamma$ 
with a cyclic ordering of the adjacent edges at each vertex and with
some degree $1$ vertices removed.\footnote{
In contrast to~\cite{Cieliebak-Fukaya-Latschev} we do not require $\Gamma$ to be connected.}
We assume that at most one vertex is removed from each edge, and we
will refer to the edges with a vertex removed not as edges but as {\em
  leaves}, so that an {\em edge} still ends in two vertices.
We denote by $d_v$ the {\em degree} of a vertex, i.e.~the number of
edges or leaves adjacent to $v$.
We denote the sets of vertices, edges and leaves of $\Gamma$ by $\Ver(\Gamma)$,
$\Edge(\Gamma)$ and $\Leaf(\Gamma)$ and their cardinalities by
$$
   k=|\Ver(\Gamma)|,\qquad e=|\Edge(\Gamma)|,\qquad s=|\Leaf(\Gamma)|.
$$
The ribbon graph $\Gamma$ can be thickened in a unique way 
(up to orientation preserving homeomorphism)
to a compact oriented surface $\Sigma_\Gamma$ with boundary such that
each leaf ends on the boundary $\p\Sigma_\Gamma$.
We require that $\Gamma$ has at least one vertex and
satisfies the following condition:
\begin{equation}\label{eq:one-boundary-vertex}
\text{Each boundary component of $\Sigma_\Gamma$ has at least one leaf
  ending on it.}
\end{equation}
Note that the boundary components of $\Sigma_\Gamma$ induce additional
structure on the set of leaves: it gets subdivided into
subsets according to the boundary components, and each subset obtains
a cyclic order according to the boundary orientation. 
We denote by $s_b$ the number of leaves ending on boundary component
$b$, so that
$$
  s = s_1+\dots+s_\ell.
$$
The {\em signature} of $\Gamma$ is $(k,\ell,g)$ where
$k=|\Ver(\Gamma)|$, $\ell$ is the number of boundary
components $\Sigma_\Gamma$, and $g$ is its genus defined in terms of
the Euler characteristic by
\begin{equation}\label{eq:chi}
   k-e = \chi(\Gamma) = \chi(\Sigma_\Gamma) = 2-2g-\ell. 
\end{equation}
A {\em flag} in $\Gamma$ is a pair $(v,t)$ consisting of a 
vertex $v$ and an adjacent 
edge or leaf $t$. 
A flag $(v,t)$ is called {\em interior} if $t$ is an edge, and {\em
  exterior} if $t$ is a leaf. 
We denote the set of flags of $\Gamma$ by
$$
  \Flag(\Gamma) = \Flag_\inn(\Gamma)\amalg \Flag_\ext(\Gamma).
$$
For our purposes it will be
convenient to describe a ribbon graph in terms of its flags: a leaf
corresponds to a single flag $x$, an edge corresponds to an
unordered pair $\{x,y\}$ of flags, and a vertex corresponds to a cyclically ordered
tuple $[x_1,\dots,x_d]$ of flags. An oriented edge corresponds to an
ordered pair $(x,y)$, and a vertex with an ordering of the
adjacent edges corresponds to an ordered tuple $(x_1,\dots,x_d)$. 

An {\em isomorphism} $\Gamma\to\Gamma'$ between ribbon graphs is a
bijection $\Flag(\Gamma)\to \Flag(\Gamma')$ mapping leaves to leaves,
vertices to vertices, edges to edges, and preserving the cyclic orderings at the
vertices. It induces a homeomorphism $\Sigma_{\Gamma}\to\Sigma_{\Gamma'}$.
Note that a ribbon graph may have nontrivial {\em automorphisms}
(i.e.~self-isomorphisms). These will disappear when we consider
labelled graphs in the next subsection. 

{\bf Trivalent ribbon graphs. }
Of particular importance in this paper will be ribbon graphs that are
{\em trivalent}, i.e., each vertex has degree $3$. Note that in this
case condition~\eqref{eq:one-boundary-vertex} rules out self-loops,
i.e.~edges with both ends on the same vertex. Trivalency of $\Gamma$ 
gives $2e+s=3k$, and eliminating $e$ from this equation
and~\eqref{eq:chi} expresses $k$ in terms of $\ell,g,s$ as
\begin{equation}\label{eq:k}
   k = 2(2g+\ell-2)+s.
\end{equation}
We will call $(\ell,g)$ the {\em type} of the trivalent ribbon graph $\Gamma$.

{\bf Marked ribbon graphs. }
We say that a connected ribbon graph $\Gamma$ is {\em marked} if an
edge is chosen, and {\em o-marked} if the marked edge is in addition
oriented. A (o-)marked graph will be denoted by $\wh\Gamma=(\Gamma,l)$ with $l$
the (oriented) marked edge.

\subsection{Labellings, edge and vertex orders}\label{ss:basiccomb}


\begin{definition}\label{def:labelling}
Consider the following additional data on a ribbon graph $\Gamma$ of
signature $(k,\ell,g)$:
\begin{enumerate}
\item a numbering of the boundary components of $\Sigma_\Gamma$ by
   $1,\dots,\ell$; 
\item a numbering of the leaves ending on the $b$-th boundary
  component by $1,\dots,s_b$ compatible with the cyclic order given by the orientation;
\item a numbering of the vertices by $1,\dots,k$; 
\item a numbering of the flags at each vertex compatible with the
  cyclic order given by the orientation; 
\item a numbering of the edges by $1,\dots,e$;
\item an orientation of each edge.
\end{enumerate}
We call the data (i) and (ii) a {\em labelling}, the additional data
(iii) -- (vi) an {\em extension of the labelling}, and all data
together an {\em extended labelling}. 
Note that a labelling is uniquely determined by an ordered set of
$\ell$ leaves, the first ones on their boundary components. 
\end{definition}

\begin{remark}\label{rem:labelling}
This notion of a labelling differs from the one
in~\cite{Cieliebak-Fukaya-Latschev} where a ``labelling'' comprises items
(i)--(iv). The reason for this is that the operations associated to
ribbon graphs in~\cite{Cieliebak-Fukaya-Latschev} require a choice of
items (i)--(iv), whereas the operations in this paper require only a choice of items (i)--(ii). 
Given $k,\ell\geq 1$ and $g\geq 0$,  
we denote by $RG_{k,\ell,g}$ the set of isomorphism classes of connected
ribbon graphs of signature $(k,\ell,g)$ together with choices of items
(i)--(iv) in Definition~\ref{def:labelling}.
\end{remark}

An {\em isomorphism} $\Gamma\to\Gamma'$ between labelled ribbon graphs
is an isomorphism of ribbon graphs which preserves the labellings. In
other words, the induced homeomorphism $\Sigma_{\Gamma}\to\Sigma_{\Gamma'}$ 
matches the numberings of the boundary components and leaves. 

\begin{lemma}\label{lem:no-auto}
If an automorphism of a connected ribbon graph fixes at least one flag,
then it must be the identity.
\end{lemma}

\begin{proof}
Let $\phi$ be an automorphism of a connected ribbon 
graph $\Gamma$ that fixes a flag. To show that $\phi$ is the identity, 
we make the following two observations about the action of 
$\phi$ on $\Gamma$.
\begin{itemize}
 \item [(V)] Let $[x_1,\dots,x_d]$
 be a vertex of $\Gamma$ and assume that 
 $\phi(x_{j_0})=x_{j_0}$ for some 
 $j_0\in\{1,\dots,d\}$. Then 
 $\phi(x_j)=x_j$ for all 
 $j\in\{1,\dots,d\}$.
\end{itemize}
To see this, note that as an isomorphism of a {\em ribbon} graph
$\phi$ must preserve the ribbon structure of $\Gamma$, 
and in particular the cyclic order in $[x_1,\dots,x_d]$. Therefore, fixing 
one of these flags implies fixing all of them.
\begin{itemize}
 \item [(E)] Let $\{x,y\}$ be an edge
 and assume that $\phi(x)=x$. Then $\phi(y)=y$.
\end{itemize}
To see this note that $\{\phi(x),\phi(y)\}$ is also an edge 
of $\Gamma$. Since one flag of an edge 
uniquely determines the other one 
and $\phi(x)=x$, we get $\phi(y)=y$.

Observations (V) and (E) together with connectedness of $\Gamma$ imply the 
desired conclusion.
\end{proof}

\begin{lemma}\label{lem:no-auto1}
A labelled ribbon graph
has no nontrivial automorphisms. 
\end{lemma}

\begin{proof}
Since an automorphism of a labelled ribbon graph respects the
labelling, it must map each boundary component of the graph to itself
and fix all exterior flags of the graph. Lemma~\ref{lem:no-auto1}
thus follows by applying Lemma~\ref{lem:no-auto} to all connected 
components of the labelled ribbon graph in question. 
\end{proof}

For a ribbon graph $\Gamma$, an extended labelling gives us two bijective maps  
$$ 
   O_e:\{1,\dots,2e,\dots,|\Flag(\Gamma)|\}\longrightarrow
   \Flag(\Gamma),\quad O_v:\{1,\dots,
   |\Flag(\Gamma)|\}\longrightarrow \Flag(\Gamma). 
$$
The first one is called the {\em edge order} on $\Flag(\Gamma)$; it is
determined by items (i), (ii), (v) and (vi) of the extended labelling,
mapping the numbers $1,\dots,2e$ to the flags corresponding to edges
according to (v) and (vi), and the remaining numbers to the leaves
according to (i) and (ii). 
The second one is called the {\em vertex order}; it is determined by
items (iii) and (iv), numbering the flags in the order (iii) of vertices
and using the ordering (iv) at each vertex.   
By composition we obtain the {\em reordering permutation}
\begin{equation}\label{eq:reod}
   \bar R_\Gamma:=O_v^{-1}\circ O_e:\{1,\dots,|\Flag(\Gamma)|\}\longrightarrow 
   \{1,\dots,|\Flag(\Gamma)|\}.
\end{equation}
The map $\bar R_\Gamma$ behaves as follows under changes of the
extended labelling of $\Gamma$. A change in (i), (ii), (v) and (vi) leads to  
precomposition of the edge order $O_e$ with some permutation 
$\eta\in S_{|\Flag(\Gamma)|}$, whereas a change in (iii) and (iv) leads to precomposition 
of the vertex order $O_v$ with a some permutation $\sigma^{-1}$, so
altogether $\bar R_\Gamma$ is replaced by $\sigma\circ \bar R_\Gamma\circ \eta$. 
We say that $\eta$ and $\sigma$ act on a graph $\Gamma$ with extended labelling 
and give us a graph $\sigma\Gamma\eta$.

\begin{definition}\label{def:R}
The set of isomorphism classes of labelled connected trivalent ribbon graphs
of genus $g\ge 0$ with $\ell\ge 1$ boundary components 
will be denoted by $\RR_{\ell,g}$.
We denote by $\RR^m_{\ell,g}$
(resp.~$\RR^{om}_{\ell,g}$) the set of isomorphism classes of labelled 
trivalent ribbon graphs with a marked (resp.~o-marked)
edge. Recall that marked graphs are by definition connected.
\end{definition}

In general we do not impose any relation between the labelling
and the o-marked edge. However, the presence of an o-marked edge
gives rise to a class of {\em special labellings} that we describe next.


{\bf Special labellings of marked ribbon graphs. }
Consider a labelled ribbon graph $\Gamma$ with an o-marked edge $l$.
The labelling is called {\em special} if it has the following
properties (see the right hand side of Figure~\ref{fig:cut-glue}):
\begin{itemize}
\item The boundary component that runs parallel to $l$ (in the
  direction of its orientation) on its left has number $1$ in the
  numbering of boundary components. The first leaf on this boundary
  component encountered when moving in the direction of the boundary
  orientation from the portion to the left of $l$ has number $1$ in
  the numbering of its leaves.
\item If the boundary component that runs parallel to $l$ on its right
  differs from the one on its left, then it has number $2$ in the
  numbering of boundary components. In this case, the first leaf on this boundary
  component encountered when moving in the direction of the boundary
  orientation from the portion to the right of $l$ has number $1$ in
  the numbering of its leaves. 
\item  If cutting $\Gamma$ along $l$ yields a disconnected
  graph $\Gamma_1\amalg\Gamma_2$, where $\Gamma_1$ contains the initial
  point of $l$ and $\Gamma_2$ its endpoint, then the boundary
  components of $\Gamma$ other than the one parallel to $l$ are
  numbered such that the ones on $\Gamma_1$ come before the ones on
  $\Gamma_2$. 
\end{itemize}
The set of isomorphism classes of marked trivalent graphs of type
$(\ell,g)$ with a special labelling will be denoted by 
$\RR_{\ell,g}^{oms}$, where the ``s'' in the superscript stands
for ``special labelling''.

{\bf The sign exponent $\eta_3$. }
Let $\Gamma$ be an extended labelled ribbon graph. We recall the basic
properties of the sign exponent $\eta_3(\Gamma)$ defined
in~\cite[Appendix A]{Cieliebak-Fukaya-Latschev} for connected
$\Gamma$. 
The sign exponent $\eta_3(\Gamma)$ does not depend on items (ii) and
(iv) in Definition~\ref{def:labelling}, and it changes by $1$ under
each of the following operations: 
\begin{itemize}
\item swapping the order of two adjacent boundary components in (i);
\item swapping the order of two adjacent vertices in (iii);
\item swapping the order of two adjacent edges in (v);
\item flipping the orientation of an edge in (vi).
\end{itemize}
For a disconnected graph $\Gamma=\Gamma_1\amalg\Gamma_2$
with connected $\Gamma_1$ and $\Gamma_2$ we define 
$$
  \eta_3(\Gamma):=\eta_3(\Gamma_1)+\eta_3(\Gamma_2)
$$
for the canonical disjoint union extended labelling of $\Gamma$, and we
extend the definition to all other extended labellings using the four
items above. 


\subsection{Operations on ribbon graphs}\label{ss:op-graphs}
Recall that by a ``graph'' we will always mean a ribbon graph as in~\S\ref{ss:ribbonbasicdef}.
We will describe three operations on such graphs: disjoint union, cutting, and gluing.

{\bf Disjoint union.}
Let $\Gamma_1$ and $\Gamma_2$ be two connected (extended) labelled graphs. Then the 
disjoint union $\Gamma_1\amalg\Gamma_2$ inherits a natural (extended)
labelling, putting $\Gamma_1$ before $\Gamma_2$ in all numberings. We
will assume this (extended) labelling on $\Gamma_1\amalg\Gamma_2$
unless otherwise specified. 

{\bf Cutting and gluing.}
Given a marked graph $(\Gamma,l)$ we define a new graph $\Gamma\setminus
l$ by cutting open the marked edge $l$. Formally, this just means that
we remove the corresponding unordered pair $l=\{u,v\}$ from the set of
edges, so the interior flags $u,v$ become exterior ones. Note
that the resulting graph $\Gamma\setminus l$ is not marked, but it has
a distinguished unordered pair of exterior flags $u,v$. 

Conversely, given a graph $\Gamma$ with a distinguished
unordered pair of exterior flags $u,v$ we define a new graph
$\Gamma\cup\{u,v\}$ by gluing $u$ and $v$ to a new edge.
Formally, this just means that we add the unordered pair $l:=\{u,v\}$ to
the set of edges, so the exterior flags $u,v$ become interior
ones. Note that the resulting graph $\Gamma\cup\{u,v\}$ has no more
distinguished unordered pair of exterior flags, by it is marked by the
interior edge $l=\{u,v\}$.
The resulting cutting and gluing operations
$$
   \{\text{marked graphs}\} \longleftrightarrow \{\text{graphs with
     a given unordered pair of exterior flags}\} 
$$
are clearly inverse to each other, and they induce operations
$$
   \{\text{o-marked graphs}\} \longleftrightarrow \{\text{graphs with
     a given ordered pair of exterior flags}\} 
$$
Note that cutting a graph $\Gamma$ of type $(\ell,g)$ increases its Euler
characteristic 
$$
\chi_{\ell,g}=2-2g-\ell
$$
by $1$ and therefore results
in a graph of type $(\ell-1,g)$ or $(\ell+1,g-1)$. 
In the first case two of the boundary components of 
$\wh\Gamma$ give rise to a boundary component of the cut graph.
The other boundary components of $\wh\Gamma$ are in natural
bijective correspondence with those of the cut graph.
The space of isomorphism classes of o-marked labelled trivalent graphs 
with this property will be denoted by $\RR_{\ell,g,12}^{om}$.
In the second case one boundary component of 
$\wh\Gamma$ 
gives rise to two boundary components of the cut graph and the others
are in natural bijective correspondence. 
This case splits into two subcases according to whether the cut graph
is connected or not. The corresponding spaces of isomorphism 
classes of o-marked labelled graphs are denoted by
$\RR_{\ell,g,1c}^{om}$ and $\RR_{\ell,g,1dc}^{om}$ respectively,
where $c$ stands for ``the cut graph is connected'' and 
dc stands for ``the cut graph is disconnected''.

The three kinds of o-marked labelled graphs are shown on the right
hand sides in Figure~\ref{fig:cut-glue}. So we have a decomposition
(an analogous one holds for $\RR_{\ell,g}^{m}$)
\begin{equation}\label{eq:R-disjoint-union}
   \RR_{\ell,g}^{om} = \RR_{\ell,g,1c}^{om}\amalg 
   \RR_{\ell,g,1dc}^{om}\amalg \RR_{\ell,g,12}^{om}.   
\end{equation}
\begin{figure}
\begin{center}
\includegraphics[width=\textwidth]{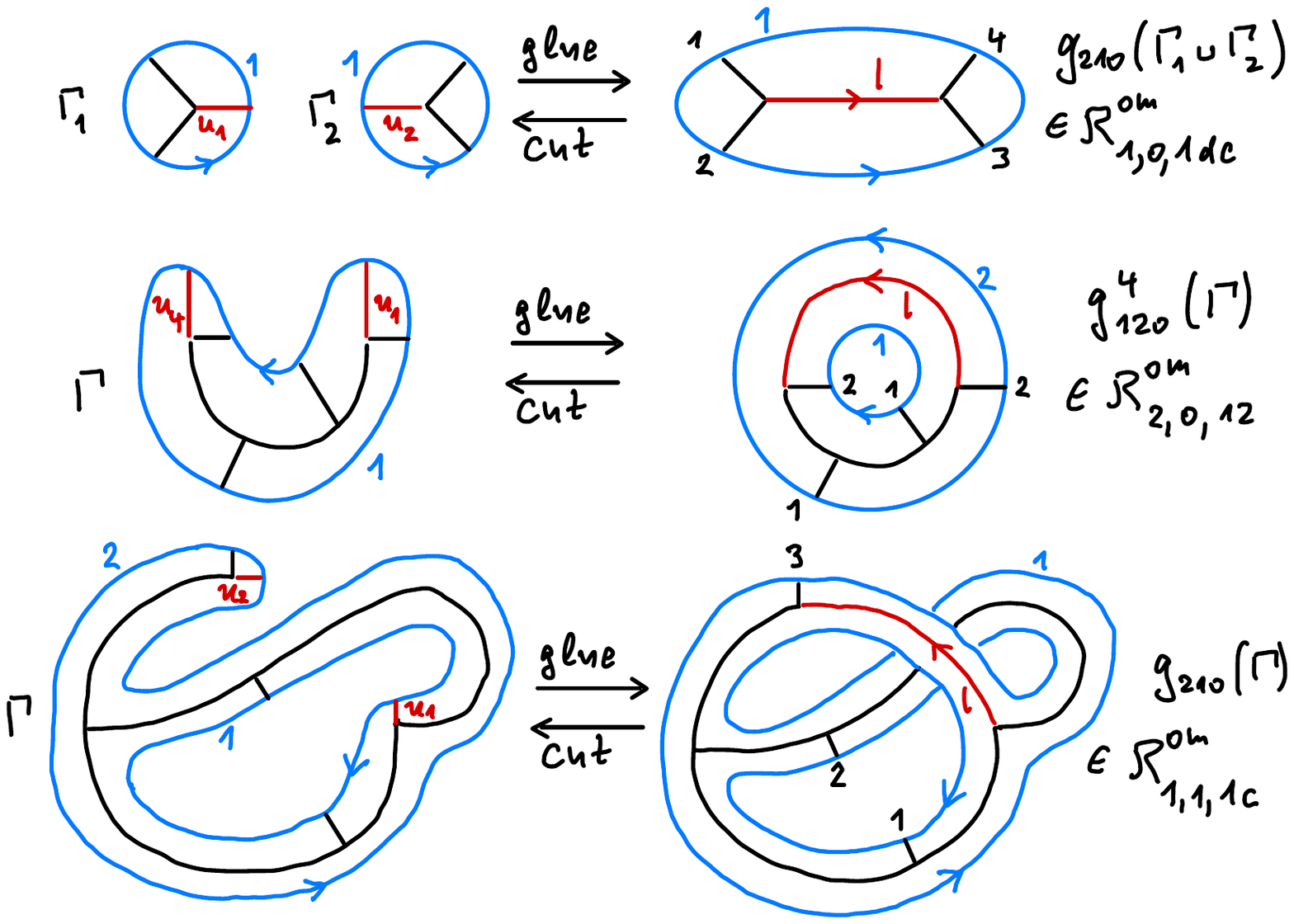}
\caption{Cutting and gluing}
\label{fig:cut-glue} 
\end{center}
\end{figure}

Let $\RR_{\ell_1,\ell_2,g_1,g_2,1dc}^{om}$ denote the subset of $\RR_{\ell,g,1dc}^{om}$
defined by the condition that the $i$-th component of the cut graph
has $\ell_i$ boundary components and genus $g_i$ (with $\ell_1+\ell_2=\ell+1$ and $g_1+g_2=g$).

Observe that 
\begin{equation}\label{eq:ell12ell}
 \coprod_{\substack{\ell_1+\ell_2=\ell+1 \\ g_1+g_2=g}}\RR_{\ell_1,\ell_2,g_1,g_2,1dc}^{om}=\RR_{\ell,g,1dc}^{om}.
\end{equation}
Next we discuss how (extended) labellings get transferred under these operations.
Consider a graph $\Gamma$ with a given ordered pair of
exterior flags $u,v$. Suppose we are given a labelling for
$\Gamma$ such that $u$ comes before $v$ in the ordering of the
exterior flags, and denote the difference of their positions by $|v-u|$.
Assume that $u,v$ lie either on the same boundary component (in which
case we require that $u,v$ are not adjacent in the cyclic order on
that component), or on adjacent ones in the ordering of boundary components.  
The glued graph $\Gamma\cup(u,v)$ inherits a labelling by requalifying
the flags $u$ and $v$ from ``exterior'' to ``interior'', and keeping
the ordering of the remaining exterior flags.  

Assume now that an extension of the labelling of $\Gamma$ is given. 
Then the labelling of the glued graph $\Gamma\cup(u,v)$ inherits an extension by
putting the new oriented edge $l=(u,v)$ in first position for the edge
order.
Note that the vertex order of flags remains the same under the gluing operation.
Since we must move $u$ past $|v-u|-1$ exterior flags to put it next to
$v$ in the edge order, the sign exponents of the reordering maps of
the two graphs are related by 
\begin{equation}\label{gluing-sign}
   \bar R_{\Gamma\cup(u,v)} \equiv \bar R_\Gamma + |v-u|-1.
\end{equation}

We will use the following three special cases of the gluing operation.
See Figure~\ref{fig:cut-glue}, where the numbers denote the positions
of the boundary components and the leaves on a boundary component
corresponding to the labelling.
In the following discussion the graphs $\Gamma$, $\Gamma_i$ are {\em
trivalent and connected} and equipped with an (extended) labelling. 

First, let $\Gamma$ be of type $(\ell,g)$ with $\ell\geq 2$.
Let $s_1,s_2$ be the numbers of exterior flags on the first
two boundary components and assume that $s_1+s_2\geq 3$. Let $u_1,u_2$
be the first exterior flags on the first two boundary components
and glue them to the o-marked labelled graph 
\begin{equation}\label{eq:g120}
  g_{210}(\Gamma):=\Gamma\cup(u_1,u_2).
\end{equation}
The set of isomorphism classes of labelled trivalent graphs 
that arise this way will be denoted by $\RR_{\ell-1,g+1,1c}^{oms}$,
where the ``s'' in the superscript again stands for ``special
labelling'' in the sense of \S\ref{ss:basiccomb}. 
In view of~\eqref{gluing-sign} the sign exponents of the reordering maps are related by
$$
   \bar R_{g_{210}(\Gamma)}\equiv\bar R_\Gamma+s_1-1.
$$
Next, let $\Gamma_1$ and $\Gamma_2$ be of types $(\ell_1,g_1)$ and
$(\ell_2,g_2)$, respectively.
Pick the disjoint union labelling on $\Gamma_1\amalg \Gamma_2$ and apply 
to it the shuffle permutation moving the first boundary component
of $\Gamma_2$ to the second position (next to the first boundary
component of $\Gamma_1$). As a result, the first two boundary
components of $\Gamma_1\amalg \Gamma_2$ now
correspond to the first boundary components of $\Gamma_1$ and $\Gamma_2$.
Assume that the numbers $s_1,s_2$ of exterior flags on these boundary
components satisfy $s_1+s_2\geq 3$ and define
$g_{210}(\Gamma_1\amalg\Gamma_2)$ as above. 
The set of isomorphism classes of labelled trivalent graphs 
that arise this way will be denoted by 
$\RR_{\ell,g,1dc}^{oms}$ with $\ell=\ell_1+\ell_2-1$ and $g=g_1+g_2$.
Define 
$$
\RR_{\ell_1,\ell_2,g_1,g_2,1dc}^{oms}:=
\RR_{\ell,g,1dc}^{oms}\cap \RR_{\ell_1,\ell_2,g_1,g_2,1dc}^{om}.
$$
The sign exponents of the reordering maps are again related by
$$
   \bar R_{g_{210}(\Gamma_1\amalg\Gamma_2)}\equiv \bar R_{\Gamma_1\amalg\Gamma_2}+s_1-1.
$$
Finally, let $\Gamma$ be of type $(\ell,g)$ 
with $s_1\geq 4$ exterior flags on the first boundary component. Let
$u_1,u_j$ be the first and $j$-th flag on the first boundary component
of $\Gamma$, for some $j\in \{3,\dots,s_1-1\}$, and glue them to the
o-marked labelled graph
$$
  g_{120}^j(\Gamma):=\Gamma\cup(u_1,u_2).
$$
Note that the first boundary component of $\Gamma$ gives rise to the
first two of $g_{120}^j(\Gamma)$. 
The set of isomorphism classes of trivalent labelled graphs 
that arise this way will be denoted by 
$\RR_{\ell+1,g,12}^{oms}$. 
In view of~\eqref{gluing-sign} the
sign exponents of the reordering maps are related by
$$
   \bar R_{g_{120}^j(\Gamma)}\equiv\bar R_\Gamma+j-1.
$$
\begin{definition}\label{def:Rlgs}
Let $(s_1,\dots,s_\ell)$
be an ordered set of $\ell$ positive integers. Then the (necessarily finite) subset 
$$
\RR_{\ell,g}(s_1,\dots,s_\ell)\subset 
\RR_{\ell,g}
$$
consists of isomorphism classes of connected trivalent ribbon graphs with
$s_b$ leaves on the $b$-th boundary component. A similar notation 
applies to other sets of graphs.
\end{definition}

With this notation, the gluing operation yields bijections (whose
inverse is always given by the cutting operation)
\begin{equation}\label{eq:glue1}
   g_{210}: \RR_{\ell+1,g-1}(s_1,\dots,s_{\ell+1})\stackrel{\cong}\longrightarrow
   \RR_{\ell,g,1c}^{oms}(s_1+s_2-2,\dots,s_{\ell+1}), 
\end{equation}
\begin{equation}\label{eq:glue2}
\begin{aligned}
   g_{210}: &
   \RR_{\ell_1,g_1}(s_1,\dots,s_{\ell_1})\times 
   \RR_{\ell_2,g_2}(t_1,\dots,t_{\ell_2}) \cr
   &\stackrel{\cong}\longrightarrow
   \RR_{\ell_1,\ell_2,g_1,g_2,1dc}^{oms} 
   (s_1+t_1-2,s_2,\dots,s_{\ell_1},t_2,\dots,t_{\ell_2}), 
\end{aligned}
\end{equation}
\begin{equation}\label{eq:glue3}
   g_{120}^j: \RR_{\ell-1,g}(s_1,\dots,s_{\ell-1})\stackrel{\cong}\longrightarrow
   \RR_{\ell,g,12}^{oms}(j-2,s_1-j,s_2,\dots,s_{\ell-1}) 
\end{equation}
for $3\leq j\leq s_1-1$.

For $\ell\in\N$ let $S_\ell$ be the group of permutations of
$\{1,\dots,\ell\}$. For a decomposition $\ell_1+\dots+\ell_r=\ell$ let 
$Sh_{\ell_1,\dots,\ell_r}\subset S_\ell$ be the set of $\sigma\in S_\ell$ 
satisfying $\sigma(i)<\sigma(j)$ whenever
$\ell_1+\cdots+\ell_{s-1}<i<j\leq \ell_1+\cdots+\ell_s$ for some
$s\in\{1,\dots,r\}$ (these permutations are called {\em shuffles}). 
Note that the number of elements of $Sh_{\ell_1,\dots,\ell_r}$ equals
$$
   |Sh_{\ell_1,\dots,\ell_r}| = \frac{\ell!}{\ell_1!\cdots\ell_r!}.
$$
For $s\in \N$ we write the cyclic group as $\Z_s:=\Z/s\Z$. With these
notations we have bijections
\begin{equation}\label{eq:slab1}
   Sh_{1,\ell-1}\times \Z_{s_1}\times 
   \RR_{\ell,g,1c}^{oms}(s_1,\dots,s_\ell)
   \stackrel{\cong}\longrightarrow 
   \RR_{\ell,g,1c}^{om}(s_1,\dots,s_\ell),
\end{equation}
\begin{equation}\label{eq:slab2}
   Sh_{1,\ell_1-1,\ell_2-1}\times \Z_{s_1}\times \RR_{\ell_1,\ell_2,g_1,g_2,1dc}^{oms}(s_1,\dots,s_\ell)
   \stackrel{\cong}\longrightarrow 
   \RR_{\ell_1,\ell_2,g_1,g_2,1dc}^{om}(s_1,\dots,s_\ell),
\end{equation}
\begin{equation}\label{eq:slab3}
   Sh_{1,1,\ell-2}\times \Z_{s_1}\times\Z_{s_2}\times 
   \RR_{\ell,g,12}^{oms}(s_1,\dots,s_\ell)
   \stackrel{\cong}\longrightarrow 
   \RR_{\ell,g,12}^{om}(s_1,\dots,s_\ell).
\end{equation}
Here the shuffles move the first boundary component to an arbitrary
position in the first two cases, and the first two boundary components
to arbitrary positions in the third case. In addition, in the second
case they shuffle the two sets of boundary components of cardinalities 
$\ell_i-1$, $i=1,2$, corresponding to the boundary components on the $i$-th
connected component of the cut graph.
The cyclic groups
rotate the numberings of the exterior vertices on the first boundary
component in the first two cases, and on the first two boundary
components in the third case (before applying the shuffles). 
Let us prove equation~\eqref{eq:slab2} (the others can be proved analogously). 
Consider
$\Gamma\in 
\RR_{\ell_1,\ell_2,g_1,g_2,1dc}^{om}$
with an arbitrary
labelling. Then there exists a unique element in the group
$Sh_{1,\ell_1-1,\ell_2-1}\times \Z_{s_1}$ mapping the given labelling of 
$\Gamma$ to a special labelling in the sense of~\S\ref{ss:basiccomb}.

\subsection{Duality}\label{ss:graph-duality}

Here we discuss a duality operation on trivalent graphs that will play
a crucial role in the proof of the Maurer-Cartan equation in~\S\ref{sec:proofMC}.

Let $(\Gamma,l)$ be an o-marked labelled trivalent graph with
$l=(u,v)$ the oriented marked edge. Let $(z,w,u)$ and $(v,x,y)$ be the
two vertices connected by $l$ as shown in Figure~\ref{fig:duality}. We
define the o-marked labelled graph $(I(\Gamma),I(l))$ using the same
set of flags, but assembling them into vertices and edges slightly
differently. Namely, we let $I(l):=(u,v)$ be the oriented marked edge
of $I(\Gamma)$ and $(y,z,u)$ and $(v,w,x)$ be its adjacent vertices,
see Figure~\ref{fig:duality}. The other vertices and edges stay the same. 
\begin{figure}
\begin{center}
\includegraphics[width=\textwidth]{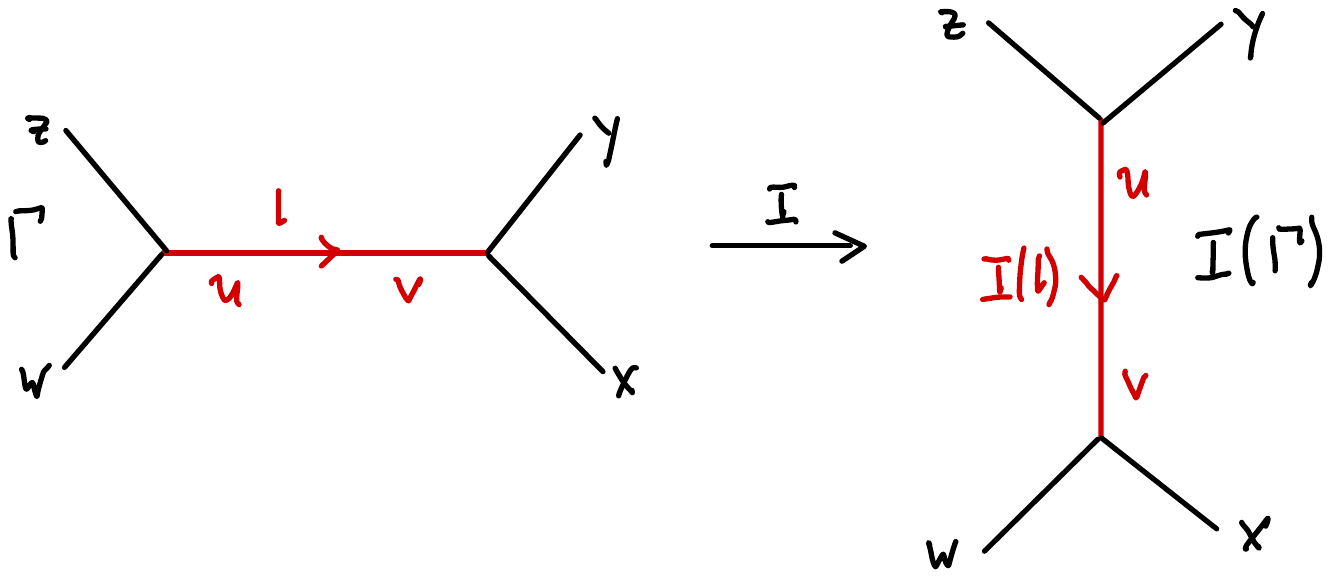}
\vspace{-3.5cm}
\caption{The duality operation $I$}
\label{fig:duality} 
\end{center}
\end{figure}
Geometrically, the operation $I$ is cutting out a model subtree and
pasting back the dual model subtree. In particular, the type of the graph 
remains the same. Since the two graphs have the same boundary
components and leaves, we can give $I(\Gamma)$ the same labelling as $\Gamma$ has.  

Next, we analyse the difference between the reordering maps 
for $\Gamma$ and for $I(\Gamma)$. 
Suppose that $\Gamma$ is equipped with an extension of the labelling such that the triples
of flags $(z,w,u)$ and $(v,x,y)$ are ordered as written in the vertex
order $O_v$, and the edge order $O_e$ contains the oriented edge $(u,v)$.
This induces an extension of the labelling for $I(\Gamma)$ as follows. 
As the sets of edges of $\Gamma$ and $I(\Gamma)$ coincide, we
choose the edge orders of $\Gamma$ and $I(\Gamma)$ to be exactly the same, $O_e^I=O_e$.
To produce a bijective correspondence between vertices we stipulate
that $(z,w,u)$ corresponds to  
$(y,z,u)$ and $(v,x,y)$ to $(v,w,x)$. We order the triples of flags
$(y,z,u)$ and $(v,w,x)$ as written, thus obtaining a vertex order $O_v^I$
for $I(\Gamma)$. Note that the composition $\sigma_4:=O_v^I\circ O_v^{-1}$ is
the cyclic permutation 
$$
   (z,w,x,y)\mapsto (y,z,w,x)
$$
of $\{z,w,x,y\}$, leaving the other flags fixed. The extended
labellings give us reordering maps $\bar R_\Gamma$, $\bar
R_{I(\Gamma)}$ and we compute
\begin{equation}\label{eq:dual-r}
   \bar R_{I(\Gamma)}^{-1}\circ \bar R_\Gamma=(O_e^I)^{-1}\circ O_v^I\circ
   O_v^{-1}\circ O_e = O_e^{-1}\circ O_v^I\circ O_v^{-1}\circ O_e=
   O_e^{-1}\circ\sigma_4\circ O_e. 
\end{equation}
The explicit description of 
how the vertices and edges of $\Gamma$ are related to those of
$I(\Gamma)$ implies (see~\cite{Volkov-thesis}) that
\begin{equation}\label{eq:eta3-duality}
  \eta_3(\Gamma)=\eta_3(I(\Gamma)).
\end{equation}
Let $(\Gamma,-l)$ denote the o-marked graph $(\Gamma,l)$ with the
reversed orientation of the marked edge. It is straightforward to
check that $I(\Gamma,-l)=(I(\Gamma),-I(l))$, so ``minus commutes with $I$''.
Therefore, the operation $I$ descends to an operation on marked
labelled graphs that we denote 
\begin{equation}\label{eq:dualtrival}
    \ol I:\RR_{\ell,g}^{m}\to \RR_{\ell,g}^{m}.
\end{equation}
Note that the operation $I$ on o-marked labelled graphs is an
operation of order $4$, whereas the operation $\ol I$ on marked labelled
graphs is an involution.


\begin{lemma}\label{lem:I-free}
The involution $\ol I$ has no
fixed points.
\end{lemma}

\begin{proof}
Suppose the marked labelled graph $\wh\Gamma$ is a fixed point of $\ol I$.
This means that there exists an isomorphism of marked labelled graphs
$\phi:\Gamma\to \ol I(\Gamma)$. Let $\ol\Gamma$ be the labelled graph with
one $4$-valent vertex $v$ obtained from 
$\wh\Gamma$ by collapsing the marked edge. It
follows that $\phi$ induces an automorphism $\ol\phi:\ol\Gamma\to\ol\Gamma$  
which fixes $v$ and rotates the flags at $v$ by $\pm 90$ degrees. In
particular, $\ol\phi$ is not the identity. But this contradicts Lemma~\ref{lem:no-auto1}
applied to $\ol\Gamma$. 
%
\end{proof}

\section{Algebraic preliminaries}\label{sec:alg}

In this section we collect some basic notions and facts about cochain
complexes and differential graded algebras with
pairings. See~\cite{Cieliebak-Fukaya-Latschev,Cieliebak-Hajek-Volkov}
for more background.

\subsection{Graded vector spaces}\label{ss:gradedvect}

Let $A=\bigoplus_{i\in \Z}A^i$ be a $\Z$-graded  
$\R$-vector space. For $m\in\Z$ let $A[m]$ be the degree shift of $A$ by $m$, i.e. 
$A[m]^i:=A^{i+m}$. Most often we will need the degree shift by $1$, that is $A[1]$.
For $x\in A$ of homogeneous degree, i.e. $x\in A^k$ for some $k\in\Z$,
the degree of $x$ as an element of $A$ will be denoted by $\deg x$ and
the degree of $x$ as an element of $A[1]$ will be denoted by $|x|$, so that
$$
  |x|=\deg x-1.
$$
We define the {\em graded dual} of $A$ by
$$
  A^*:=\bigoplus_{i\in \Z}\Hom(A^i,\R),
$$
and grade it by giving $\phi\in\Hom(A^i,\R)$ degree $i$.\footnote{ 
Another frequently used convention gives $\phi\in\Hom(A^i,\R)$ degree $-i$.} 

{\bf Three actions on $A^{\otimes k}$. } 
To define an operation on a tensor product, we will usually define it
on decomposable elements of homogeneous degree and extend it by linearity.
When we say ``decomposable'' we mean ``decomposable of homogeneous
degree'' whenever homogeneity of degree is needed to keep track of signs. 
Let us denote the {\em naive action} of an element $\eta$ of the symmetric
group $S_k$ on $A^{\otimes k}$ permuting factors without signs by
\begin{equation}\label{eq:actnaiv}
   \eta(x_1\otimes\cdots \otimes x_k) =
   x_{\eta(1)}\otimes\cdots\otimes x_{\eta(k)}.
\end{equation}
Now there are two natural sign conventions to incorporate the grading.
We define 
\begin{equation}\label{eq:actanalg}
   \eta_{an}(x):=(-1)^{\eta+\eta_{an}(x)}\eta(x),\qquad
   \eta_{alg}(x):=(-1)^{\eta_{alg}(x)}\eta(x),
\end{equation}
which we will refer to as the {\em analytic and algebraic action}. 
Here $(-1)^\eta$ is the sign of the permutation $\eta$, 
$\eta_{an}(x)$ is the sign exponent for permuting the elements
$x_j$ with their degrees in $A$, and 
$\eta_{alg}(x)$ is the
sign exponent for permuting the elements with their shifted degrees
$|x_j|=\deg x_j-1$.
We introduce the operation
$$
  P:A^{\otimes k}\to A[1]^{\otimes k}, 
$$
that acts on decomposables 
$x=x_1\otimes\dots\otimes x_k$ by
\begin{equation}\label{eq:defP}
  P(x) := (-1)^{P(x)}x
\end{equation}
with the sign exponent
$$
  P(x) := \sum_{j=1}^k(k-j)\deg x_j.
$$
Here the sign exponent can be described as follows: writing
$\theta_1x_1\theta_2x_2\cdots\theta_kx_k$ 
with formal variables $\theta_i$ of degree $1$, it is the sign
exponent for pulling all the formal variables $\theta_i$ to the left
of all the elements $x_j$.  
The two actions are related by the commuting diagram
\begin{equation}\label{eq:alg-ana-action}
\begin{aligned}
\xymatrix{
  A^{\otimes k} \ar[d]_{P} \ar[r]^{\eta_{an}} &A^{\otimes k} \ar[d]^{P} \\
  A[1]^{\otimes k} \ar[r]^{\eta_{alg}} &A[1]^{\otimes k}.
}
\end{aligned}
\end{equation}
This terminology allows us to introduce the following maps on tensor powers of $A$. Let $t\in S_k$
denote the cyclic permutation $(1,2,\dots,k)\mapsto (k,1,\dots,k-1)$
and $t_{an}$, $t_{alg}$ its associated analytic and
algebraic actions. Explicitly,
$$
  t_{alg}(x_1\otimes\cdots\otimes x_k) :=
  (-1)^{|x_k|(|x_1|+\cdots+|x_{k-1}|)} x_k\otimes x_1\otimes\cdots\otimes x_{k-1}\,.
$$
We set
$$
  N_{an}:=1+t_{an}+\dots+t_{an}^{k-1},
  \qquad N_{alg}:=1+t_{alg}+\dots+t_{alg}^{k-1}.
$$

{\bf The dual cyclic bar complex. }
We define the {\em bar complex} of $A$ by 
$$
  BA := \bigoplus_{k=1}^{\infty}A[1]^{\otimes k},
$$
Its graded dual is given by
$$
  B^*A = \prod_{k=1}^{\infty}(A[1]^{\otimes k})^*.
$$
Note that the direct sum becomes a direct product.
Similarly, we define the {\em cyclic bar complex} 
$$
  B^{\text{\rm cyc}}A := \bigoplus_{k=1}^{\infty}A[1]^{\otimes k}/\im(1-t_{alg})
$$
and its graded dual, the {\em dual cyclic bar complex} 
$$
  B^{\text{\rm cyc}*}A = \prod_{k=1}^{\infty}\Bigl(A[1]^{\otimes k}/\im(1-t_{alg})\Bigr)^*.
$$
Note that $t_{alg}$ generates the algebraic action of $\Z_k$ on 
$A[1]^{\otimes k}$, and dually on $(A[1]^{\otimes k})^*$, 
and we can identify 
$(A[1]^{\otimes k}/\im(1-t_{alg}))^*$ with the subcomplex of
functionals on $A[1]^{\otimes k}$ fixed under the action of $\Z_k$.

\begin{remark}[Filtrations]\label{rem:filtrations}
Consider a direct sum $B=\bigoplus_{k=1}^\infty B_k$ of $\Z$-graded
vector spaces $B_k$ (the index $k$ is not the grading!). Its has an
increasing filtration by the subspaces $\FF_\lambda =
\bigoplus_{k\leq\lambda}B_k$, which induces a decreasing filtration on
its graded dual $B^*=\prod_{k=1}^\infty B_k^*$ by the subspaces 
$$
\FF^\lambda = \{\phi\in B^*\mid \phi|_{B_k}=0\text{ for all }k\leq\lambda\}.
$$
Then $B^*$ is complete with respect to this filtration in the sense
of~\cite{Cieliebak-Fukaya-Latschev}; in fact, it is the completion of 
$\bigoplus_{k=1}^\infty B_k^*$.
(This is analogous to the well-knkown fact that the topological dual of
a normed vector space is complete.)
Applied to $B_k:=A[1]^{\otimes k}/\im(1-t_{alg})$, this shows that the
dual cyclic bar complex $B^{\text{\rm cyc}*}A$ is complete with
respect to its canonical filtration. 
\end{remark}

\begin{remark}
The symmetric group $S_k$ acts on the $k$-fold tensor product 
$A^{\otimes k}$ by~\eqref{eq:actnaiv} and~\eqref{eq:actanalg}. Since this action 
is based on precomposition, it is a right action (even though we
write it on the left). We extend these actions to the dual 
$(A^{\otimes k})^*$ by taking conjugate maps.
Since conjugation is a contravariant functor,
the resulting action on $(A^{\otimes k})^*$ is a left action.
\end{remark}

\subsection{Cochain complexes with pairing}\label{ss:coch}

A {\em pairing of degree $n\in \Z$} on a cochain complex $(A,d)$ is a
bilinear form $(\cdot,\cdot)\colon A\times A\to \R$, which for all
homogeneous $x, y\in A$ satisfies the degree condition 
\begin{equation*}
	(x,y)\neq 0\quad\Longrightarrow\quad\deg x+\deg y=n,
\end{equation*}
graded symmetry
\begin{align}\label{Eq:Symmetry}
   ( x, y ) = (-1)^{\deg x \deg y} 
   ( y, x ),
\end{align}
and compatibility with the differential
\begin{equation}\label{eq:DiffCyclic}
   ( d x, y ) = (-1)^{1+\deg x} 
   ( x, d y ). 
\end{equation}
We write $x\perp y$ if $( x, y ) = 0$ and say that $x,y$ are \emph{orthogonal}.
The subcomplex of elements of $A$ orthogonal to a given subcomplex
$B\subset A$ will be denoted by  
\[
	B^{\perp} := \{ x\in A\mid x\perp B\}.
\]
We call a pairing $(\cdot,\cdot)\colon A\times
A\to\R$ \emph{nondegenerate} if the induced map
\begin{equation*}
	A\longrightarrow Hom(A,\R),\qquad
	x\longmapsto ( x,\cdot )
\end{equation*}
is injective, and {\em perfect} if it is an isomorphism. Observe that a nonnegatively graded cochain complex with a perfect pairing is finite dimensional. 
Following the terminology in~\cite{Cieliebak-Fukaya-Latschev},
a cochain complex with a perfect pairing is called {\em cyclic}. If a perfect pairing
restricts to a subcomplex as a perfect pairing, then the subcomplex is called 
{\em cyclic}.

{\bf Projections and propagators. }
Next we recall some notions and facts from~\cite{Cieliebak-Hajek-Volkov}.
Let $(A, d,(\cdot,\cdot))$ be a cochain complex with pairing.
A {\em projection} on $A$ is a degree $0$ chain map $\pi\colon A\to A$
that satisfies $\pi\circ \pi = \pi$.  
A \emph{homotopy operator} with respect to the $\pi$ is a degree
  $-1$ map $P\colon A\to A$ satisfying
\begin{equation}\label{eq:P2}
    d\circ P+P\circ d=\pi-\Id. 
\end{equation}
Its existence implies that $\pi$ is a {\em quasi-isomorphism}, i.e.,
the induced map on cohomology $H(\pi)\colon H( A)\to H( A)$ is an
isomorphism.  
Note that the projection $\pi$ is determined by $P$ via equation~\eqref{eq:P2}.
We say that a homotopy operator $P\colon  A\to A$ is a
\emph{propagator} if it satisfies in addition the symmetry property 
\begin{equation}\label{eq:symmpropprelim}
	 (Px,y) = (-1)^{\deg x}(x,Py). 
\end{equation}
The associated projection 
$\pi\colon A\to A$ is then also \emph{symmetric},
\begin{equation*}
	 (\pi x,y) = ( x,\pi y).
\end{equation*}
A propagator $P$ with respect to $\pi$ is called {\em semi-special} if
\begin{equation}\label{eq:P4}
  P\circ \pi = \pi\circ P = 0,
\end{equation}
and {\em special} if in addition
\begin{equation}\label{eq:P5}
  P\circ P = 0.
\end{equation}
A special propagator with respect to a given projection is not unique in general.

\begin{lemma}[{\cite[Remark 2]{Cattaneo-Mnev}}]\label{lem:propagator}
Any propagator $P$ can be modified to a semi-special propagator $P_2$
and then to a special propagator $P_3$ with
respect to the same projection $\pi$ by setting 
\[
   P_2 := (\pi-\Id)\circ P\circ (\pi-\Id),\qquad P_3 := -P_2\circ  d\circ P_2.
\]
\end{lemma}

Given a subcomplex $B\subset  A$, we say that a projection $\pi\colon
A\to A$ is onto $B$ if $\im\pi=B$, and we identify
$\pi$ with the induced surjection $\pi\colon A\onto B$ in this case. 
A homotopy operator~$P\colon A\to A$ with respect to a projection
$\pi\colon A\onto B$ exists if and only if~$\pi$ is a
quasi-isomorphism.  
In the case with pairing we have 

\begin{lemma}[{\cite[Lemma~2.2]{Cieliebak-Hajek-Volkov}}]\label{lem:retraction}
Let $( A, d,(\cdot,\cdot))$ be a cyclic cochain complex and
$B\subset  A$ a subcomplex. 
A propagator $P$ with respect to a projection $\pi\colon A\onto B$
exists if and only if~$\pi$ is symmetric and a quasi-isomorphism. 
\end{lemma}

If $(A,d,(\cdot,\cdot))$ is a cyclic cochain complex, then any
quasi-isomorphic cyclic subcomplex $B\subset A$ admits a unique
symmetric projection $\pi_B\colon A \onto B$ by sending the orthogonal
complement $B^\perp$ to $0$.
Lemmas~\ref{lem:propagator} and~\ref{lem:retraction} now imply

\begin{cor}\label{cor:retraction}
Let $( A, d,(\cdot,\cdot))$ be a cyclic cochain complex and
$B\subset  A$ a quasi-isomorphic cyclic subcomplex. 
Then there exists a special propagator $P\colon A\to A$ whose
associated projection $\pi$ satisfies $\im\pi=B$.
\qed
\end{cor}

\subsection{Harmonic subspaces}\label{ss:hodgesetup}

Let $(A,d,(\cdot,\cdot))$ be a cochain complex with pairing. In view
of~\eqref{eq:DiffCyclic}, the pairing descends to its homology $H(A)$.  
In this subsection we consider the case that $H(A)$ is finite dimensional
and the induced pairing on $H(A)$ is nondegenerate, so $H(A)$ becomes
a cyclic cochain complex with trivial differential.
Our main example for this will be the de Rham complex 
$(\Om^*(M), d,(\cdot,\cdot))$ of a closed oriented manifold with the 
intersection pairing, see~\S\ref{sec:prop}.

By a {\em harmonic subspace} $\HH$ we mean any complement of $\im d$ in
$\ker d$, so that
$$
   \ker d=\HH\oplus\im d.
$$

\begin{lemma}\label{lem:Horthog}
Let $(A,d,(\cdot,\cdot))$ be a cochain complex with pairing such that
$H(A)$ is finite dimensional and the induced pairing on $H(A)$ is nondegenerate.
Let $\HH\subset\ker d$ be a harmonic subspace. Then we get a direct sum decomposition
\begin{equation}\label{eq:Horthog}
  A=\HH\oplus\HH^\perp\quad\text{with}\quad \HH^\perp\cap\ker d=\im d.
\end{equation}
The projection $\pi_\HH:A\to A$ onto $\HH$ along $\HH^\perp$ is
symmetric, and it is a quasi-isomorphism as a map $A\to\HH$. 
\end{lemma}

\begin{proof}
Nondegeneracy of the pairing on $H(A)$ and the fact that every
cohomology class has a unique harmonic representative implies that 
the restriction of the pairing to $\HH$ is nondegenerate.
Pick a basis $\{e_a\}$ of $\HH$ and define its dual basis $\{e^b\}$ of
$\HH$ by $(e_a,e^b)=\delta_a^b$. Then
$$
\pi_\HH:A\to A,\qquad \pi_\HH(x) :=
\sum_a(e_a,x)e^a
$$
is a projection with image $\HH$ and kernel $\HH^\perp$, which
shows the first equation in~\eqref{eq:Horthog}. 
The inclusion $\HH^\perp\cap\ker d\supset\im d$ is obvious. 
For the converse inclusion, consider $x\in\HH^\perp\cap\ker d$.
In view of~\eqref{eq:DiffCyclic} this implies $x\perp \HH\oplus\im
d=\ker d$, and therefore $0=(x,y)=([x],[y])$ for all $y\in\ker d$,
where $[x],[y]$ denote the cohomology classes. By nondegeneracy of the 
pairing on homology this implies $[x]=0$, hence $x\in\im d$. 
This proves the second equation in~\eqref{eq:Horthog}, and the
asserted properties of $\pi_\HH$ follow directly from this. 
\end{proof}

Assume now that $P$ is a homotopy operator with respect to $\pi_\HH$. 
Then equation~\eqref{eq:P2} implies
$$
  A=\ker d+\im P.
$$
Assume in addition that $P$ is special and consider any $y\in \im
P\cap\ker d$. Then there exists $x\in A$ with $Px=y$, and therefore  
$d(Px)=0$. Equation~\eqref{eq:P2} now implies $P(dx)=\pi x-x$. We
apply $P$ to both sides and use properties~\eqref{eq:P5}
and~\eqref{eq:P4} of a special homotopy operator to conclude
$y=Px=0$. Therefore, the above sum is direct and we get 
\begin{equation}\label{eq:towardshodge}
  A=\ker d\oplus\im P.
\end{equation}
Finally, we relate this decomposition to the decomposition~\eqref{eq:Horthog}.
Equation~\ref{eq:P4} implies that $\pi_\HH|_{\im P}=0$. Since $\pi_\HH$
is a projection along $\HH^\perp$, we get the inclusion $\im P\subset \HH^\perp$. 
In other words, 
\begin{equation}\label{eq:keyortho}
  \HH\perp\im P.
\end{equation}
This last relation will play a crucial role in~\S\ref{sec:gauge}.
Note that in the preceding discussion we have not used the symmetry
  property~\eqref{eq:symmpropprelim}.

\subsection{Differential graded algebras with pairings}\label{ss:dga1} 

A {\em differential graded algebra (DGA)} $(A,d,\wedge)$ is a 
nonnegatively graded cochain complex $(A,d)$ equipped with an
associative product $\wedge\colon A\times A \to A$ of degree  
$0$ satisfying the Leibniz rule
\[
	d(x\wedge y) = d x \wedge y + 
	(-1)^{\deg x} x \wedge d y.
\]
A {\em pairing} on a DGA $(A,d,\wedge)$ is a pairing 
$(\cdot,\cdot)\colon A \times A\to\R$ on the cochain complex $(A,d)$ 
satisfying in addition
\begin{equation}\label{Eq:ProdCyclic}
	( x \wedge y,z) = (x, y\wedge z).
\end{equation}
Using \eqref{Eq:Symmetry}, condition \eqref{Eq:ProdCyclic} is
equivalent to the \emph{cyclicity condition}
\begin{equation}\label{Eq:ProdCyclicOld}
	(x\wedge y, z)  = (-1)^{\deg z(\deg x+\deg y)} (z\wedge x, y).
\end{equation}
The last condition is in turn equivalent to the {\em triple intersection product}
\begin{equation}\label{eq:deftriple}
   \m^\can(x_0,x_1,x_2):= 
   (-1)^{\deg x_1+n}(x_0\wedge x_1,x_2)
\end{equation}
defining an element in $B_3^{\text{\rm cyc}*}A$. 
A DGA with a perfect pairing is called {\em cyclic}. 

Our main example of a DGA with pairing is the de Rham algebra
$(\Om^*(M),d,\wedge)$ of a closed oriented manifold $M$ with the wedge
product $\wedge$ and the {\em intersection pairing}
\begin{equation}\label{eq:defineintpair}
  (\alpha,\beta):=\int_M\alpha\wedge\beta.
\end{equation}
Note that the intersection pairing 
$(\cdot,\cdot)$ is nondegenerate but not perfect.

\section{$\IBL_\infty$-algebras}\label{sec:IBLinfty}



In this section we recall from~\cite{Cieliebak-Fukaya-Latschev} the
basic notions and properties of $\IBL_\infty$-algebras. 
See also~\cite{Hoffbeck-Leray-Vallette} 
for a presentation in the language of properads. 
As before, we restrict to the ground field $\R$. 

\subsection{Basic definitions and properties}\label{ss:IBL-basic}

We begin with the basic definitions. 

{\bf $\IBL_\infty$-algebras. }
Let $C= \bigoplus_{i\in\Z} C^i$ be a $\Z$-graded $\R$-vector space.
Recall that the degree shifted version $C[1]$ of $C$ is defined by
by $C[1]^i:=C^{i+1}$, so the degrees $\deg c$ in $C$ and
$|c|$ in $C[1]$ are related by
$$
   |c| = \deg c -1. 
$$
We introduce the $k$-fold symmetric product 
$$
    E_k C := \left(C[1] \,\otimes_\R \cdots
    \,\otimes_\R C[1]\right)/S_k 
$$
as the quotient of the $k$-fold tensor product under the standard
action of the symmetric group $S_k$ permuting the factors with signs,
and the {\em reduced symmetric algebra}
$$
   EC := \bigoplus_{k\geq 1}E_kC. 
$$
We will write the equivalence class of $c_1\otimes\dots\otimes c_k$ in
$E_kC$ as $c_1\cdots c_k$. 

Let $d\in\Z$. An {\em $\IBL_\infty$-structure of degree $d$} on $C$ is
a collection of linear maps 
$$
    \frak p_{k,\ell,g} : E_kC \to E_{\ell}C,\qquad k\ge 1,\ \ell \ge 1,\
    g\geq 0
$$ 
of degree 
$$
    |\frak p_{k,\ell,g}| = -2d(k+g-1)-1,
$$
that satisfy the quadratic relation 
\begin{equation*}
    \wh\fp\circ \wh\fp=0.
\end{equation*}
To make sense of this relation, we extend the $\fp_{k,\ell,g}$ to operations 

\begin{equation}\label{eq:phat}
  \wh\fp_{k,\ell,g}:EC\longrightarrow EC
\end{equation}
and then form generating series $\wh\fp$.
See~\cite{Cieliebak-Fukaya-Latschev} for the precise meaning of this
relation, which will not be relevant in this paper.
Geometrically, one should think of $\fp_{k,\ell,g}$ as an operation
associated to a compact connected oriented surface of genus $g$ with
$k$ incoming and $\ell$ outgoing boundary components. The quadratic
relation then amounts to the vanishing of the sum over all partial
compositions of such operations corresponding to splittings of the
surface into two surfaces.  

The tuple $(C,\{\frak p_{k,\ell,g}\}_{k,\ell \geq 1,g\geq 0})$ is
called an {\em $\IBL_\infty$-algebra of degree $d$}. 
An $\IBL_\infty$-algebra is called a {\rm dIBL}-algebra if its only
nontrivial operations are $\fp_{1,1,0}$, $\fp_{2,1,0}$ and $\fp_{1,2,0}$,
and an {\em $\IBL$-algebra} if its only nontrivial operations are
$\fp_{1,1,0}$ and $\fp_{2,1,0}$. After adjusting the operations by
suitable signs, an IBL-algebra is the  
same as an involutive Lie bialgebra, which explains its name.

The quadratic relations for an $\IBL_\infty$-algebra imply that
$\fp_{1,1,0}:C\to C$ squares to zero and the operations $\fp_{2,1,0}$
and $\fp_{1,2,0}$ descend to homology $H(C,\fp_{1,1,0})$, where they
induce the structure of an IBL-algebra.

{\bf $\IBL_\infty$-morphisms. }
Let $(C^+,\{\fp^+_{k,\ell,g}\})$ and $(C^-,\{\fp^-_{k,\ell,g}\})$ be 
two $\IBL_\infty$-algebras of the same degree $d$. 
An {\em $\IBL_\infty$-morphism} from $C^+$ to $C^-$ is a collection of
linear maps  
$$
    \ff_{k,\ell,g} : E_k C^+ \to E_{\ell} C^-,\qquad k,\ell \ge 1,\
    g\geq 0
$$ 
of degree 
$$
    |\ff_{k,\ell,g}| = -2d(k+g-1)
$$ 
whose generating series $\ff$ satisfies the relation
\begin{equation}\label{eq:mor2}
    e^\ff\wh\fp^+ - \wh\fp^- e^\ff=0.
\end{equation}
Again, the precise meaning of this relation will not be relevant in
this paper. 

An {\em $\IBL_\infty$-quasi-isomorphism} is an $\IBL_\infty$-morphism
which induces an isomorphism on homology.
In~\cite{Cieliebak-Fukaya-Latschev} the notion of
a {\em homotopy of $\IBL_\infty$-morphisms} is defined and the
following facts are proved:

({\em Homotopy inverse}) Every $\IBL_\infty$-quasi-isomorphism is an
  $\IBL_\infty$-homotopy equivalence.

({\em Homotopy transfer}) Every $\IBL_\infty$-algebra $(C,\{\fp_{k,\ell,g}\}$
induces an $\IBL_\infty$-structure $\{\fq_{k,\ell,g}\}$ on its
homology $H=H(C,\fp_{1,1,0})$ together with an $\IBL_\infty$-homotopy
equivalence $(C,\{\fp_{k,\ell,g}\})\to (H,\{\fq_{k,\ell,g}\})$.

\subsection{Maurer-Cartan elements}

Next we discuss Maurer-Cartan elements and the resulting
twisted $\IBL_\infty$-structures.
Let $(C,\fp=\{\fp_{k,\ell,g}\})$ be an $\IBL_\infty$-algebra of degree $d$. 
A {\em Maurer-Cartan element} $\fm$ on $(C,\fp)$ is a collection of elements 
\begin{equation}\label{eq:MCfooting}
    \fm_{\ell,g}\in \wh E_\ell C,\qquad \ell\geq 1,g\geq 0
\end{equation}
of degrees 
\begin{equation}\label{eq:MCdegrees}
   |\fm_{\ell,g}| = -2d(g-1)
\end{equation} 
whose generating series (again denoted $\m$ by abuse of language)
satisfies the {\em Maurer-Cartan equation} 
\begin{equation}\label{eq:MC}
    \wh\fp(e^{\fm}) = 0. 
\end{equation}
The last equation is equivalent to saying that
$\fm=\{\fm_{\ell,g}\}_{\ell\geq 1,g\geq 0}$ defines a morphism from the
trivial $\IBL_\infty$-algebra $\R$ to $C$, since
equation~\eqref{eq:mor2} reduces to equation~\eqref{eq:MC} for $C^+=\R$
and $C^-=C$. 
The space $\wh E_\ell C$ in~\eqref{eq:MCfooting} is a
suitable completion of $E_\ell C$ (see~\cite{Cieliebak-Fukaya-Latschev}),
which we will later spell out in the case we need it (see
equation~\eqref{eq:whEk}).  

\begin{prop}[{\cite[\S9]{Cieliebak-Fukaya-Latschev}}]\label{prop:MC}
(a) A Maurer-Cartan element $\fm$ on an $\IBL_\infty$-algebra
$(C,\fp=\{\fp_{k,\ell,g}\})$ gives rise to a {\em twisted
$\IBL_\infty$-structure} $\fp^\fm=\{\fp^\fm_{k,\ell,g}\}$ on $C$. 

(b) Given an $\IBL_\infty$-morphism $\ff=\{\ff_{k,\ell,g}\}:(C,\{\fp_{k,\ell,g}\})\to
(D,\{\fq_{k,\ell,g}\})$ and a Maurer-Cartan element $\fm$ on $C$, we
obtain a {\em pushforward Maurer-Cartan element} $\ff_*\fm$ on
$(D,\{\fq_{k,\ell,g}\})$ together with an $\IBL_\infty$-morphism 
$$
  \ff^\fm = \{\ff^\fm_{k,\ell,g}\}: (C,\{\fp^\fm_{k,\ell,g}\}) \to
  (D,\{\fq^{\ff_*\fm}_{k,\ell,g}\}).
$$
(c) If $\ff$ in (b) is an $\IBL_\infty$-homotopy equivalence, then so is $\ff^\fm$.  
\end{prop}



On a $\dIBL$-algebra, the Maurer-Cartan equation~\eqref{eq:MC}
simplifies to 
\begin{equation}\label{eq:MCdIBL}
  \wh\fp_{1,1,0}\fm+\wh\fp_{2,1,0}\fm+\frac{1}{2}\wh\fp_{2,1,0}(\fm\otimes\fm)|_{\conn}+\wh\fp_{1,2,0}\fm=0.
\end{equation}
Here $\wh\fp_{k,\ell,g}$ denotes a suitable extension of the operation
$\fp_{k,\ell,g}$, and the subscript ``conn'' denotes the connected
part of the corresponding operation. 
We will later spell this out in the case we need it, see~\S\ref{sec:proofMC}. 

If in addition the only nontrivial term of the Maurer-Cartan element
$\fm$ is $\fm_{1,0}$, then equation~\eqref{eq:MCdIBL} specializes
further to the two equations
\begin{equation}\label{eq:discMC}
  \fp_{1,1,0}\fm_{1,0}+\frac{1}{2}\fp_{2,1,0}(\fm_{1,0}\otimes\fm_{1,0})=0
\quad \text{and} \quad
  \fp_{1,2,0}\fm_{1,0}=0.
\end{equation}
In this case to the corresponding twisted $\dIBL$ structure is given by
\begin{equation}\label{eq:dIBLtwist}
  (\fp_{1,1,0}^\fm=\fp_{1,1,0}+\fp_{2,1,0}(\fm_{1,0},\cdot),\,\fp_{2,1,0},\,\fp_{1,2,0}).
\end{equation}

\subsection{Relation to differential Weyl algebras}\label{ss:Weyl}

In this section we will explain the relation between the $\IBL_\infty$
formalism and the formalism of differential Weyl algebras from
symplectic field theory (SFT)~\cite{EGH}. All statements are 
proved in~\cite{Cieliebak-Fukaya-Latschev}.

{\bf Objects.} 
Fix a ground field $\R$ and an index set $\PP$ (which
corresponds to the set of periodic orbits in SFT). 
Consider the Weyl algebra $\W$ of power series in variables 
$\{p_\gamma\}_ {\gamma\in \PP}$ and $\hbar$ with coefficients polynomial
over $\R$ in variables $\{q_\gamma\}_{\gamma\in \PP}$. 
Each variable comes with an integer grading, and we assume that
$$
|p_\gamma| + |q_\gamma| = |\hbar| = 2d
$$
for some integer d and all $\gamma\in \PP$. 
The algebra $\W$ comes equipped with an associative product $\star$ in which all variables commute
according to their grading except for $p_\gamma$ and $q_\gamma$
corresponding to the same index $\gamma$, for which we have
$$
p_\gamma \star q_\gamma - (-1)^{|p_\gamma||q_\gamma|}q_\gamma \star
p_\gamma = \kappa_\gamma \hbar
$$
for some integers $\kappa_\gamma \geq 1$ (which correspond to the
multiplicities of periodic orbits in SFT).
A homogeneous element $\HHH\in\frac{1}{\hbar}\W$ of degree $-1$
satisfying 
\begin{equation}\label{eq:weyl1}
\HHH\star\HHH = 0
\end{equation}
is called a {\em Hamiltonian}, and the pair $(\W,\HHH)$ is called a {\em
   differential Weyl algebra of degree $d$}. Indeed, the commutator
with $\HHH$ is then a derivation of $(\W,\star)$ of square 0.
We will impose two further restrictions on our Hamiltonians $\HHH$,
namely
\begin{equation}\label{eq:H-exact}
\HHH|_{p=0}=0 \quad \text{\rm and} \quad \HHH|_{q=0}=0.
\end{equation}
Then $\HHH$ can be expanded as
\begin{equation}\label{formula:H}
   \HHH = \sum_{k,\ell \geq 1, g \geq 0} H_{k,\ell,g} \hbar^{g-1},
\end{equation}
where $H_{k,\ell,g}$ is the part of the coefficient of $\hbar^{g-1}$
which has degree $k$ in the $p$'s and degree $\ell$ in the $q$'s.

Consider now the free $\R$-module $C$ generated by the elements
$q_\gamma$ for $\gamma\in\PP$, and graded by the
degrees $\deg(q_\gamma):=|q_\gamma|+1$. Then $EC=\bigoplus_{k\geq 1}E_kC$ is the non-unital commutative
algebra of polynomials in the variables $\{q_\gamma\}_{\gamma \in
  \PP}$ without constant terms. 
We can represent $\W$ as differential operators acting on the left on
$EC\{\hbar\}$ by the replacements
$$
   p_\gamma\longrightarrow\hbar\kappa_\gamma
\overrightarrow{\frac \p {\p q_{\gamma}}}.
$$
Then the Hamiltonian $\HHH$ determines operations
\begin{equation}\label{eq:H-p}
   \fp_{k,\ell,g} := \frac 1 {\hbar^k}
   \overrightarrow{H_{k,\ell,g}}: E_kC \to E_\ell C.
\end{equation}
The fact that the coefficients of $\HHH$ are polynomial in the
$q_\gamma$'s translates into
\begin{equation}\label{eq:p-finite}
\begin{aligned}
   &\text{Given $k\geq 1$, $g\geq 0$ and $a\in E_kC$, the term
   $\fp_{k,\ell,g}(a)$} \cr 
   &\text{is nonzero for only finitely many $\ell\geq 1$.}
\end{aligned}
\end{equation}
Conversely, $\HHH$ can be recovered from the operations $\fp_{k,\ell,g}$
by 
\begin{equation}\label{eq:p-H}
   H_{k,\ell,g} = \sum_{\gamma_1,\dots,\gamma_k\in\PP}\frac{1}
   {\kappa_{\gamma_1}\cdots \kappa_{\gamma_k}}
    \fp_{k,\ell,g}(q_{\gamma_1}\cdots q_{\gamma_k})p_{\gamma_1}\cdots
    p_{\gamma_k}. 
\end{equation}

\begin{prop}[\cite{Cieliebak-Fukaya-Latschev}]\label{prop:IBLWeyl}
Equations~\eqref{eq:H-p} and~\eqref{eq:p-H} define a one-to-one
correspondence between differential Weyl algebras
satisfying~\eqref{eq:H-exact} and $\IBL_\infty$-algebras
satisfying~\eqref{eq:p-finite} (both of degree $d$).    
\end{prop}

{\bf Morphisms.}
Suppose that $(\W^+,\HHH^+)$ and $(\W^-,\HHH^-)$ are
differential Weyl algebras of the same degree $d$ with indexing sets
$\PP^+$ and $\PP^-$. 
Let $\D$ denote the graded commutative associative
algebra of power series in the $p^+$ and $\hbar$ with coefficients
polynomial in the $q^-$. By definition, a {\em morphism} between the
differential Weyl algebras is an element $\F \in  \frac 1 {\hbar} \D$ 
satisfying 
\begin{equation}\label{eq:weyl2}
   e^{-\F} (\overrightarrow{\HHH^-}e^\F - e^\F \overleftarrow{\HHH^+}) = 0.
\end{equation}
Here $\HHH^+$ acts on $e^\F$ from the right by replacing each
$q^+_{\gamma}$ by $\hbar\kappa_\gamma \overleftarrow{\frac \p {\p
p^+_{\gamma}}}$, and the expression is to be viewed as an equality of
elements of $\frac 1 {\hbar} \D$. Again, we impose the additional
condition that 
\begin{equation}\label{eq:F-exact}
   \F|_{p^+=0}=0 \quad \text{\rm and} \quad \F|_{q^-=0}=0.
\end{equation}

As with $\HHH$ above, we expand $\F$ as
\begin{equation}\label{formula:F}
    \F = \sum_{k,\ell,g} F_{k,\ell,g} \hbar^{g-1},
\end{equation}
and define operators $\overrightarrow{F_{k,\ell,g}}:E_kC^+ \to
E_\ell C^-$ by substituting $p_\gamma^+$ by $\hbar\kappa_\gamma
\overrightarrow{\frac \p {\p q_{\gamma}^+}}$. In this way, we get maps 
\begin{equation}\label{eq:F-f}
   \ff_{k,\ell,g}:=\frac 1
   {\hbar^k}\overrightarrow{F_{k,\ell,g}}:E_kC^+ \to E_\ell C^-. 
\end{equation}
satisfying the condition
\begin{equation}\label{eq:f-finite}
\begin{aligned}
   &\text{Given $k\geq 1$, $g\geq 0$ and $a\in E_kC^+$, the term
   $\ff_{k,\ell,g}(a)$} \cr 
   &\text{is nonzero for only finitely many $\ell\geq 1$.}
\end{aligned}
\end{equation}
Again, $\F$ can be recovered from the operations $\ff_{k,\ell,g}$
by 
\begin{equation}\label{eq:f-F}
   F_{k,\ell,g} = \sum_{\gamma_1,\dots,\gamma_k\in\PP^+}\frac{1}
   {\kappa_{\gamma_1}\cdots \kappa_{\gamma_k}}
    \ff_{k,\ell,g}(q_{\gamma_1}^+\cdots q_{\gamma_k}^+)p_{\gamma_1}^+\cdots
    p_{\gamma_k}^+. 
\end{equation}

\begin{prop}[\cite{Cieliebak-Fukaya-Latschev}]
Equations~\eqref{eq:F-f} and~\eqref{eq:f-F} define a one-to-one
correspondence between morphisms of differential Weyl algebras
satisfying~\eqref{eq:F-exact} and morphisms of $\IBL_\infty$-algebras
satisfying~\eqref{eq:f-finite}.    
\end{prop}

{\bf Homotopies.}
Let $C^\infty([0,1])$ be the algebra of smooth functions $[0,1]\to\R$
of a variable $t\in[0,1]$. 

\begin{definition}\label{def:homotWeyl}
Two Weyl algebra morphisms $\F_0,\F_1:(\W^+,\HHH^+) \to (\W^-,\HHH^-)$ are
called homotopic if there exist $\F(q^-,t,p^+),\K(q^-,t,p^+)\in \frac
1 {\hbar} \D\otimes C^\infty([0,1])$ such that
$$
\F(q^-,0,p^+) = \F_0(q^-,p^+), \qquad 
\F(q^-,1,p^+) = \F_1(q^-,p^+)
$$
and 
\begin{equation}\label{eq:BVhomot}
0 = \frac {\p}{\p t}e^{\F(t)}  - \overrightarrow{[\HHH^-,\K(t)]}
e^{\F(t)} - e^{\F(t)}\overleftarrow{[\K(t),\HHH^+]}.
\end{equation}
\end{definition}

In~\cite[Section~7]{Cieliebak-Fukaya-Latschev} this definition and the
following proposition and remark are stated with $\F,\K$ in $\frac 1 {\hbar}
\D[t]$ instead of $\frac 1 {\hbar} \D\otimes C^\infty([0,1])$, but everything 
goes through verbatim if we replace polynomials in $t$ by smooth functions 
in $t$.

\begin{prop}[\cite{Cieliebak-Fukaya-Latschev}]\label{prop:IBLWeylehomot}
Consider two differential Weyl algebras 
$(\W^+,\HHH^+)$ and $(\W^-,\HHH^-)$, 
and denote by $(C^+,\{\fp_{k,\ell,g}^{+}\})$ and
$(C^-,\{\fp_{k,\ell,g}^{-}\})$ 
the corresponding $\IBL_\infty$-algebras, respectively.
Let $\F_0,\F_1:(\W^+,\HHH^+) \to (\W^-,\HHH^-)$ be 
Weyl algebra morphisms and denote by
$\frak f^{(0)} = \{\frak f^{(0)}_{k,\ell,g}\}$ and $\frak f^{(1)} =
\{\frak f^{(1)}_{k,\ell,g}\}$ the corresponding morphisms
$(C^+,\{\fp_{k,\ell,g}^{+}\}) \to (C^-,\{\fp_{k,\ell,g}^{-}\})$ of
$\IBL_\infty$-algebras, respectively. 
Then $\F_0$ is homotopic to $\F_1$ if and only if 
$\frak f^{(0)}$ is homotopic to $\frak f^{(1)}$. 
\end{prop} 

\begin{remark}[\cite{Cieliebak-Fukaya-Latschev}]\label{rem:homotWeyl}
In the situation of Definition~\ref{def:homotWeyl},
equation~\eqref{eq:BVhomot} together with the fact that $\F(0)$ is a
morphism implies that $\F(t)$ is a morphism for all $t\in[0,1]$. 
\end{remark}

\subsection{Gauge equivalence}\label{ss:gaugegen} 

Let $(C,\fp=\{\fp_{k,\ell,g}\})$ be an $\IBL_\infty$-algebra. 
Two Maurer-Cartan (MC) elements $\fm_0$ and $\fm_1$ on $(C,\fp)$ are
called {\em gauge equivalent} if they are homotopic as morphisms
$\R\to C$, where $\R$ denotes the trivial $\IBL_\infty$-algebra.
The purpose of this subsection is to derive from the differential Weyl
algebra formalism of~\S\ref{ss:Weyl} a workable criterion for gauge
equivalence of two MC elements connected by a smooth path of MC elements. 

Let $\{q_\gamma\}_{\gamma\in \PP}$ be a homogeneous basis of $C$. 
Let $(\WW,\mathbb{H})$ be the differential Weyl algebra corresponding
to $(C,\fp)$ via Proposition~\ref{prop:IBLWeyl}, where we define the
$\star$-product in $\WW$ using multiplicities $k_\gamma=1$
for all $\gamma\in \PP$. 
We denote the trivial Weyl algebra with 
$\PP=\emptyset$ and $\mathbb{H}=0$ by $(\R,0)$. 
We specify Definition~\ref{def:homotWeyl} to 
$$
  (\WW^+,\mathbb{H}^+):=
  (\R,0),\quad (\WW^-,\mathbb{H}^-):=(\WW,\mathbb{H})
$$
and drop all superscripts ``$-$'' from the
notation. Equation~\eqref{eq:BVhomot} then simplifies to 
\begin{equation}\label{eq:MCWeyl}
\ddt e^{\F(t)}=\overrightarrow{[\mathbb{H},\mathbb{K}(t)]}e^{\F(t)}
\end{equation}
In view of Remark~\ref{rem:homotWeyl}, we may assume that each $\F(t)$
is a morphism from $(R,0)$ to $(\WW,\mathbb{H})$. 
Then we can rewrite the right hand side of \eqref{eq:MCWeyl} as
$$
\overrightarrow{[\mathbb{H},\mathbb{K}(t)]}e^{\F(t)}=
\overrightarrow{\mathbb{H}}(\mathbb{K}(t)e^{\F(t)})-\mathbb{K}(t)\overrightarrow{\mathbb{H}}
e^{F(t)}=\overrightarrow{\mathbb{H}}(\mathbb{K}(t)e^{\F(t)}).
$$
Using this and $\overrightarrow{\mathbb{H}}=\wh\fp$, taking the
``upper right arrow'' of \eqref{eq:MCWeyl} yields 
$$
\ddt e^{\overrightarrow{\F(t)}}=\wh \fp\circ (\overrightarrow{\mathbb{K}(t)}\odot 
e^{\overrightarrow{\F(t)}}).
$$
Since $\F(t)$ is a morphism from the trivial Weyl algebra, the corresponding 
$$
\fm_t:=\overrightarrow{F(t)}
$$
is a MC element for on $\IBL_\infty$ algebra $(C,\fp)$. Setting 
$$
\bbb_t:=\overrightarrow{\mathbb{K}(t)}
$$
gives us the equation
\begin{equation}\label{eq:gaugegen}
\ddt e^{\fm_t}=\wh\fp(\bbb_te^{\fm_t}).
\end{equation}
We can further simplify this equation using the fact 
that each $\fm_t$ is a MC element. Let 
$[\wh\fp(\bbb_te^{\fm_t})]_{conn}$ denote the connected 
part of $\wh\fp(\bbb_te^{\fm_t})$. Then the right
hand side of~\eqref{eq:gaugegen} transforms as follows:
$$
\wh\fp(\bbb_te^{\fm_t})
=[\wh\fp(\bbb_te^{\fm_t})]_{conn}e^{\fm_t}
+\bbb_t\wh\fp(e^{\fm_t})
=[\wh\fp(\bbb_te^{\fm_t})]_{conn}e^{\fm_t},
$$
where the last equality uses the MC property of $\fm_t$.
The left hand side of~\eqref{eq:gaugegen}
transforms as follows:
$$
\ddt e^{\fm_t}=\left(\ddt\fm_t\right)e^{\fm_t}.
$$
Therefore, equation~\eqref{eq:gaugegen} is equivalent to 
\begin{equation}\label{eq:gaugesimpl}
\ddt \fm_t=[\wh\fp(\bbb_te^{\fm_t})]_{conn}.
\end{equation}
In the IBL case, this specializes to the following statement.

\begin{lemma}\label{lem:gaugeIBLinfty}
Assume that $\fp=(\fp_{2,1,0},\,\fp_{1,2,0})$ is an $\IBL$ structure on $C$. Let
$\{\fm_t\}_{t\in[0,1]}$ be a smooth family of Maurer-Cartan elements for $\fp$. Then  
all the $\fm_t$ are gauge equivalent to each other if
there exists a smooth family $\{\bbb_t\}_{t\in[0,1]}$, with
$\bbb_t=\{\bbb_{(\ell,g)t}\}_{\ell\ge 1,g\ge 0}$ and
$\bbb_{(\ell,g)t}\in E_\ell C$, such that the following equation holds:
\begin{equation}\label{eq:gaugeIBL} 
\ddt \fm_t=\wh \fp_{2,1,0}\bbb_t+\wh \fp_{2,1,0}(\fm_t,\bbb_t)|_{\conn}+\wh \fp_{120}\bbb_t.
\end{equation}
\end{lemma}

\begin{proof}
Observe that in the IBL case 
$$
[\wh\fp(\bbb_te^{\fm_t})]_{\conn}=
\wh \fp_{2,1,0}\bbb_t+\wh \fp_{2,1,0}(\fm_t,\bbb_t)|_{\conn}+\wh \fp_{120}\bbb_t
$$
and use equation~\eqref{eq:gaugesimpl}.
\end{proof}

Note the similarity of equation~\eqref{eq:gaugeIBL} and
equation~\eqref{eq:MCdIBL}.

\begin{remark}\label{rem:gaugealg}
At the level $(\ell,g)$ equation~\eqref{eq:gaugeIBL} spells out
\begin{equation}\label{eq:gaugealg}
\ddt\fm_{(\ell,g)t}=\wh p_{2,1,0}\bbb_{(\ell+1,g-1)t}+\sum_{\substack{\ell_1+\ell_2=\ell+1\\ g_1+g_2=g}}
\wh \fp_{2,1,0}(\fm_{(\ell_1,g_1)t},\bbb_{(\ell_2,g_2)t})|_{\conn}+\wh \fp_{120}\bbb_{(\ell-1,g)t}.
\end{equation}
\end{remark}

\section{$\IBL_\infty$-structures associated to cochain complexes and DGAs}

In this section we recall from~\cite{Cieliebak-Fukaya-Latschev} the
canonical $\dIBL$ structure on the dual cyclic bar complex of a cylic
cochain complex, the canonical Maurer-Cartan element of a cyclic DGA,
and its pushforward to cohomology. The explicit formula~\eqref{cor:alg-MC}
for the latter will serve as a blueprint for subsequent constructions
on the de Rham complex.

\subsection{The ${\rm dIBL}$ structure associated to a cyclic cochain
  complex}\label{ss:cyc-cochain1} 

Let $(A=\bigoplus_{i=0}^nA^i,d,(\cdot,\cdot))$ be a nonnegatively
graded cyclic cochain complex of degree $n$.
To be compatible with~\cite{Cieliebak-Fukaya-Latschev}, we introduce
the {\em cyclic pairing}
\begin{equation}\label{eq:defcycpair}
  \la x,y\ra:=(-1)^{\deg x}(x,y).
\end{equation} 
Let $\{e_a\}$ be a basis of $A$, and $\{e^a\}$ the dual basis of $A$
defined by
$$
     \la e_a,e^b\ra = \delta_a^b.
$$
Recall the bar complex $BA$ and the operations $t_{alg},N_{alg}:BA\to
BA$ from~\ref{ss:gradedvect}.
We extend the differential $d$ from $A[1]$ to $BA$ as a derivation
with respect to the shifted degrees, so that it satisfies
$$
  d\circ t_{alg}=t_{alg}\circ d.
$$
We define the coproduct
\begin{equation*}
   c_{120}:BA\longrightarrow BA\otimes BA
\end{equation*}
on elementary tensors $x=x_1\otimes\cdots\otimes x_k$ of homogeneous degree by 
\begin{equation}\label{eq:defc120}
   c_{120}(x) := \sum_{k_1=0}^k
   (-1)^{|e^a||x^{(1)}|+|e_a|+(n-1)|e_ax^{(1)}|}(e_a\otimes
   x_1\otimes\cdots \otimes x_{k_1})\otimes (e^a\otimes x_{k_1+1}\otimes\cdots\otimes x_k),
\end{equation}
where we have abbreviated $x^{(1)}:=x_1\otimes\cdots\otimes x_{k_1}$.
Here and in the following the Einstein summation convention is understood. 
Analogously, we define the product 
\begin{equation*}
   c_{210}:BA\otimes BA\longrightarrow BA
\end{equation*}
on elementary tensors $x=x_1\otimes\cdots\otimes x_{k_1}$ and
$y=y_1\otimes\cdots\otimes y_{k_2}$ by
\begin{equation}\label{eq:defc210}
   c_{210}(x\otimes y):=\frac{1}{2}
   (-1)^{|e^a||x|+|e_a| + (n-1)|x|}e_a\otimes
   x_1\otimes\cdots\otimes x_{k_1}\otimes e^a\otimes y_1\otimes\cdots\otimes
   y_{k_2}. 
\end{equation}
It is straightforward to verify that the above maps
$c_{120}$ and $c_{210}$ do not depend on the chosen basis $\{e_a\}$.
They induce on the dual bar complex $B^*A$ the operations
\begin{equation}\label{eq:dIBLcanon}
\fp_{1,1,0}:=d^*,\qquad
\fp_{2,1,0}:=(c_{120}\circ N_{alg})^*,\qquad
\fp_{1,2,0}:=(c_{210}\circ N_{alg}^{\otimes 2})^*.
\end{equation}
It follows that the operations descend to the dual cyclic bar complex
$B^{\text{\rm cyc}*}A$ and we have

\begin{proposition}[{\cite{Cieliebak-Fukaya-Latschev}}]\label{prop:canondIBL}
The triple $(\fp_{1,1,0},\,\fp_{2,1,0},\,\fp_{1,2,0})$ defined in~\eqref{eq:dIBLcanon}
constitutes a ${\rm dIBL}$ structure  of degree $n-3$ on $B^{\text{\rm cyc}*}A[2-n]$.
\end{proposition}

Following~\cite{Hajek-thesis}, we denote this dIBL-algebra by 
$$
  \dIBL(A) := (B^{\text{\rm cyc}*}A[2-n],\,\fp_{1,1,0},\,\fp_{2,1,0},\,\fp_{1,2,0}).
$$

\subsection{The operations $\wh\fp_{1,1,0}$, $\wh\fp_{2,1,0}$ and
  $\wh\fp_{1,2,0}$}\label{ss:p210} 

Let $( A, d,(\cdot,\cdot))$ be a cyclic cochain complex and $\dIBL(A)$
the associated $\dIBL$ algebra. In this subsection we spell out in
this case the ``hat'' operations that appeared in the general context
in~\eqref{eq:phat}. 

The operations $c_{120}$ and $c_{210}$ on $BA$ from~\eqref{eq:defc120}
and \eqref{eq:defc210} induce operations on higher
tensor powers of $(BA)[3-n]$ by applying them to the first (resp.~first
two) tensor factors and applying the identity to all the others:
$$
   c_{120}^1:=c_{120}\otimes id^{\otimes k},\qquad c_{210}^{12}=c_{210}\otimes id^{\otimes k},
$$
with $k$ being the number extra tensor factors. Since we apply the operations 
to the first (resp.~first two) factor(s), the degree shift $[3-n]$ does not matter.
By dualizing we obtain operations on appropriate tensor powers of 
$(B^*A)[3-n]$ as follows:
\begin{gather*}
   b_{210}:=c_{120}^*,\qquad b_{210}^{12}:=c_{120}^{1*}=b_{210}\otimes id^{\otimes k},\cr
   b_{120}:=c_{210}^*,\qquad b_{120}^1:=c_{210}^{12*}=b_{120}\otimes id^{\otimes k}.
\end{gather*}
Abbreviate $C:=(B^{{\rm cyc}*}A)[2-n]$. 
Recall the operation of $N_{alg}$ on $BA$ and denote the dual operation 
on $B^*A$ by $N:=N_{alg}^*$. We extend $N$ to tensor powers of $(B^*A)[3-n]$ using 
the first (resp.~first two) tensor factor(s) again:
$$
   N^1:=N\otimes id^{\otimes k},\qquad N^{12}:=N\otimes N\otimes id^{\otimes k}.
$$
Using this, we define operations on tensor powers of $C[1]=B^{{\rm cyc}*}(A)[3-n]$ by
\begin{equation}\label{eq:p21012}
  p_{210}^{12}:=N^1\circ b_{210}^{12}=\fp_{2,1,0}\otimes\id^{\otimes k},\qquad
  p_{120}^1:=N^{12}\circ b_{120}^1 = \fp_{1,2,0}\otimes\id^{\otimes k}.
\end{equation}
We also define
$$
  p_{110}^1:=\fp_{1,1,0}\otimes id^{k}.
$$
To define the ``hat'' operations, we use the following shuffles acting on
$(B^{\rm cyc}A)[3-n]^{\otimes\ell}$ by permuting the tensor factors with signs. For
$i<j$, the shuffle $\eta_{i,j}$ moves the elements at positions $i$
and $j$ to positions $1$ and $2$, leaving the other ones in place. 
The shuffle $\eta_j$ moves the $j^{th}$ element to the first position.
Dualizing these actions, we get actions of the shuffles on the graded dual
$(B^{\rm cyc}A[3-n]^{\otimes\ell})^*$ that we denote by the same
letters. The ``hat'' operations are now given by 
\begin{equation}\label{eq:210120hat}
   \wh\fp_{1,1,0} = \sum_jp_{110}^1\circ\eta_j^{-1},\quad
   \wh\fp_{2,1,0} = \sum_{i<j}p_{210}^{12}\circ\eta_{i,j}^{-1},\quad
   \wh \fp_{1,2,0} = \sum_jp_{120}^1\circ\eta_j^{-1}.
\end{equation}
They descend to the completed symmetric product $\wh E_\ell C$ defined
in~\S\ref{ss:MC-cyc} below.


\subsection{Maurer-Cartan elements on the dual cyclic bar complex}\label{ss:MC-cyc}

As in the previous subsection, let $A$ be a cyclic cochain complex and
$$
  C:=B^{\text{\rm cyc}*}A[2-n]
$$
its dual cyclic bar complex with its canonical $\dIBL$ structure. Let
us first spell out the relevant completions in the $\IBL_\infty$-formalism in this case.
Recall from~\S\ref{ss:IBL-basic} the $k$-fold symmetric product 
$$
    E_k C = \left(C[1]\otimes \cdots \otimes C[1]\right)/S_k
    = \left(B^{\text{\rm cyc}*}A[3-n]\otimes \cdots \otimes B^{\text{\rm cyc}*}A[3-n]\right)/S_k. 
$$
The completed $k$-fold symmetric product $\wh E_k C$ is defined
in~\cite{Cieliebak-Fukaya-Latschev} as the completion of $E_kC$ with
respect to the canonical filtration on $C$ (see
Remark~\ref{rem:filtrations}). In the present case, this simply becomes
\begin{equation}\label{eq:whEk}
    \wh E_k C 
    = \left(B^{\text{\rm cyc}}A[3-n]\otimes \cdots \otimes B^{\text{\rm cyc}}A[3-n]\right)^*/S_k.
\end{equation}

In the case of a cyclic DGA we get a canonical Maurer-Cartan element:

\begin{prop}\cite[Proposition 12.5]{Cieliebak-Fukaya-Latschev}.\label{propIBLI2}
Let $(A,d,\wedge,(\ ,\ ))$ be a cyclic DGA of degree $n$. Then 
the triple intersection product $\m^\can$ in~\eqref{eq:deftriple}
defines Maurer-Cartan element $\fm^\can$ on $\dIBL(A)$ with
$(\fm^\can)_{1,0}=\m^\can$ and $(\fm^\can)_{\ell,g}=0$ for
$(\ell,g)\neq(1,0)$. 
The homology of the corresponding twisted twisted $\dIBL$-algebra 
$\dIBL^{\fm^\can}(A)$ is isomorphic to the cyclic cohomology 
$HC_{\lambda}^*(A)$.
\end{prop}

We call $\fm^\can$ the {\em canonical Maurer-Cartan
element} associated to the cyclic DGA $(A,d,\wedge,(\ ,\ ))$. 
According to~\eqref{eq:dIBLtwist}, the corresponding twisted $\dIBL$
structure is given by
\begin{equation*}
  \dIBL^{\m^\can}(A) = (\fp_{1,1,0}+\fp_{2,1,0}(\m^\can,\cdot),\,\fp_{2,1,0},\,\fp_{1,2,0}).
\end{equation*}

\subsection{The ${\rm dIBL}$ structure associated to a subcomplex}\label{ss:cyc-cochain} 

In this and the following subsection we recall the algebraic constructions
from~\cite{Cieliebak-Fukaya-Latschev} which, taken together, 
associate to each cyclic DGA a natural $\IBL_\infty$-strucure on its homology.
The most tricky part of these constructions are the signs, which we
will not spell out but refer to~\cite{Cieliebak-Fukaya-Latschev}.

Let $(A,d,(\ ,\ ))$ be a cyclic cochain complex of degree $n$, and
$B\subset A$ a quasi-isomorphic cyclic subcomplex. 
We will be mostly interested in the case that $B$ is isomorphic to the
cohomology of $(A,d)$, which is possible due to the following

\begin{lem}[{\cite[Lemma 11.1]{Cieliebak-Fukaya-Latschev}}]\label{lem:B}
There exists a quasi-isomorphic cyclic subcomplex $B\subset \ker d\subset A$ satisfying
$$
    \ker d = \im d\oplus B. 
$$ 
\end{lem}

Returning to the general case, 
we denote the induced structures on $B$ again by $d$ and
$\ (\ ,\ )$. Since $(B,d,(\ ,\ \ ))$ is again a cyclic cochain
complex, Proposition~\ref{prop:canondIBL} equips $(B^{\text{\rm cyc}*}B)[2-n]$ with a
canonical $\text{\rm dIBL}$ structure $\dIBL(B)$. 

\begin{thm}[{\cite[Theorem 11.3]{Cieliebak-Fukaya-Latschev}}]\label{thm:homotopyequiv}
Let $(A,d,(\ ,\ ))$ be a cyclic cochain complex of degree $n$, and
$B\subset A$ a quasi-isomorphic cyclic subcomplex. 
Then there exists an $\IBL_{\infty}$-homotopy equivalence 
$$
    \frak f : \dIBL(A) \to \dIBL(B)
$$
such that $\frak f_{1,1,0} : (B^{\text{\rm cyc}*}A)[2-n] \to (B^{\text{\rm cyc}*}B)[2-n]$ 
is the map induced by the dual of the inclusion $\iota:B\into A$.
\end{thm}

\begin{proof}[Sketch of proof]
We pick a basis $\{e_a\}_{a\in I}$ of $A$
and denote by $\{e^a\}_{a\in I}$ the dual basis of $A$ with
respect to the cyclic pairing $\la\ ,\ \ra$ defined in~\eqref{eq:defcycpair}, i.e.
$$
  \la e_a,e^b\ra = \delta_a^b.
$$
By Corollary~\ref{cor:retraction}, there exists a propagator $P$ whose
associated projection $\pi$ satisfies $\im\pi=B$. 
We set 
$$
    P_{ab} := \la Pe_a,e_b\ra,\quad 
    P^{ab} := \la Pe^a,e^b\ra. 
$$ 
Recall the set $RG_{k,\ell,g}$ of isomorphism classes of ribbon graphs
from Remark~\ref{rem:labelling}.
Consider $\Gamma \in RG_{k,\ell,g}$ with the $v$-th vertex of degree
$d_v$ and $s_b$ leaves ending on the $b$-th boundary component.
We associate to $\Gamma$ a map 
\begin{equation}\label{constructionf}
\mathfrak f_\Gamma : B^{\text{\rm cyc}*}_{d_1}A \otimes \cdots \otimes B^{\text{\rm cyc}*}_{d_k}A 
\to B^{\text{\rm cyc}*}_{s_1}B \otimes \cdots \otimes B^{\text{\rm cyc}*}_{s_\ell}B
\end{equation}
as follows. Let $\varphi^v \in B^{\text{\rm cyc}*}_{d_v}A$, $1\leq v\leq k$, and
$\alpha^b_j\in B$, $1\leq b\leq\ell$, $1\leq j\leq s_b$ be given. 
Let $\frak I = {\rm Map}(\Flag_\inn(\Gamma),I)$
be the set of maps from interior flags of $\Gamma$ to the index set $I$ for the
basis of $A$. Thus $\frak i \in \frak I$ assigns to each interior flag $(v,t)$
a basis vector $e_{\frak i(v,t)}$. For an exterior flag 
$(v,t)$ we {\it define} $e_{\frak i(v,t)} := \alpha^b_j$ when the leaf
$t$ meets $\p\Sigma$ in the $j$-th vertex on the $b$-th boundary
component. (Note that $\frak i(v,t)$ is a priori not defined for $\frak i \in \frak I$.)

For each vertex $v$ let $(v,1),\dots,
(v,d_v)$ be the adjacent flags in their chosen numbering.  
We arbitrarily orient the edges, so for each edge
$t$ we have two flags $(1,t)$, $(2,t)$. Now we set
\begin{equation}\label{fdefstatesum}
\begin{aligned}
    &\ \ \ \mathfrak f_\Gamma(\varphi^1 \otimes \cdots \otimes \varphi^k)
    (\alpha^1_1\cdots \alpha^1_{s_1}\otimes
    \cdots \otimes\alpha^\ell_1\cdots \alpha^\ell_{s_\ell})\cr
    &:= \frac 1 {\ell!\,\prod_v d_v}\sum_{\frak i \in
      \frak I}\pm\hspace{-10pt} \prod_{t\in
    \Edge(\Gamma)} \hspace{-10pt}P^{\frak i(1,t),\frak i(2,t)}\hspace{-10pt}
    \prod_{v \in \Ver(\Gamma)}\hspace{-10pt} \varphi^v(e_{\frak i(v,1)},
     \cdots, e_{\frak i(v,d_v)}),
\end{aligned}
\end{equation}
with suitable signs defined in~\cite{Cieliebak-Fukaya-Latschev}. 
The symmetry properties imply that $\ff_\Gamma$ induces a linear map 
$$
    \mathfrak f_\Gamma : E_k(B^{\text{\rm cyc}*}A)[2-n] \to E_\ell(B^{\text{\rm cyc}*}B)[2-n]
$$
and we define
$$
    \mathfrak f_{k,\ell,g} := \sum_{\Gamma\in RG_{k,\ell,g}}\mathfrak
    f_\Gamma: E_k(B^{\text{\rm cyc}*}A)[2-n] \to E_\ell(B^{\text{\rm cyc}*}B)[2-n]. 
$$
It is straightforward but tedious to verify that the collection of
maps $\mathfrak f_{k,\ell,g}$ defines an $\IBL_\infty$-morphism $\frak
f : \dIBL(A) \to \dIBL(B)$. Since $\frak f_{1,1,0} : (B^{\text{\rm
    cyc}*}A)[2-n] \to (B^{\text{\rm cyc}*}B)[2-n]$ induces an
isomorphism on homology, $\frak f$ is an $\IBL_\infty$-homotopy equivalence. 
\end{proof}

\begin{remark}
Equation (11.7) in~\cite{Cieliebak-Fukaya-Latschev}, which
corresponds to our equation~\eqref{fdefstatesum}, has an additional
factor $1/|\Aut(\Gamma)|$ on the right hand side which we do not have. 
The reason is the following. In~\cite{Cieliebak-Fukaya-Latschev},
$\Gamma$ denotes an unlabelled ribbon graph and equation (11.7) sums
over all ``labellings'' of $\Gamma$ (corresponding to items (i)--(iv)
in our Definition~\ref{def:labelling}). This leads to an overcounting
if different ``labellings'' are related by automorphisms of $\Gamma$,
which is compensated by dividing by $|\Aut(\Gamma)|$. By contrast, we
denote by $\Gamma$ an {\em isomorphism class of ``labelled'' ribbon graphs},
so different ``labellings'' related by automorphisms of the underlying
unlabelled graph are already identified and there is no further
division by the order of the automorphism group.  
\end{remark}

\subsection{Pushing forward the canonical Maurer-Cartan element to homology}\label{ss:MC-pushforward}

Recall from~\eqref{eq:deftriple} and Proposition~\ref{propIBLI2}
the canonical triple product MC element $\fm^{\can}$ for 
$\dIBL(A)$ and the corresponding twisted 
structure $\dIBL^{\fm^\can}(A)$.
This structure can be further pushed to the cohomology $H = H(A,d)$ of
the cyclic DGA $(A,d,\wedge,\la\ ,\ \ra)$ as follows. The
product and the pairing descend to cohomology and make it again a
cyclic DGA (with trivial differential)
$(H,d=0,\wedge,(\ ,\ \ ))$. Proposition~\ref{prop:canondIBL} 
associates to this cyclic DGA a canonical dIBL-algebra (actually an IBL-algebra) 
$$
   \dIBL(H) = \Bigl((B^{{\text{\rm cyc}}*}H)[2-n],\fp_{1,1,0}=0,\,\fp_{1,2,0},\,\fp_{2,1,0}\Bigr).
$$
By Lemma~\ref{lem:B}, we can embed $H$ as a subcomplex into $A$. Then
Theorem~\ref{thm:homotopyequiv} (with $B$ replaced by $H$)
provides an $\IBL_{\infty}$-homotopy equivalence 
$$
    \frak f : \dIBL(A) \to \dIBL(H)
$$
such that $\frak f_{1,1,0} : (B^{{\text{\rm cyc}}*}A)[2-n] \to (B^{{\text{\rm cyc}}*}H)[2-n]$ 
is the map induced by the dual of the inclusion $\iota:H\into
A$ from Lemma~\ref{lem:B}. According to Proposition~\ref{prop:MC}(b),
the Maurer-Cartan element $\fm^\can$ on
$(B^{{\text{\rm cyc}}*}A)[2-n]$ can be pushed forward via $\ff$ to a
Maurer-Cartan element $\ff_*\fm^\can$ on $(B^{{\text{\rm cyc}}*}H)[2-n]$. By
Proposition~\ref{prop:MC}(a), this induces a twisted $\IBL_\infty$-structure
$$
    \fp^{\ff_*\fm^\can} =
    \{\fp^{\ff_*\fm^\can}_{1,\ell,g},\fp_{2,1,0}\}_{\ell\geq 1,g\geq 0}
$$
on $(B^{\text{\rm cyc}*}H)[2-n]$. By Proposition~\ref{prop:MC}(c), 
this structure is homotopy equivalent to the twisted $\text{\rm dIBL}$-structure
$\dIBL^{\fm^\can}$. We summarize the discussion in the following theorem.

\begin{thm}\label{thm:homologyBLI}
Let $(A,d,\wedge,(\ ,\ \ ))$ be a cyclic DGA of degree $n$ with cohomology
$H=H(A,d)$. Then the triple intersection product $\m^{\can}$ induces via homotopy
transfer a Maurer-Cartan element 
$\ff_*\fm^{\can}$ on $\dIBL(H)$ such that
the associated twisted $\IBL_\infty$-structure $\dIBL^{\ff_*\fm^{\can}}(H)$ is
$\IBL_\infty$-homotopy equivalent to the twisted $\text{\rm
  dIBL}$-structure $\dIBL^{\fm^{\can}}(A)$ in Proposition~$\ref{propIBLI2}$. 
In particular, its homology equals the Connes cyclic cohomology $HC_{\lambda}^*(A)$ of 
$(A,d,\wedge)$. 
\qed
\end{thm}

The construction of the map $\ff$ leads to the following explicit
description of the pushforward Maurer-Cartan element $\ff_*\fm^{\can}$ on
$(B^{\text{\rm cyc}*}H)[2-n]$. 
Recall from Definition~\ref{def:R} the set $\RR_{\ell,g}$ of isomorphism classes of 
labelled connected trivalent ribbon graphs of genus $g$ with $\ell$
boundary components.

\begin{cor}[{\cite[Remark 12.11]{Cieliebak-Fukaya-Latschev}}]\label{cor:alg-MC}
In the situation of Theorem~\ref{thm:homologyBLI}, the value of 
$(\ff_*\fm^\can)_{\ell,g}$ on a tensor product $\alpha=\alpha^1_1\cdots \alpha^1_{s_1}\otimes\dots\otimes 
   \alpha^\ell_1\cdots \alpha^\ell_{s_\ell}$ of $\alpha^b_j\in
H$ is given by  
$$
   (\ff_*\fm^\can)_{\ell,g}(\alpha) = 
   \sum_{\Gamma\in\RR_{\ell,g}}(\ff_*\fm^\can)_\Gamma(\alpha),
$$
where 
\begin{align}\label{mdefstatesum}
    (\ff_*\fm^\can)_\Gamma(\alpha)
    := \frac 1 {\ell!}\sum_{\frak i \in
      \frak I}\pm\hspace{-10pt} \prod_{t\in
    \Edge(\Gamma)} \hspace{-10pt}P^{\frak i(1,t),\frak i(2,t)}\hspace{-10pt}
    \prod_{v \in \Ver(\Gamma)}\hspace{-10pt} \m^{\can}(e_{\frak i(v,1)}, e_{\frak i(v,2)}, e_{\frak i(v,3)}).
\end{align}
\end{cor}

\begin{proof}
The statement follows from Theorem~\ref{thm:homologyBLI}. To see this,
define $k$ by equation~\eqref{eq:k} with $s:=s_1+\dots+s_\ell$ and
denote by $RG_{k,\ell,g}^3$ denote the subset of $RG_{k,\ell,g}$
corresponding to trivalent graphs. Observe that the natural forgetful map
$$ 
  RG_{k,\ell,g}^3\longrightarrow \RR_{\ell,g}
$$ 
is a $k!3^k$ to $1$ map. Indeed, the choice of item (iii) in
Definition~\ref{def:labelling} accounts for the factor $k!$,
and the choice of item (iv) accounts for the factor $3^k$ because
our graphs are trivalent. Using this, we compute
\begin{align*}
  (\ff_*\fm^\can)_{\ell,g}(\alpha)
  &= \frac{1}{k!}\ff_{k,\ell,g}(\m^\can\otimes\cdots\otimes\m^\can)(\alpha) \cr
  &= \sum_{\Gamma\in RG_{k,\ell,g}^3} \frac{1}{k!}\ff_\Gamma(\m^\can\otimes\cdots\otimes\m^\can)(\alpha) \cr
  &= \sum_{\Gamma\in \RR_{\ell,g}} 3^k\ff_\Gamma(\m^\can\otimes\cdots\otimes\m^\can)(\alpha).
\end{align*}
Here the first equality follws from definition of the pushforward of a
Maurer-Cartan element, the second one from the definition of
$\ff_{k,\ell,g}$, and the third one from the preceding discussion.
Now we substitute $\ff_\Gamma(\cdots)$ in the last line by
formula~\eqref{fdefstatesum}. Since $\Gamma$ is trivalent, the factor
$3^k$ offsets the factor $1/\prod_v d_v$ in~\eqref{fdefstatesum}
and we obtain~\eqref{mdefstatesum}.
\end{proof}


To prove Theorem~\ref{thm:existence-intro} from the Introduction, we
would like to apply Corollary~\ref{cor:alg-MC} to the de Rham algebra
$(\Om^*(M),d,(\cdot,\cdot))$ of a closed oriented manifold $M$ with
its intersection pairing~\eqref{eq:defineintpair}. Unfortunately, this
pairing is nondegenerate but not perfect. Hence the de Rham
algebra is not a cyclic DGA, the canonical Maurer-Cartan element does
not exist, and we cannot directly apply the constructions
of~\S\ref{ss:cyc-cochain}. In the next sections we will define a
Maurer-Cartan element on de Rham cohomology, which plays a role of the
``pushforward of the canonical one'', by replacing the sums over basis
elements by integrals over configuration spaces.

\section{Propagators}\label{sec:prop}

In this section we discuss the propagators (or integral kernels)
that will enter the construction of the Maurer-Cartan element. 
Our discussion is very similar to the one in~\cite{Hajek-thesis}, see
also~\cite{Cieliebak-Hajek-Volkov}. From the formal algebraic viewpoint 
we are in the setting of~\S\ref{ss:hodgesetup}, with the role of the
ambient cochain complex played by the de Rham algebra of a closed
oriented manifold.

\subsection{Fibre integration and Poincar\'e duality}

In this subsection we collect some standard facts about fibre
integration and Poincar\'e duality and fix some notation.

{\bf Fibre integration. }
Let $F\longrightarrow E\stackrel{p}{\longrightarrow}B$ be an oriented smooth fibre bundle with compact
oriented fibre manifold $F$ (possibly with boundary) over an oriented
base manifold $B$ (without boundary). Then the total space $E$
inherits an orientation according to the convention ``fibre first, base second''. 
The following result can be found in~\cite{Bott-Tu} with a different
orientation convention; a proof with our convention is given in~\cite{Hajek-thesis}.

\begin{lem}[Fibre integration]\label{lem:fibre-int}
There exists a unique linear map $p_*:\Om^*(E)\to\Om^{*-\dim F}(B)$
with the following properties for all $\alpha\in\Om^*(E)$ and $\beta\in\Om^*_c(B)$:
\begin{equation}\label{eq:fibre-Fubini}
   \int_E\alpha\wedge p^*\beta = \int_B p_*\alpha\wedge\beta;
\end{equation}
\begin{equation}\label{eq:fibre-projection}
   p_*(\alpha\wedge p^*\beta) = p_*\alpha\wedge\beta;
\end{equation}
\begin{equation}\label{eq:fibre-Stokes}
   (-1)^{\dim F}dp_*\alpha = p_*d\alpha - (p|_{\p F})_*(\alpha|_{\p E}).
\end{equation}
\end{lem}

Note that if $\alpha$ has compact support we can take $\beta=1$ in~\eqref{eq:fibre-Fubini}
and obtain $\int_E\alpha = \int_B\int_F\alpha$ (which is a variant of
Fubini's theorem), and the general case
of~\eqref{eq:fibre-Fubini} follows from this special case by
integrating~\eqref{eq:fibre-projection} over $B$.  
Note also that in the case $\p F=\varnothing$ equation~\eqref{eq:fibre-Stokes} says that $p_*$
is a chain map up to sign. 

The map $p_*$ is called {\em pushforward} or {\em fibre integral}. We
will also write it as $\int_F$. In this notation, it is apparent that~\eqref{eq:fibre-Stokes}
is just a variant of Stokes' theorem:
\begin{equation}\label{eq:fibre-Stokes2}
   (-1)^{\dim F}d(\int_F\alpha) = \int_Fd\alpha - \int_{\p F}\alpha.
\end{equation}

{\bf Poincar\'e duality. }
Our next goal is to express the Poincar\'e dual of the diagonal 
$\Delta_2\subset M\times M$ in terms of the basis $e_a$. We first
recall the basic properties of cup and cap products, following the 
conventions in Hatcher~\cite{Hatcher}. Let $X$ be a closed oriented
manifold (which will later be $M\times M$). We denote homology classes on $X$ by
$a,b,c$ and cohomology classes by $\alpha,\beta,\gamma$. We write the
pairing between cohomology and homology as $\int_a\beta$. Then the cup
and cap product are related by
$$
   \int_c\alpha\cup\beta = \int_{c\cap\alpha}\beta.
$$
Cap product with the fundamental class defines the Poincar\'e
duality isomorphism
$$
   PD: H^k(X)\to H_{\dim X-k}(X),\qquad PD(\alpha):=[X]\cap \alpha. 
$$ 
So we have
$$
   \int_X\alpha\cup\beta = \int_{PD(\alpha)}\beta.
$$
We denote the inverse map to Poincar\'e duality also by $PD$. 
The intersection product of two homology classes is then given by
\begin{equation}\label{eq:PD-intersection}
   a\cap b = \int_a PD(b) = \int_X PD(a)\cup PD(b). 
\end{equation}

\subsection{Harmonic projections}\label{ss:harmproj}

From now on $M$ denotes a closed oriented manifold of dimension
$n$. Recall the de Rham complex
$$
   \bigl(\Om = \Om^*(M),d,\wedge\bigr)
$$
with the intersection pairing $(\alpha,\beta) = \int_M\alpha\wedge\beta$
defined in~\eqref{eq:defineintpair}.
Let us fix a complementary subspace $\HH$ to 
$\im d$ in $\ker d$,
i.e.~such that
$$
   \ker d = \im d\oplus\HH. 
$$
The space $\HH$ is a harmonic subspace in the sense of~\S\ref{ss:hodgesetup}
and we will refer to its elements as {\em harmonic forms}, although
they need not be harmonic with respect to any metric. The de Rham
cohomology $H^*(M)$ is finite dimensional and the induced pairing
on $H^*(M)$ is nondegenerate.
Therefore, by Lemma~\ref{lem:Horthog},
$$
   \HH^\perp = \{\alpha\in\Om\mid (\alpha,\beta)=0 \text{ for all }\beta\in\HH\}
$$
is a complement to $\HH$ in $\Om$ and we have the orthogonal projection
$$
   \Pi:\Om = \HH\oplus\HH^\perp \to \HH. 
$$
We pick a basis $h_i$ of $\HH$ and define its dual basis $h^i$ by 
$$
   (h_i,h^j) = \delta_i^j.
$$
In terms of these bases (see the proof of Lemma~\ref{lem:Horthog}) we have
\begin{equation}\label{eq:projdef}
   \Pi = \sum_i (h_i,\cdot)h^i,
\end{equation}
or more explicitly (using $\deg\alpha=\deg h^i$), 
\begin{equation}\label{eq:int-kernel}
   (\Pi\alpha)(y)
   = \sum_i \Bigl(\int_{M_x}h_i(x)\wedge\alpha(x)\Bigr)h^i(y)
   = \int_{M_x}\Pi(x,y)\wedge\alpha(x)
\end{equation}
with the smooth integral kernel $\Pi\in\Om^n(M_x\times M_y)$ (which we
denote by the same letter by a slight abuse of language) given by 
\begin{equation}\label{eq:Pi-kernel-h}
   \Pi(x,y) = \sum_i(-1)^{h^i}h_i(x)\wedge h^i(y).
\end{equation}
Here and in the sequel we sometimes denote by $M_x$ the factor of $M$
corresponding to the variable $x$. The integration over $M_x$ in~\eqref{eq:int-kernel} is
viewed as the fibre integral with respect to the projection $M_x\times
M_y\to M_y$ onto the second factor. This projection is chosen so that
the convention ``fibre first, base second'' gives the canonical
orientation of $M_x\times M_y$. 

\begin{lem}\label{lem:Poincare}
The integral kernel~\eqref{eq:Pi-kernel-h} is closed and represents
the Poincar\'e dual to the diagonal $\Delta_2=\{x=y\}\subset M\times M$.
Moreover, it has the symmetry
\begin{equation}\label{eq:Pi-symmetry}
   \Pi(x,y)=(-1)^n\Pi(y,x). 
\end{equation}
\end{lem}

\begin{proof}
Closedness of $\Pi$ follows from closedness of the $h_i$ and $h^i$. 
The previous subsection yields the following characterization 
of the Poincar\'e dual of a closed oriented submanifold $Y$ of a closed
oriented manifold $X$: A closed form $\alpha$ on $X$ represents the Poincar\'e dual of $Y$
if and only if for each closed form $\beta$ on $X$ we have
$$
   \int_X\alpha\wedge\beta=\int_Y\beta.  
$$
It is enough to check this equality on a set of closed forms
representing a basis of cohomology in the relevant degree. We now
apply this to $X:=M\times M$,
$Y:=\Delta_2$, $\alpha=\Pi$, and $\beta:=h^k(x)\wedge h_\ell(y)$ 
for some $k,\ell$. Then the right hand side of the above equation becomes
$$
   \int_M h^k(x)\wedge h_\ell(x) = (-1)^{h^kh_l}\delta_\ell^k.
$$ 
The left hand side becomes
\begin{align*}
  \sum_i\int_{M\times M}(-1)^{h^i}h_i(x)\wedge h^i(y)\wedge
  h^k(x)\wedge h_\ell(y) 
  = \sum_i(-1)^\sigma\delta_i^k\delta_\ell^i = (-1)^\sigma\delta_\ell^k
\end{align*}
with the sign exponent $h^i$ plus the one arising from moving $h^i(y)$ to the 
right past $h^k(x)\wedge h_\ell(y)$. Since both sides vanishes unless $i=\ell=k$, we obtain
$$
   \sigma = h^i + h^i(h^k+h_\ell) 
   \overset{i=\ell=k}{=} h^kh_k
$$
and the Poincar\'e duality assertion follows. For the symmetry of the
integral kernel $\Pi(x,y)$, we note that by Lemma~\ref{lem:Horthog}
the projection $\Pi$ is symmetric,
$$
   (\Pi\alpha,\beta) 
   = (\alpha,\Pi\beta),
$$
and compute both sides:
\begin{align*}
  (\Pi\alpha,\beta)
  &= \int_{M_y}\Bigl(\int_{M_x}\Pi(x,y)\wedge\alpha(x)\Bigr)\wedge\beta(y) 
  = \int_{M_x\times M_y}\Pi(x,y)\wedge\alpha(x)\wedge\beta(y), \cr
  (\alpha,\Pi\beta)   
  &= \int_{M_x}\alpha(x)\wedge\Bigl(\int_{M_y}\Pi(y,x)\wedge\beta(x)\Bigr) \cr
  &= (-1)^{\alpha\beta}\int_{M_x}\Bigl(\int_{M_y}\Pi(y,x)\wedge\beta(x)\Bigr)\wedge\alpha(x) \cr
  &= \int_{M_y\times M_x}\Pi(y,x)\wedge\alpha(x)\wedge\beta(y)
  = (-1)^n\int_{M_x\times M_y}\Pi(y,x)\wedge\alpha(x)\wedge\beta(y).
\end{align*}
Since this holds for all $\alpha,\beta$, the symmetry follows. 
\end{proof}

{\bf Switching to the algebraic convention. }
In order to be consistent with~\cite{Cieliebak-Fukaya-Latschev},
we now replace the pairing~\eqref{eq:defineintpair} by 
the cyclic one, see~\eqref{eq:defcycpair}. Explicitly, 
\begin{equation}\label{eq:pairing}
   \la\alpha,\beta\ra := (-1)^\alpha\int_M\alpha\wedge\beta.
\end{equation}
This leaves the harmonic subspace $\HH$ and the orthogonal
projection $\Pi:\Om\to\HH$ unchanged, and it makes
$(\HH,d,\la\ ,\ \ra)$ a cyclic
cochain complex in the sense of~\cite{Cieliebak-Fukaya-Latschev}. 
For a basis $e_a$ of $\HH$ we now define its dual basis $e^a$ by 
$$
   \la e_a,e^b\ra = \delta_a^b.
$$
Then bases $h_a,h^a$ as above determine bases $e_a,e^a$ by
$$
   e_a=h_a,\qquad e^a=(-1)^{e_a}h^a.
$$
The kernel of $\Pi$ writes in the new bases
\begin{equation*}
   \Pi(x,y) = \sum_i(-1)^{e^a+e_a}e_a(x)\wedge e^a(y) = (-1)^n\sum_ie_a(x)\wedge e^a(y),
\end{equation*}
which in view of the symmetry of $\Pi$ becomes
\begin{equation}\label{eq:Pi-kernel}
   \Pi(x,y) = \sum_ie_a(y)\wedge e^a(x).
\end{equation}

\subsection{Oriented real blow-up and propagators}\label{ss:blowupprop}

We denote by
$$
   \wt M^2 := \Bl(M^2,\Delta_2)
$$
the {\em oriented real blow-up} of the diagonal $\Delta_2$ in $M^2=M\times M$.
This is the compact oriented manifold with boundary obtained by
replacing the diagonal by its unit sphere normal bundle $N_{\Delta_2}$.
Thus the boundary $\p\wt M^2$ is canonically diffeomorphic to
$N_{\Delta_2}$. Note, however, that the orientation of $\p\wt M^2$ as
boundary of $\wt M^2$ is {\em opposite} to the orientation of
$N_{\Delta_2}$ as boundary of the unit disk normal bundle, oriented by the
usual convention ``fibre first, base second''.  

The oriented real blow-up comes with a smooth blow-down map
$$
    \wt M^2\stackrel{\pi}\longrightarrow M^2
$$
which restricts to a diffeomorphism $\wt M^2\setminus\p\wt M^2\to
M^2\setminus\Delta_2$ on the interior and to the bundle projection
$N_{\Delta_2}\to\Delta_2$ on the boundary. The projections $p_i:M\times M\to
M$ onto the two factors induce smooth fibre bundles
$$
   p_i:\wt M^2\to M,\qquad i=1,2
$$
with fibre the oriented real blow-up of $M$ at a point. The map
$\tau(x,y)=(y,x)$ canonically lifts to an involution
$$
   \tau:\wt M^2\to\wt M^2. 
$$
 
We denote the pullback of $\Pi\in\Om^n(M^2)$ from~\eqref{eq:Pi-kernel}
under the blow-down map by $\wt\Pi$. We will view $\wt M^2$ as the
fibre bundle 
$$
   p_2:F\to \wt M^2\to M
$$
via projection onto the second factor and denote by $\int_F$ the
corresponding fibre integration. We orient the sphere $\p F$ as the
boundary of $F$, which is {\em opposite} to its orientation as the
boundary of a unit normal disk. 


\begin{lemma}\label{lem:Green}
The form $\wt\Pi$ is exact. Moreover,
there exists a (non-unique) smooth $(n-1)$-form $\wt G$ on $\wt M^2$ such that
\begin{equation}\label{eq:propag}
   d\wt G = (-1)^n\wt\Pi.
\end{equation}
Any such $G$ satisfies
\begin{equation}\label{eq:fibre-int-G}
   \int_{\p F}\wt G=(-1)^n,
\end{equation}
and it can be chosen to also satisfy
\begin{equation}\label{eq:Green-symmetry}
   \tau^*\wt G = (-1)^n \wt G.
\end{equation}
\end{lemma}

\begin{proof}
Since $H^n(\wt M^2;\R)\cong{\rm Hom}(H_n(\wt M^2;\R),\R)$,
exactness of $\wt\Pi$ is equivalent to
the vanishing of $\int_{\wt c}\wt\Pi$ for each $n$-cycle
$\wt c$ in $\wt M^2$. Now each such $\wt c$ can be deformed to make it
disjoint from $\p\wt M^2$, so that $\wt c$ projects to an $n$-cycle
$c$ in $M^2$ disjoint from $\Delta_2$ and using
equation~\eqref{eq:PD-intersection} we obtain 
$$
   \int_{\wt c}\wt\Pi = \int_c\Pi = c\cap\Delta_2 = 0.
$$
This shows exactness $\wt\Pi$ and thus  the existence of a primitive $\wt G$ of 
$(-1)^n\wt\Pi$. For the second assertion, consider the $n$-cycle
$c=\{p\}\times M$ in $M^2$ for some $p\in M$. It lifts to an $n$-chain
$\wt c$ in $\wt M^2$ whose boundary equals the fibre $F$ over $p$, 
so by Stokes'theorem and equation~\eqref{eq:PD-intersection} we find  
$$
   \int_{F}\wt G = \int_{\wt c}d\wt G = (-1)^n\int_c\Pi = (-1)^nc\cap\Delta_2 = (-1)^n.
$$
The last assertion follows by replacing $\wt G$ by $\frac12(
\wt G+(-1)^n\tau^*\wt G)$, recalling that $\tau^*\Pi=(-1)^n\Pi$.
\end{proof}

By a slight abuse of language we will call $\wt G$ as in
Lemma~\ref{lem:Green} a {\em propagator}. It 
gives rise to a linear map
$$
   P:\Om^*(M)\longrightarrow \Om^{*-1}(M)
$$
by the formula
\begin{equation}\label{eq:HG}
   P\alpha(y) := \int_{x\in M}G(x,y)\alpha(x) = \int_F\wt G\wedge p_1^*\alpha,
\end{equation}
where the right hand side is the fibre integral with respect to the
projection $p_2:\wt M^2\to M$ onto the second factor as in Lemma~\ref{lem:Green}. 
 
\begin{lemma}
The map $P$ defines a chain homotopy between $\Id$ and $\Pi$,
\begin{equation}\label{eq:homo}
   d\circ P + P\circ d = \Pi - \Id.
\end{equation}
\end{lemma}

\begin{proof}
For $\alpha\in\Om^*(M)$ we have by definition
$$
   P\alpha = \int_F\wt G\wedge p_1^*\alpha,\qquad P(d\alpha) = \int_F\wt G\wedge p_1^*d\alpha.
$$
We compute
\begin{align*}
  d\circ P\alpha
  &= d\Bigl(\int_F\wt G\wedge p_1^*\alpha\Bigr) \cr
  &= (-1)^n\int_Fd(\wt G\wedge p_1^*\alpha)
  - (-1)^n\int_{\p F}\wt G\wedge p_1^*\alpha \cr
  &= (-1)^n\int_Fd\wt G\wedge p_1^*\alpha - \int_F\wt G\wedge d(p_1^*\alpha)
  - (-1)^n\int_{\p F}\wt G\wedge p_1^*\alpha \cr
  &= \int_F\Pi\wedge p_1^*\alpha - \int_F\wt G\wedge p_1^*d\alpha
  - (-1)^n\Bigl(\int_{\p F}\wt G\Bigr)\wedge\alpha \cr
  &= \Pi\alpha - P\circ d\alpha - \alpha. 
\end{align*}
Here in the second equality we have used~\ref{eq:fibre-Stokes2}, in
the fourth equality~\eqref{eq:propag}, and in the fifth
equality~\eqref{eq:fibre-projection} and~\eqref{eq:fibre-int-G}.
\end{proof}

\begin{remark}\label{rem:symmGG}
The symmetry~\eqref{eq:Green-symmetry} for the integral kernel $\wt G$
implies that the homotopy operator $P$ is symmetric, i.e., it is a
propagator in the terminology of \S\ref{ss:coch} (see
equation~\eqref{eq:symmpropprelim}). 
\end{remark}

\begin{remark}\label{rem:spec}
We can apply Lemma~\ref{lem:propagator}
to the propagator $P$ to get a semi-special 
propagator $P_2$ for the same projection 
$\Pi$. Propositions~4.2.8 and 4.2.11 of~\cite{Hajek-thesis} imply that
there exists a smooth $(n-1)$-form $\wt G_2$ on $\wt M^2$
satisfying the conditions~\eqref{eq:propag},
\eqref{eq:fibre-int-G} and~\eqref{eq:Green-symmetry} of Lemma~\ref{lem:Green}
such that $P_2$ is obtained from $G_2$ by the formula~\eqref{eq:HG}.
(The corresponding assertions for a special propagator in the
  literature contain gaps.)
\end{remark}

We will denote the pushforward of $\wt G$ to $M\times M$ (which is singular along the diagonal) again by $G$.  
Combining~\eqref{eq:propag}, ~\eqref{eq:Pi-kernel}
and~\eqref{eq:Pi-symmetry} we then have
$$
   dG(x,y) = (-1)^n\Pi(x,y) = \Pi(y,x) = \sum_ae_a(x)\wedge e^a(y),
$$
hence
\begin{equation}\label{eq:dG}
   d\wt G = \sum_a\pi_1^*e_a\wedge \pi_2^*e^a. 
\end{equation}
In the sequel we will often abbreviate the above sum as $e_a\times
e^a$, using the cross product notation and the Einstein summation
convention, or even as $e_ae^a$ to save space. Let 
$$
\iota:\HH\otimes\HH^*\to\HH\otimes\HH\to
\Om^*(M^2)
$$
denote the composition of the following two maps: the first one is the
identification of $\HH^*$ with $\HH$ by means of the pairing
$\la\cdot,x\ra\mapsto x$, and the second one is the cross product.
Observe that $e_a\times e^a\in \Om^*(M^2)$
is the image under $\iota$ of the identity
$$
\id=\sum_a e_a\otimes \la\cdot,e^a\ra \in  
\HH\otimes\HH^*\cong \Hom(\HH,\HH).
$$
In particular the sum $e_a\times e^a$ depends only on $\HH$ and not on
the choice of the basis. The discussion in this section can be
summarized in the following result.  

\begin{proposition}\label{prop:existsprop}
Let $M$ be a closed oriented manifold and $\bigl(\Om =
\Om^*(M),d,\,\wedge\bigr)$ its de Rham algebra equipped with 
the algebraic pairing $\la\cdot,\cdot\ra$, see~\eqref{eq:pairing}. Fix any 
complement $\HH$ of $\im d$ in $\ker d$ and define the projection $\Pi$ onto 
$\HH$ by~\eqref{eq:projdef}. Let $e_a$ be a basis 
of $\HH$ and $e^a$ the dual basis with respect to $\la\cdot,\cdot\ra$. Let 
$\wt G$ be any symmetric primitive of $\pi^*(\sum_a e_a\times e^a)$
on $\wt M^2$. Then the integral operator $P$ defined by~\eqref{eq:HG}
is a propagator with respect to $\Pi$ in the sense of
\S\ref{ss:coch}. Moreover,  the primitive $\wt G$ can be chosen in
such a way that the propagator $P$ is semi-special.
\end{proposition}

\section{Configuration spaces}\label{sec:config}

Throughout this section, $\Gamma$ will denote a trivalent ribbon graph of
signature $(k,\ell,g)$
with a chosen extended labelling in the sense of
Definition~\ref{def:labelling}. We will use the notation from \S\ref{sec:graphs}. 
In particular, $s_b$ is the number of leaves on the $b$-th boundary component, 
$s=s_1+\cdots+s_\ell$, $k$ is the number of vertices determined
by~\eqref{eq:k}, $e$ is the number of edges, and 
$$
\bar R_\Gamma:=O_v^{-1}\circ O_e:
\{1,\dots,|\Flag(\Gamma)|\}\longrightarrow \{1,\dots,|\Flag(\Gamma)|\}
$$
is the reordering permutation from~\eqref{eq:reod}. 

In addition, we fix an $n$-dimensional oriented closed manifold $M$. 

\subsection{Configuration spaces associated to trivalent ribbon graphs}\label{ss:config}

The manifold $M$ defines a contravariant functor from
the category of finite sets to the category of manifolds, associating
to each finite set $S$ the manifold $M^S:={\rm Map}(S,M)$ and to each
map $\phi:S\to T$ the smooth map 
$$
   M^\phi:M^T\to M^S,\qquad x\mapsto x\circ\phi.
$$
If $S=\{1,\dots,s\}$ we identify $x\in M^S$ via $x_i:=x(i)$ with the
$s$-tuple $(x_1,\dots,x_s)\in M^s$, and the map induced by $\phi:S\to T=\{1,\dots,t\}$
is identified with
$$
   M^\phi:M^t\to M^s,\qquad (x_1,\dots,x_t)\mapsto(x_{\phi(1)},\dots,x_{\phi(s)}).
$$
In this notation, the reordering permutation $\bar R_\Gamma$ induces a diffeomorphism
$$
   R_\Gamma:=M^{\bar R_\Gamma}:M^{3k}\stackrel{\cong}\longrightarrow M^{3k}.
$$ 
Recall from \S\ref{ss:basiccomb} that precomposing the vertex order
with a permutation $\sigma^{-1}$ and the edge order with
a permutation $\eta$ results in replacing
$\bar R_\Gamma$ by $\sigma\bar R_\Gamma\eta$. Since the functor defined by $M$ is contravariant, $R_\Gamma$ gets replaced by
$M^\eta R_\Gamma M^\sigma$. It will be convenient to denote the domain
and target of the map $R_\Gamma$ by different symbols $Y_\Gamma$ and
$X_\Gamma$ so that we have
$$
   R_\Gamma:Y_\Gamma\longrightarrow X_\Gamma.
$$
We think of $Y_\Gamma$ and $X_\Gamma$ as $M^{3k}$ written in the
vertex and edge order, respectively. To emphasize this, we will also write
$$
   Y_\Gamma = (M^3)^k,\qquad
   X_\Gamma = (M^2)^e\times M^s.
$$
Let $\Delta_3\subset M^3$ be the triple diagonal and define the (slim)
{\em vertex diagonal}
$$
   \Delta_3^k\subset Y_\Gamma
$$
with its natural orientation, as well as its image
\begin{equation}\label{eq:Delta3Gamma}
   \Delta_{3\Gamma}:=(-1)^{\bar
     R_\Gamma+(n-1)\eta_3(\Gamma)}R_\Gamma(\Delta_3^k)\subset X_\Gamma.
\end{equation}
Here we orient $R_\Gamma(\Delta_3^k)$ by declaring the restriction
$R_\Gamma|_{\Delta_3^k}$ to be orientation preserving and writing
the desired sign explicitly in front. The exponent $\bar R_\Gamma$  
denotes the sign exponent of the permutation $\bar R_\Gamma$, and 
$\eta_3(\Gamma)$ was introduced 
in \S\ref{ss:basiccomb}.

\begin{lemma}\label{lem:Delta3Gamma}
The orientation of $\Delta_{3\Gamma}$ does not depend on the vertex
order of flags, i.e.~items (iii) and (iv) in Definition~\ref{def:labelling}.
\end{lemma}

\begin{proof}
A change in items (iii) and (iv) of the extended labelling changes the
vertex order of flags by precomposition with a permutation $\sigma^{-1}$.
We denote the graph with the new extended labeling by $\Gamma'$ and
set $\Delta\eta_3:=\eta_3(\Gamma')-\eta_3(\Gamma)$.
Since the new reordering maps are $\bar R_{\Gamma'}=\sigma\bar
R_\Gamma$ and $R_{\Gamma'}=R_\Gamma M^\sigma$, the new triple diagonal is
$$
  \Delta_{3\Gamma'}
  = (-1)^{\sigma+\bar R_{\Gamma}+(n-1)\eta_3(\Gamma')}R_\Gamma
  M^\sigma(\Delta_3^k)
  = (-1)^{(n-1)(\sigma+\Delta\eta_3)}\Delta_{3\Gamma},
$$
where we have used $M^\sigma(\Delta_3^k)=(-1)^{n\sigma}\Delta_3^k$.
(Note that this last equality essentially uses the fact that our graph 
is odd-valent; this is in sharp contrast
with~\cite{Cieliebak-Fukaya-Latschev}, where statements analogous to
Lemma~\ref{lem:Delta3Gamma} do not rely on the valency, which is
allowed to be arbitrary.) 
Now we distinguish two cases.
If $\sigma$ corresponds to swapping the order of two adjacent
vertices, then the sign exponents are $\Delta\eta_3=1=\sigma$.   
If $\sigma$ corresponds to cyclically changing the order of adjacent
flags at a vertex, then the sign exponents are $\Delta\eta_3=0=\sigma$.   
In both cases we conclude $\Delta_{3\Gamma'}=\Delta_{3\Gamma}$. 
\end{proof}

Let now $l$ be an edge of $\Gamma$. Let 
$\Delta_2\subset M^2$
be the diagonal with its canonical orientation and define the double
diagonal corresponding to the edge $l$,
$$
 \Delta_2^l:=(M^2\times\cdots\times M^2\times \Delta_2\times
 M^2\times\cdots\times M^2)\times M^s\subset X_\Gamma.
$$
Here $\Delta_2$ comes at the position corresponding to the edge $l$ in the
numbering (v) in Definition~\ref{def:labelling}. We orient $\Delta_2^l$ as a product and
note that this orientation does not depend on the orientation of the edge $l$.
We also denote
$$
   \Delta_2^{\{l\}} := \Delta_2^l\setminus \bigcup_{k\in\Edge(\Gamma)\setminus \{l\}}\Delta_2^k
$$
and we define the (fat) {\em edge diagonal}
$$
   \Delta_{2\Gamma} := \bigcup_{l\in\Edge(\Gamma)}\Delta_2^l\subset X_\Gamma.
$$
Let
$$
   \widetilde X_{\Gamma,l}:=(M^2\times\cdots\times M^2\times \widetilde
   M^2\times M^2\times\cdots\times M^2)\times M^s
$$
be the (oriented real) blow-up of $X_\Gamma$ along 
$\Delta_2^l$, and
$$
   \widetilde X_\Gamma:=(\widetilde M^2)^e\times M^s
$$
the blow-up of $X_\Gamma$ along $\Delta_{2\Gamma}$. This is a manifold
with corners with smooth blow-down map 
$$
   \pi:\widetilde X_\Gamma\longrightarrow X_\Gamma.
$$

\begin{remark}\label{rem:edgeliftbup}
Consider a change of extension of a labelling of 
$\Gamma$ affecting the edge order only. That is we precompose the edge 
order by some
permutation $\eta$ of the set 
Assume that we changed the edge 
$\{1,\dots,3k\}$. The induced map
$M^\eta:X_\Gamma\longrightarrow X_\Gamma$ permutes the $M^2$ factors
in $X_\Gamma=(M^2)^e\times M^s$ (renumbering the edges) or flips the
two $M$ factors within one $M^2$ factor (reorienting an edge). It
lifts to the blow-up to give a map  
$$ 
   \wt M^\eta:\wt X_\Gamma\longrightarrow\wt X_\Gamma.
$$
\end{remark}

\begin{definition}\label{def:CGamma}
The {\em compactified configuration space} of $M$ corresponding to the graph $\Gamma$ is the
proper transform of the vertex diagonal in the blow-up of the edge diagonal,
$$
   \CC_\Gamma := PT(\Delta_{3\Gamma}) := \text{closure of
   }\pi^{-1}(\Delta_{3\Gamma}\setminus\Delta_{2\Gamma}) \text{ in } \widetilde X_\Gamma.
$$
\end{definition}

 The following properties on $\CC_\Gamma$ are proved
in~\cite{Cieliebak-Volkov-stringtop}. We recall the corresponding
notions and results in Appendix~\ref{sec:blow-up}, see
in particular Proposition~\ref{prop:Stokes} and
equation~\eqref{eq:descr-prim}. 

\begin{thm}
\label{thm:config-space}
(a) The space $\CC_\Gamma$ is the disjoint union
$$
  \CC_\Gamma = \CC_\Gamma^0\amalg \p^{\rm q-reg}\CC_\Gamma \amalg \p_{\geq 2}\CC_\Gamma
$$
of its interior $\CC_\Gamma^0=\pi^{-1}(\Delta_{3\Gamma}\setminus\Delta_{2\Gamma})$,
its codimension $1$ (so-called quasi-regular) boundary $\p^{\rm q-reg}\CC_\Gamma$, and
the part $\p_{\geq 2}\CC_\Gamma$ of codimension at least two. 
Stokes' theorem holds on $\CC_\Gamma$ in the sense that for each
$\gamma\in\Om^*(\wt X_\Gamma)$ the integral
$$
  \int_{\CC_\Gamma}\gamma := \int_{\CC_\Gamma^0}\gamma
$$
exists, and for each
$\beta\in\Om^*(\wt X_\Gamma)$ the following two integrals exist and
are equal:
$$
   \int_{\CC_\Gamma}d\beta = \int_{\p^{\rm q-reg}\CC_\Gamma}\beta.
$$
(b) The quasi-regular boundary
is the disjoint union
$$
   \p^{\rm q-reg}\CC_\Gamma = \p^{\rm main}\CC_\Gamma \amalg \p^{\rm hidden}\CC_\Gamma
$$
of its {\em main part} 
$$
   \p^{\rm main}\CC_\Gamma := \coprod_{l\in\Edge(\Gamma)}\p^l\CC_\Gamma, \qquad
   \p^l\CC_\Gamma := \p\CC_\Gamma\cap \pi^{-1}(\Delta_2^{\{l\}})
$$
and its {\em hidden part} $\p^{\rm hidden}\CC_\Gamma$. 
The component 
$\p^l\CC_\Gamma
=\p\widetilde X_{\Gamma,l}|_{\Delta_{3\Gamma}\cap\Delta_2^{\{l\}}}$
of the main part is the restriction of the $S^{n-1}$-bundle $\widetilde
X_{\Gamma,l}\to \Delta_2^l$ to $\Delta_{3\Gamma}\cap\Delta_2^{\{l\}}$.
\end{thm}

A fundamental difficulty arises from the fact that the hidden
part $\p^{\rm hidden}\CC_\Gamma$ is in general nonempty and may lead
to unwanted terms on the right hand side of Stokes' theorem. 
This difficulty is resolved in Theorem~\ref{thm:key-vanish} below.

\begin{remark}
Let us point out one apparent difficulty that arises with
$\p^l\CC_\Gamma$ if there are two edges $l,k$ with the same endpoints. 
Then the set $\Delta_2^{\{l\}}$ and thus
$\p^l\CC_\Gamma$ is empty. To deal with this, simply note that all
our discussions carry over to this case.
For example, equation~\eqref{eq:signinvol} holds true even in
this case, with both sides being empty sets, and any integral over the
empty set is zero. 
\end{remark}

\subsection{Duality}\label{ss:duality}

In this subsection we apply the duality operation
from~\S\ref{ss:graph-duality} to configuration spaces. 
Recall the setup: $(\Gamma,l)$ is an o-marked trivalent graph with an
extended labelling, 
$l=(u,v)$ is the oriented marked edge, 
$(z,w,u)$ and $(v,x,y)$ are the
two vertices connected by $l$, and $(I(\Gamma),I(l))$ is the new
o-marked extended labelled graph defined in Figure~\ref{fig:duality}. 
Recall from~\eqref{eq:dual-r} that the reordering maps $\bar
R_\Gamma$ and $\bar R_{I(\Gamma)}$ are related by 
\begin{equation}\label{eq:dual-r2}
   \bar R_{I(\Gamma)}^{-1}\circ \bar R_\Gamma = 
   O_e^{-1}\circ\sigma_4\circ O_e,
\end{equation}
where $\sigma_4$ is the cyclic permutation $(z,w,x,y)\mapsto
(y,z,w,x)$ leaving all other flags fixed. 
Applying the functor $S\mapsto M^S$ from~\S\ref{ss:config}, this
induces the following relation on diffeomorphisms of $M^{3k}$:
\begin{equation}\label{eq:dual-ph}
  \Phi := R_\Gamma\circ R_{I(\Gamma)}^{-1}
  = M^{\bar R_{I(\Gamma)}^{-1}\circ \bar R_\Gamma}
   = M^{O_e^{-1}\circ\sigma_4\circ O_e}.
\end{equation}
Since $\sigma_4$ fixes the flags $u,v$, the map $\Phi$ acts as the
identity on the corresponding $M^2$ factor of $X_\Gamma$, and its
blow-up $\wt\Phi$ acts as the identity on the $\widetilde 
M^2$ factor of $\wt X_{\Gamma,l}$. Therefore, $\wt\Phi$
restricts to the identity map between the boundary loci where the
variables corresponding to the flags $u,v,x,y,z,w$ all coincide,
\begin{equation}\label{eq:or-rev}
   \wt\Phi|_{\p^{I(l)}\CC_{I(\Gamma)}}=\id:\p^{I(l)}\CC_{I(\Gamma)}\longrightarrow \p^l\CC_\Gamma.
\end{equation}
On the other hand, applying equation~\eqref{eq:Delta3Gamma} to
$\Gamma$ and $I(\Gamma)$ yields the following equalities of oriented
manifolds:  
\begin{align*}
  \Delta_{3\Gamma}
  &= (-1)^{\bar R_\Gamma+(n-1)\eta_3(\Gamma)}R_\Gamma(\Delta_3^k) \cr
  &= (-1)^{\bar R_{I(\Gamma)}^{-1}\circ\bar R_\Gamma+(n-1)(\eta_3(\Gamma)+\eta_3(I(\Gamma)))}R_\Gamma\circ R_{I(\Gamma)}^{-1}(\Delta_{3I(\Gamma)}).
\end{align*}
Since $\eta_3(\Gamma)=\eta_3(I(\Gamma))$ by~\eqref{eq:eta3-duality}
and $\sigma_4$ is an odd permutation, in view of~\eqref{eq:dual-r2}
the sign exponent in the last displayed equation is odd, so
using~\eqref{eq:dual-ph} we get 
$$
  \Delta_{3\Gamma} = -\Phi(\Delta_{3I(\Gamma)}).
$$
This shows that the map $\Phi|_{\Delta_{3\Gamma}}:\Delta_{3\Gamma}\to\Delta_{3I(\Gamma)}$ 
is orientation reversing, hence the map in~\eqref{eq:or-rev} is
orientation reversing as well. In other words, even though the 
boundary loci coincide as manifolds, their boundary orientations differ:
\begin{equation}\label{eq:signinvol}
   \p^{I(l)}\CC_{I(\Gamma)}=-\p^l\CC_\Gamma.
\end{equation}
This relation will play a crucial role in the proof of the
Maurer-Cartan equation in~\S\ref{sec:proofMC}.

\subsection{Integrals over configuration spaces}\label{sec:opforms}



We retain the notation from above, so $\Gamma$ is a trivalent ribbon graph
with an extended labelling and $M$ is a closed oriented $n$-manifold.
Let $\wt G\in\Om^{n-1}(\wt M^2)$ be a propagator as in
\S\ref{sec:prop} and $G$ its pushforward to $M^2$ (singular
along the diagonal). In addition, let
$$
   \alpha = \alpha_1\otimes\dots\otimes\alpha_s 
$$
be a decomposable tensor of differential forms $\alpha_j\in\Om^*(M)$
corresponding to the leaves of $\Gamma$. We call such a decomposable 
$\alpha$ {\em adapted to $\Gamma$}. 
To this data we associate the following differential forms, where for the
last two we assume in addition that a marked oriented edge $l$ 
has been chosen:
\begin{itemize}
\item $\cross(\alpha):=\alpha_1\times\dots\times \alpha_s$ on $M^s$;
\item $G^e:=G\times\dots\times G$ on $(M^2)^e$;
\item $\wt G^e:=\wt G\times\dots\times\wt G$ on $(\widetilde M^2)^e$;
\item $G^e(\alpha):=G^e\times\cross(\alpha)$ on $X_\Gamma$; 
\item $\wt G^e(\alpha):=\wt G^e\times\cross(\alpha)$ on $\wt X_\Gamma$; 
\item $G^e_l:=G\times\dots\times dG\times\dots\times G$ on
  $(M^2)^e$ (with $dG$ at the position $n(l)$ of $l$);
\item $\wt G^e_l:=\wt G\times\dots\times dG\times\dots\times \wt G$ on
  $(\wt M^2)^e$ (with $dG$ at the position $n(l)$ of $l$);
\item $G^e_l(\alpha):=G^e_l\times\cross(\alpha)$ on $X_\Gamma$;
\item $\wt G^e_l(\alpha):=G^e_l\times\cross(\alpha)$ on $\wt X_\Gamma$.
\end{itemize}
If $\alpha$ is not adapted to $\Gamma$ we set $G^e(\alpha):=0$,
and similarly for all other quantities.
Note that the first five forms do not depend on $\Gamma$, and
the last four forms depend on $\Gamma$ only through the position 
$n(l)$ of the edge $l$ in the numbering of edges.
Consider the following integrals, where the equalities follow directly
from the definitions of $\Delta_{3\Gamma}$, $\Delta_3^k$, and
$\CC_\Gamma$:
\begin{align*}
  I_\Gamma(\alpha)
  &:= \int_{\CC_\Gamma}\wt G^e(\alpha)
  = \int_{\Delta_{3\Gamma}}G^e(\alpha)
  = (-1)^{\bar R_\Gamma+(n-1)\eta_3(\Gamma)}\int_{\Delta_3^k}R_\Gamma^*G^e(\alpha), \cr 
  I_{\Gamma,l}(\alpha)
  &:= (-1)^{(n-1)(n(l)-1)}\int_{\Delta_{3\Gamma}}G^e_l(\alpha) 
  = (-1)^{\bar R_\Gamma+(n-1)(\eta_3(\Gamma)+n(l)-1)}\int_{\Delta_3^k}R_\Gamma^*G^e_l(\alpha), \cr
  I_{\Gamma-l}(\alpha) &:= \int_{\p^l\CC_\Gamma}\wt G^e(\alpha).
\end{align*}
Here all the forms with $\sim$ live on the blow-up and are
smooth, hence integrable 
by Theorem~\ref{thm:config-space}(b). The forms without $\sim$ are singular, but
their integrals still exist in the Lebesgue sense.  
The integrals not involving the marked edge $l$ will enter the
definition of the Maurer-Cartan element in
Theorem~\ref{thm:existence-intro}, while the integrals involving $l$
will be used in the proof of the Maurer-Cartan relation. 

The following observation will prove useful. Let $\Gamma_1$,
$\Gamma_2$ be two extended labelled connected graphs and
$\Gamma:=\Gamma_1\amalg\Gamma_2$ with the canonical disjoint union
labelling. 
Then
\begin{equation}\label{eq:2comp}
I_\Gamma(\alpha_1\alpha_2)=(-1)^{(n-1)\alpha_1e_2}I_{\Gamma_1}(\alpha_1)I_{\Gamma_2}(\alpha_2),
\end{equation}
where the sign exponent results from pulling the decomposable tensor
$\alpha_1$ associated to $\Gamma_1$ to the right past the $e_2$
propagators associated to $\Gamma_2$ ($e_2$ being the number of edges
of $\Gamma_2$), whose total degree is $(n-1)e_2$. 

\begin{lemma}\label{lm:exten}
The integrals $I_\Gamma(\alpha)$, $I_{\Gamma,l}(\alpha)$ and $I_{\Gamma-l}(\alpha)$
do not depend on the extension of the labelling of $\Gamma$.
%
\end{lemma}

\begin{proof}
%
Independence of these integrals on items (iii) and (iv) in
Definition~\ref{def:labelling} follows from
the fact (Lemma~\ref{lem:Delta3Gamma}) that the oriented manifold
$\Delta_{3\Gamma}$ does not depend on the vertex order, hence neither
do $\CC_\Gamma$ and $\p^l\CC_\Gamma$.

A change in items (v) and (vi)
changes the edge order of flags by precomposition with a permutation 
$\tau$. We denote the graph with the new extended labeling by 
$\Gamma'$, the new number of the marked edge by $n'(l)$, 
and set $\Delta\eta_3:=\eta_3(\Gamma')-\eta_3(\Gamma)$.
Since the new reordering maps are $\bar R_{\Gamma'}=\bar
R_\Gamma\tau$ and $R_{\Gamma'}=M^\tau R_\Gamma$, the new triple diagonal is
\begin{align}\label{eq:diagtrans}
  \Delta_{3\Gamma'}
  = (-1)^{\tau+\bar R_{\Gamma}+(n-1)\eta_3(\Gamma')}M^\tau R_\Gamma(\Delta_3^k)
  = (-1)^{\tau+(n-1)\Delta\eta_3}M^\tau(\Delta_{3\Gamma}).
\end{align}
Analogous transformations hold for $\CC_\Gamma$ and $\p^l\CC_\Gamma$:
\begin{align}\label{eq:diagtrans2}
  \CC_{\Gamma'} = (-1)^{\tau+(n-1)\Delta\eta_3}\wt M^\tau(C_\Gamma),\qquad
  \p^l\CC_{\Gamma'} = (-1)^{\tau+(n-1)\Delta\eta_3}\wt M^\tau(\p^lC_\Gamma),
\end{align}
where $\wt M^\tau$ is the lift of $M^\tau$ to blow-ups, see Remark~\ref{rem:edgeliftbup}.
Now we need to transform the integrands. We claim that 
\begin{equation}\label{eq:formstrans}
\begin{aligned}
  (M^\tau)^*G^e(\alpha) &= (-1)^{\tau+(n-1)\Delta\eta_3}G^e(\alpha),\cr
  (M^\tau)^*G^e_l(\alpha)
  &= (-1)^{\tau+(n-1)\Delta\eta_3+(n-1)(n'(l)-n(l))}G^e_l(\alpha).
\end{aligned}
\end{equation}
To see this we distinguish two cases.
Assume first that $\tau$ corresponds to swapping the order of two adjacent
edges. Then the sign exponents are $\tau=0$ and $\Delta\eta_3=1$, and
$(M^\tau)^*G^e(\alpha)=(-1)^{n-1}G^e(\alpha)$ because $G$ has degree
$n-1$. This yields the first equation. For  
the second equation assume first that $l$ is not one of the 
edges swapped. Then $(M^\tau)^*G^e_l(\alpha)=(-1)^{n-1}G^e_l(\alpha)$
and $n(l)=n'(l)$ giving us the second equation
in~\eqref{eq:formstrans}. If now $l$ is one of the edges swapped, then  
$(M^\tau)^*G^e_l(\alpha)=G^e_l(\alpha)$ (because $G$
and $dG$ have different parities)
and $n'(l)-n(l)=1$. This again gives us the second 
equation in \eqref{eq:formstrans}.


Now assume that $\tau$ corresponds to flipping the orientation of an edge.
Then the sign exponents are $\tau=1=\Delta\eta_3$, and
$(M^\tau)^*G^e(\alpha)=(-1)^nG^e(\alpha)$ by equation~\eqref{eq:Green-symmetry}. This gives us the first equation.
For the second equation note that 
$(M^\tau)^*G^e_l(\alpha)=(-1)^nG^e_l(\alpha)$ by
equation~\eqref{eq:Green-symmetry}, no matter whether
the flipped edge is $l$ or not.

The desired independence of $I_\Gamma(\alpha)$, $I_{\Gamma,l}(\alpha)$
and $I_{\Gamma-l}(\alpha)$ on the extension of the labelling  
follows from~\eqref{eq:diagtrans},~\eqref{eq:diagtrans2}
and~\eqref{eq:formstrans} by invariance of integration, noting the
prefactor $(-1)^{(n-1)(n(l)-1)}$ in the definition of $I_{\Gamma,l}(\alpha)$.
\end{proof}

As a conclusion of this lemma, the quantity $I_\Gamma(\alpha)$ is
defined for a labelled graph, and the quantities
$I_{\Gamma,l}(\alpha)$ and $I_{\Gamma-l}(\alpha)$ are defined for a marked
labelled graph.

The following key technical result is proved
in~\cite{Cieliebak-Volkov-stringtop}. We recall its derivation in
Appendix~\ref{sec:blow-up}, see 
Corollary~\ref{cor:Galpha-invar}.

\begin{thm}\label{thm:key-vanish}
For any labelled trivalent ribbon graph $\Gamma$ and any adapted
collection $\alpha$ of differential forms we have 
$$
  \int_{\p^{\rm hidden}\CC_\Gamma}\wt G^e(\alpha)=0.
$$ 
\end{thm}


Combined with Stokes' theorem this has the following consequence,
which will be the crucial step in our later proof of the Maurer-Cartan
relation. 

\begin{cor}\label{cor:fulldiff}
For any labelled trivalent ribbon graph $\Gamma$ and any adapted
collection $\alpha$ of {\em closed} differential forms we have 
$$
  \sum_{l\in\Edge(\Gamma)}I_{\Gamma,l}(\alpha) =
  \sum_{l\in\Edge(\Gamma)}I_{\Gamma-l}(\alpha).
$$
\end{cor}

\begin{proof}
We pick an extension of the labelling of $\Gamma$ in which an edge $l$
has position $n(l)$. From the definition of $G^e(\alpha)$,
$G^e_l(\alpha)$ and closedness of $\alpha$ we get
$$
  dG^e(\alpha) = dG^e\times\cross(\alpha) =
  \sum_{l\in\Edge(\Gamma)}(-1)^{(n-1)(n(l)-1)}G^e_l(\alpha). 
$$
Using this, we compute
\begin{align*}
  \sum_{l\in\Edge(\Gamma)}I_{\Gamma,l}(\alpha)
  &\stackrel{(1)}= \sum_{l\in\Edge(\Gamma)}(-1)^{(n-1)(n(l)-1)}
  \int_{\Delta_{3\Gamma}}G^e_l(\alpha) \cr 
  &\stackrel{(2)}= \int_{\Delta_{3\Gamma}}dG^e(\alpha) 
  \stackrel{(3)}= \int_{\CC_\Gamma}d\wt G^e(\alpha) \cr
  &\stackrel{(4)}= \int_{\p\CC_\Gamma}\wt G^e(\alpha) 
  \stackrel{(5)}= \int_{\p^{\rm main}\CC_\Gamma}\wt G^e(\alpha) \cr
  &\stackrel{(6)}= \sum_{l\in\Edge(\Gamma)}\int_{\p^l\CC_\Gamma}\wt G^e(\alpha)
  \stackrel{(7)}= \sum_{l\in\Edge(\Gamma)}I_{\Gamma-l}(\alpha).
 \end{align*}
Here the equalities (1), (3), (6) and (7) are just the corresponding
definitions, (2) uses the equation above, (4) is Stokes' theorem in
the form of Theorem~\ref{thm:config-space}(b), and (5) follows from
Theorem~\ref{thm:key-vanish}. 
\end{proof}

{\bf Permutation action. }
Our next goal is to describe the behaviour of the integrals
$I_\Gamma(\alpha)$ and $I_{\Gamma,l}(\alpha)$ under change of
labelling of the graph $\Gamma$. We abbreviate
$$
  \Om := \Om^*(M),\qquad 
  \bs:=(s_1,\dots,s_\ell),\qquad
  \Om(\bs):=\bigotimes_{b=1}^\ell\Om[1]^{\otimes s_b}[3-n].
$$
Let $s=s_1+\cdots+s_\ell$ be the number of leaves of 
$\Gamma$. Note that abstractly 
$$
  \Om(\bs)\cong \Om[1]^{\otimes s}[(3-n)\ell],
$$
but it is important to keep in mind the additional structure induced by $\bs$.
Any decomposable 
$$
  \alpha=\alpha^1\otimes\dots\otimes\alpha^{\ell}\in \Om(\bs)
$$ 
with
$$
  \alpha^b=\alpha^b_1\otimes\dots\otimes\alpha^b_{s_b}\in \Om[1]^{\otimes s_b}
$$
can also be written in the form
$$
  \alpha=\alpha_1\otimes\dots\otimes\alpha_s\in \Om[1]^{\otimes s}.
$$
We recall the operation
$$
  P:\Om^{\otimes s} \to \Om^{\otimes s},\qquad 
  P(\alpha) = (-1)^{P(\alpha)}\alpha,\quad\text{with}\quad P(\alpha) = \sum_{j=1}^s(s-j)\deg\alpha_j
$$
defined by equation~\eqref{eq:defP}. By analogy we introduce the operation
$$
  P_b:\bigotimes_{b=1}^\ell\Om[1]^{\otimes s_b}\to \bigotimes_{b=1}^\ell\Om[1]^{\otimes s_b},\qquad 
  P_b(\alpha) := (-1)^{(3-n)P_b(\alpha)}\alpha
$$
(the subscript ``$b$'' stands for ``boundary'') with the sign exponent
$$
  P_b(\alpha) := \sum_{b=1}^\ell(\ell-b)|\alpha^b|,
$$
where $|\alpha^b|$ is the total shifted degree of
$\alpha^b\in\Om[1]^{\otimes s_b}$ corresponding to the $b$-th boundary component.  

\begin{definition}\label{def:Sls}
We denote by $S(\bs)$
the subgroup of the symmetric group $S_s$ on $s=s_1+\cdots+s_\ell$ elements 
consisting of compositions of permutations of the following two types: 
permutations of the $\ell$ boundary components, and cyclic permutations of the
$s_b$ leaves on one boundary component. 
\end{definition}

Let us discuss how $\eta\in S(\bs)$ acts on the tensor product            
$\Om(\bs)$.
Recall from \S\ref{ss:gradedvect} the naive action permuting factors without signs by
$$
   (\alpha_1\otimes\cdots\otimes\alpha_s)\eta = \alpha_{\eta(1)}\otimes\cdots\otimes\alpha_{\eta(s)}.
$$
For $\eta\in S(\bs)$ let $\eta_b$ be the corresponding permutation
of boundary components and set
$$
  (\bs\eta_b)_j:=(\bs)_{\eta_b(j)}.
$$
We define the analytic and algebraic actions
$$
  \Om[1]^{\otimes s}[(3-n)\ell]\cong\Om(\bs)\longrightarrow
  \Om(\bs\eta_b)\cong\Om[1]^{\otimes s}[(3-n)\ell]
$$
as follows:
$$
   \eta_{an}(\alpha):=(-1)^{\eta+(n-1)\eta_b+\eta_{an}(\alpha)}\eta(\alpha),\qquad
   \eta_{alg}(\alpha):=(-1)^{\eta_{alg}(\alpha)}\eta(\alpha).
$$
Here $(-1)^\eta$ is the sign of the permutation $\eta$, 
$\eta_{an}(\alpha)$ is the sign exponent for permuting the forms
$\alpha_j$ with their degrees in $\Om$, and $\eta_{alg}(\alpha)$ is the
sign exponent for cyclically permuting the forms on each boundary
component with their shifted degrees $|\alpha_j|=\deg(\alpha_j)-1$, and
then permuting the boundary words of forms according to $\eta_b$ with
degrees additionally shifted by $3-n\equiv n-1$.  
The two actions are related by the commuting diagram
\begin{equation}\label{eq:alg-ana-action2}
\xymatrix{
  \Om^{\otimes s} \ar[d]_{P_b\circ P} \ar[r]^{\eta_{an}} &\Om^{\otimes s} \ar[d]^{P_b\circ P} \\
  \Om[1]^{\otimes s}[(3-n)\ell] \ar[r]^{\eta_{alg}} &\Om[1]^{\otimes s}[(3-n)\ell].
}
\end{equation}
The above actions $\eta_{an}$ and $\eta_{alg}$ are not to be confused
with the corresponding ones on the total tensor product defined
in~\S\ref{ss:gradedvect}. In the present context the structure on the
total tensor product induced by $\bs$ plays an essential role.

\begin{lemma}\label{lem:act}
Let $M^\eta:M^s\to M^s$ be the diffeomorphism induced by $\eta\in S(\bs)$
and $\alpha\in\Om^*(M)^{\otimes s}$. Then
$$
   (M^\eta)^*\cross(\alpha) = (-1)^{\eta+(n-1)\eta_b} \cross\bigl(\eta_{an}^{-1}(\alpha)\bigr).
$$
\end{lemma}

\begin{proof}
With the projection $\pi_j:M^s\to M$ onto the $j$-th factor we have
$\pi_j\circ M^\eta=\pi_{\eta(j)}$, and therefore
\begin{align*}
   (M^\eta)^*\cross(\alpha)
   &= (M^\eta)^*\pi_1^*\alpha_1\wedge\cdots\wedge (M^\eta)^*\pi_s^*\alpha_s \cr 
   &= \pi_{\eta(1)}^*\alpha_1\wedge\cdots\wedge \pi_{\eta(s)}^*\alpha_s \cr 
   &= (-1)^{\eta_{an}(\alpha)}\pi_1^*\alpha_{\eta^{-1}(1)}\wedge\cdots\wedge \pi_s^*\alpha_{\eta^{-1}(s)} \cr 
   &= (-1)^{\eta+(n-1)\eta_b}\cross\bigl(\eta_{an}^{-1}(\alpha)\bigr).
\end{align*}
\end{proof}

\begin{lemma}\label{lem:invI}
Let $\Gamma\in \RR_{\ell,g}$ be a labelled graph with $s$ exterior flags, and $l$ an
edge of $\Gamma$. Let
$\eta\in S(\bs)$ describe a change of labelling (altering items (i) and
(ii) in Definition~\ref{def:labelling}) and denote by $\Gamma\eta$
the graph with the new labelling. Then
$$
 I_{\Gamma\eta}(\alpha)=I_\Gamma(\eta_{an}^{-1}\alpha),\qquad 
 I_{\Gamma\eta,l}(\alpha)=I_{\Gamma,l}(\eta_{an}^{-1}\alpha).
$$
\end{lemma}

\begin{proof}
Choose an extended labelling for $\Gamma$. This gives rise to ordering
maps $O_v$ and $O_e$, and the change of labelling corresponds to
precomposition of $O_e$ by $\eta$. 
Let $\eta_b$ denote the permutation 
of boundary components corresponding 
to the change in item (i) in Definition~\ref{def:labelling}. Observe that
$$
\bar R_{\Gamma\eta}=\eta+\bar R_\Gamma,\quad
\eta_3(\Gamma\eta)=\eta_b+\eta_3(\Gamma).
$$
Since $M^\eta$ acts only on the
part of $G^e(\alpha)$ corresponding to $\alpha$, Lemma~\ref{lem:act} yields
$$
   R_{\Gamma\eta}^*G^e(\alpha) 
   = R_\Gamma^*(M^\eta)^*G^e(\alpha) 
   = R_\Gamma^*\bigl(G^e\times(M^\eta)^*\cross(\alpha)\bigr)
   = (-1)^{\eta+(n-1)\eta_b} R_\Gamma^*G^e(\eta_{an}^{-1}\alpha). 
$$
It follows that
\begin{align*}
   I_{\Gamma\eta}(\alpha) 
   &= (-1)^{\bar R_{\Gamma\eta}+(n-1)\eta_3(\Gamma\eta)}\int_{\Delta_3^k}R_{\Gamma\eta}^*G^e(\alpha)\cr
   &= (-1)^{\eta+(n-1)\eta_b+\bar R_\Gamma+(n-1)\eta_3(\Gamma)}
   \int_{\Delta_3^k}R_{\Gamma\eta}^*G^e(\alpha)\cr 
   &= (-1)^{\bar R_\Gamma+(n-1)\eta_3(\Gamma)}\int_{\Delta_3^k}R_\Gamma^*G^e(\eta_{an}^{-1}\alpha)
   = I_\Gamma(\eta_{an}^{-1}\alpha).
\end{align*}
The proof of the second equation is analogous.
\end{proof}

We define the sign exponents
\begin{equation}\label{eq:seGamma}
  se_\Gamma:=(\ell+1)(s+1),\qquad \wt{se}_\Gamma:=\ell(s+1)+1.
\end{equation}
Note the following two properties of 
$se_\Gamma$. For any two graphs $\Gamma_1$ and $\Gamma_2$ we have
\begin{equation}\label{eq:seGamma-sum}
se_{\Gamma_1}+se_{\Gamma_2}+se_{\Gamma_1\amalg\Gamma_2}=s_1\ell_2+s_2\ell_1+1,
\end{equation}
and for any tree $\Gamma$ we have
\begin{equation}\label{eq:seGammatree}
se_{\Gamma}=0.
\end{equation}
Note that the choice of $se_\Gamma$ with these properties is nonunique. 
We could add any linear function of the numerical invariants
of a graph that vanishes on trees, e.g.~the genus or the number of
connected components minus the number of boundary components. 
For $\Gamma$ and $\alpha$ as above we define
\begin{equation}\label{eq:mGamma}
\begin{aligned}
  \fm_\Gamma(\alpha) &:= (-1)^{s_\Gamma(\alpha)}I_{\Gamma}(\alpha), \cr
   \n_{\Gamma,l}(\alpha) &:= (-1)^{\tilde s_\Gamma(\alpha)}I_{\Gamma,l}(\alpha),\cr
   \n_{\Gamma-l}(\alpha) &:= (-1)^{\tilde s_\Gamma(\alpha)}I_{\Gamma-l}(\alpha).
\end{aligned}
\end{equation}
with the sign exponents
\begin{align}
  s_\Gamma(\alpha)
  &:= n\ell+s(s+1)/2+P(\alpha)+(n-1)(se_\Gamma+P_b(\alpha)), \label{eq:sGamma} \\
  \tilde s_\Gamma(\alpha)
  &:=
  n(\ell-1)+1+(s+1)(s+2)/2+P(\alpha)+(n-1)(\wt{se}_\Gamma+P_b(\alpha)).
  \label{eq:sGammatilde}
\end{align}
Now Lemma~\ref{lem:invI} together with the commuting
diagram~\eqref{eq:alg-ana-action2} gives  

\begin{lemma}\label{lem:action}
For $\Gamma$, $\alpha$ and $\eta$ as in Lemma~\ref{lem:invI} we have
$$
   \fm_{\Gamma\eta}(\alpha)=\fm_{\Gamma}(\eta_{alg}^{-1}\alpha),\qquad
   \n_{\Gamma\eta,l}(\alpha)=\n_{\Gamma,l}(\eta_{alg}^{-1}\alpha).
$$
\hfill$\square$
\end{lemma}

From now on we will always mean the algebraic action when we say
``action'' and write $\eta$ and $\sigma$ instead of $\eta_{alg}$ and
$\sigma_{alg}$ to ease notation. 
Recall that according to our convention the action of $S(\bs)$ on
functionals is defined by dualizing the one on forms. 
Then the identities in Lemma~\ref{lem:action} take the form
\begin{equation}\label{eq:action}
   \fm_{\Gamma\eta} = \eta^{-1}\fm_{\Gamma},\qquad
   \n_{\Gamma\eta,l} = \eta^{-1}\n_{\Gamma,l}.
\end{equation}


\section{The Maurer-Cartan element}\label{sec:proofMC}


In this section we prove Theorem~\ref{thm:existence-intro} from the
Introduction. We begin in~\S\ref{ss:defMC} by defining the
Maurer-Cartan element $\fm:=\{\fm_{\ell,g}\}_{\ell\ge 1,g\ge 0}$, using
the integrals over configuration spaces from~\S\ref{sec:opforms}.
Subsections~\S\ref{ss:identities} -- \S\ref{ss:maincompute} are
devoted to the proof of the Maurer-Cartan equation for $\fm$. 
Finally, in~\S\ref{ss:cyc-coh} we prove that the $\fm$-twisted homology equals
the Connes cyclic cohomology of the de Rham complex. 

\subsection{Definition of the Maurer-Cartan element}\label{ss:defMC}

Let $M$ be a closed oriented manifold of dimension $n$, and $(\Om^*(M),d,(\cdot,\cdot))$
its de Rham algebra with the intersection pairing~\eqref{eq:defineintpair}.
Recall that this pairing is nondegenerate but not perfect.
We fix a complementary subspace $\HH$ to $\im d$ in $\ker d$,
i.e.~such that
$$
   \ker d = \im d\oplus\HH. 
$$
The pairing $(\cdot,\cdot)$ restricts to $\HH$ as a perfect pairing,
so we get a cyclic cochain complex $(\HH, d=0, (\cdot,\cdot))$.
Proposition~\ref{prop:canondIBL} associates to this cyclic complex a
canonical dIBL-algebra (actually an IBL-algebra)  
\begin{equation}\label{eq:deRhamcanon}
   \dIBL(\HH) = \Bigl(C:=(B^{{\text{\rm cyc}}*}\HH)[2-n],\fp_{1,1,0}=0,\,\fp_{1,2,0},\,\fp_{2,1,0}\Bigr).
\end{equation}
Given a number of boundary components $\ell\geq 1$ and a genus $g\geq 0$,  
recall from Definitions~\ref{def:R} and~\ref{def:Rlgs} the set $\RR_{\ell,g}$ 
and its finite subsets $\RR_{\ell,g}(s_1,\dots,s_\ell)$.
We define
\begin{equation}\label{eq:mlg}
   \fm_{\ell,g} := \frac{1}{\ell!}\sum_{\Gamma\in \RR_{\ell,g}}\fm_{\Gamma}\in 
   (B\HH[3-n]^{\otimes\ell})^*,
\end{equation}
with $\fm_{\Gamma}\in (B\HH[3-n]^{\otimes\ell})^*$ defined by
equation~\eqref{eq:mGamma} as an integral over the configuration space
associated to $\Gamma$. Here to define the $\fm_\Gamma$ we fix a (not
necessarily special) propagator $\wt G$ provided by Proposition~\ref{prop:existsprop}. 

Let us verify that $\fm_{g,\ell}$ defines an element in the completed symmetric
product 
\begin{equation*}
    \wh E_\ell C = (B^{\text{\rm cyc}}\HH[3-n]^{\otimes\ell})^*/S_\ell
\end{equation*}
(see~\ref{eq:whEk}). We first check that $\fm_{\ell,g}$ lies in $(B^\cyc\HH[3-n]^{\otimes\ell})^*$.
For this, consider for fixed $s_1,\dots,s_\ell$ the action of the group
$G=\Z_{s_1}\times\cdots\times\Z_{s_\ell}$ on $\RR_{\ell,g}(s_1,\dots,s_\ell)$ 
rotating the numberings of the flags on each boundary component.
For $\Gamma\in \RR_{\ell,g}(s_1,\dots,s_\ell)$ let $G_{\Gamma}$ be its
isotropy group and $G\cdot\Gamma$ its orbit for the $G$-action. 
Then the contribution of the orbit of $\Gamma$ to the sum in~\eqref{eq:mlg} is
$$
   \frac{1}{\ell!|G_{\Gamma}|}\sum_{\eta\in G}\fm_{\eta\Gamma} =
   \frac{1}{\ell!|G_{\Gamma}|}\fm_{\Gamma}\circ N^{\otimes\ell},
$$
where we have used
Lemma~\ref{lem:action} to trade all possible
rotations of a boundary component for the symmetrization operator $N$. 
Summing over all orbits, this shows that $\fm_{\ell,g}$ lies in 
$(B^\cyc\HH[3-n]^{\otimes\ell})^*$. Projecting
$\fm_{\ell,g}$ to the quotient under permutations of the $\ell$ tensor
factors thus gives an element in $\wh E_\ell C$. 
%

The following subsections are devoted to the proof that the collection 
$$
  \fm:=\{\fm_{\ell,g}\}_{\ell\ge 1,g\ge 0}
$$
defines a Maurer-Cartan element for $\dIBL(\HH)$.
The main part of the proof consists in showing that $\fm$ satisfies
the Maurer-Cartan equation~\eqref{eq:MCdIBL}, which in this case reduces to
\begin{equation}\label{eq:MCIBL}
  \wh\fp_{2,1,0}\fm+\frac{1}{2}\wh\fp_{2,1,0}(\fm\otimes\fm)|_{\conn}+\wh\fp_{1,2,0}\fm=0,
\end{equation}
because the differential on $\HH$ is zero. 

\subsection{Notation and useful identities}\label{ss:identities}

Throughout this section, by a ``graph'' we will always mean a
(possibly disconnected) trivalent labelled ribbon graph $\Gamma\in
\RR_{\ell,g}$ with a chosen extension of the labelling, and $\alpha$ 
will be a decomposable tensor of differential forms adapted to $\Gamma$.  
By Lemma~\ref{lm:exten} the expressions $I_\Gamma(\alpha)$
and $\fm_\Gamma(\alpha)$ do not depend on the extension of the labelling,
but the extension will enter various sign computations below.

Recall the notation related to a graph $\Gamma$:
$s$ is the number of flags, $k$ is the number of vertices,
$e$ is the number of edges, $\ell$ is the number of boundary
components, and $g$ is the genus of the associated surface.
Counting flags in two ways yields the identity
\begin{equation}\label{eq:flag}
  3k = 2e+s,
\end{equation}
and computing the Euler characteristic in two ways yields
\begin{equation}\label{eq:Euler}
  k-e = 2-2g-\ell.
\end{equation}
Recall that $n$ is the dimension of the manifold $M$. The condition that
the total degree of the form $G^e(\alpha)$ equals the dimension of the
configuration space $\CC_\Gamma$ reads
\begin{equation}\label{eq:deg-dim}
  nk = \deg\alpha + (n-1)e.
\end{equation}
For the degree discussions below we compute
\begin{equation}\label{eq:deg-s}
\begin{aligned}
  \deg\alpha-s &= (nk-(n-1)e)-(3k-2e)
  =(n-3)(k-e) \cr
  &=(n-3)(2-2g-\ell).
\end{aligned}
\end{equation}
For later use in sign computations, let us spell out these identities
mod $2$. Equation~\eqref{eq:flag} becomes
\begin{equation}\label{eq:flag-mod2}
  k \equiv s \mod 2.
\end{equation}
We use this to eliminate $k$ from the other
identities, so~\eqref{eq:Euler} and~\eqref{eq:deg-dim} become
\begin{equation}\label{eq:Euler-mod2}
  s+e \equiv \ell \mod 2,
\end{equation}
\begin{equation}\label{eq:deg-dim-mod2}
  \deg\alpha \equiv ns + (n-1)e \equiv \begin{cases}
    s & n \text{ odd} \\ e & n \text{ even}. \end{cases}
\end{equation}
Let now $\Gamma_i\in \RR_{\ell_i,g_i}$ for $i=1,2$ be two connected
trivalent ribbon graphs with extended labellings.
We denote by $s_i,k_i,e_i,\ell_i$ their number of flags, vertices,
edges, and boundary components. 
Let us introduce some notation for shuffle permutations in this
situation. Let $\mu_{\can}$ denote the shuffle of
$
\{1,\dots,\ell_1+\ell_2\}
$
sending $2$ to $\ell_1+1$ and observe that the sign of $\mu_{\can}$ is
$(-1)^{\ell_1-1}$. We denote by $G_1$ the set of shuffles of 
$
\{1,\dots,\ell_1\}
$
sending $1$ to some position in $\{1,\dots,\ell_1\}$, by $G_2$ the
set of shuffles of 
$
\{\ell_1+1,\dots,\ell_1+\ell_2\}
$
sending $\ell_1+1$ to some position in
$\{\ell_1+1,\dots,\ell_1+\ell_2\}$, 
and by $G$ the set of shuffles of
$
\{1,\dots,\ell_1+\ell_2\}
$
sending $1$ to $\{1,\dots,\ell_1\}$ and $2$ to 
$\{\ell_1+1,\dots,\ell_1+\ell_2\}$.
Note that there is a canonical bijection 
\begin{equation}\label{eq:G1G2}
  G_1\times G_2\longrightarrow G,\qquad 
  \xi\mapsto \xi\circ\mu_{\can}.
\end{equation}
Recall that we denote the resulting actions on graphs, forms and 
functionals by the same letters as the group elements themselves.
The following result is proved in~\cite{Volkov-thesis}, where
$g_{120}^l$ and $g_{210}$ are the gluing operations defined
in~\S\ref{ss:op-graphs} and $\eta_3(\Gamma)$ is the sign exponent
from~\S\ref{ss:basiccomb}. 
 
\begin{lemma}[\cite{Volkov-thesis}]\label{lem:eta3}
Let $\Gamma,\Gamma_i$ for $i=1,2$ be connected trivalent ribbon graphs
with extended labellings. Then for each $j\in \{3,\dots,s_1-1\}$ we have the following
identities mod $2$: 
\begin{align}
  \eta_3(\Gamma) + \eta_3\bigl(g_{120}^j(\Gamma)\bigr)
  &= s-1, \label{eq:eta3-1} \\
  \eta_3(\Gamma) + \eta_3\bigl(g_{210}(\Gamma)\bigr)
  &= s-1,  \label{eq:eta3-2} \\
  \eta_3(\Gamma_1) + \eta_3(\Gamma_2) + \eta_3\bigl(g_{210}((\Gamma_1\amalg\Gamma_2)
  \mu_{\can})\bigr)
  &= 1+\ell_1s_2+\ell_1+s+s_1s_2. \label{eq:eta3-3} 
\end{align}
\end{lemma}

As a corollary we get the following statement.

\begin{cor}\label{cor:deltaeta3}
Let $\Gamma_i\in \RR_{\ell_i,g_i}$ for $i=1,2$ be as above. Then for
any $\xi_i\in G_i$ we have
\begin{align*}
&\eta_3(
(g_{210}(\Gamma_1\xi_1\amalg\Gamma_2\xi_2)\mu_{\can})
)-
\eta_3((\Gamma_1\xi_1\amalg\Gamma_2\xi_2)
\mu_{\can})=\cr
&\eta_3((g_{210}(\Gamma_1\amalg\Gamma_2)
\mu_{\can}))-
\eta_3((\Gamma_1\amalg\Gamma_2)
\mu_{\can}).
\end{align*}
\end{cor}

\subsection{Degree of $\fm_{\ell,g}$}\label{ss:mdegree}

Recall from Proposition~\ref{prop:canondIBL} that $\dIBL(\HH)$ has
degree $d=n-3$. Therefore, the degree~\eqref{eq:MCdegrees} of a
Maurer-Cartan element specifies to 
$$
  |\fm_{\ell,g}|=-2(n-3)(g-1).
$$
Let us verify this for $\fm_{\ell,g}$ defined in~\S\ref{ss:defMC}.  
For this, it is enough to compute the degree of $\fm_\Gamma$ for some
$\Gamma\in\RR_{\ell,g}$. Pick a decomposable tensor
$$
  \alpha=\alpha^1\otimes\dots\otimes\alpha^\ell
  \in (B\HH[3-n])^{\otimes\ell},
$$ 
where $\alpha^b=\alpha_1^b\otimes\dots\otimes \alpha_{s_b}^b\in B\HH[3-n]$,
and assume that 
$
\fm_\Gamma(\alpha)\ne 0.
$
Recall that $|\alpha^j|$ denotes the degree of $\alpha^j$ in $B\HH$,
and let $\|\alpha^j\|$ denote its degree in $B\HH[3-n]$. We extend 
$\|\cdot\|$ to elements of $(B\HH[3-n])^{\otimes\ell}$ and compute
\begin{align*}
  \|\alpha\| &= \sum_{b=1}^\ell\|\alpha^b\| 
  = \sum_{b=1}^\ell|\alpha^b|-\ell(3-n) \cr
  &= \sum_{b=1}^\ell(\deg\alpha^b-s_b)+\ell(n-3) 
  = \deg\alpha-s+\ell(n-3)\cr
  &= -2(n-3)(g-1),
\end{align*}
where we have used equation~\eqref{eq:deg-s} for the last equality.
Therefore,
$$
  |\fm_{\ell,g}| = |\fm_\Gamma| = \|\alpha\| = -2(n-3)(g-1).
$$

\subsection{Sign computations 1: connected versus disconnected graphs}\label{ss:conndisc}

Recall from~\S\ref{ss:op-graphs} the cutting operation $c$ which is
inverse to the gluing operation $g$. It associates to a
connected o-marked graph $\wh\Gamma$ the graph $c(\wh\Gamma)$
obtained by cutting open the marked oriented edge.
Recall that $c(\wh\Gamma)$ may be disconnected, and a special extended
labelling of $\wh\Gamma$ induces an extended labelling of $c(\wh\Gamma)$. 
The proof of the Maurer-Cartan equation involves a comparison of
integrals associated to $\wh\Gamma$ and $c(\wh\Gamma)$. In the case
that $c(\wh\Gamma)$ is disconnected this requires a nontrivial sign
computation. To formulate its outcome, we introduce the auxiliary quantity
$$
  \kkk_{\wh\Gamma}(\alpha) :=
  (-1)^{(n-1)\bigl(\eta_3(\wh\Gamma)-\eta_3(c(\wh\Gamma))+s-1\bigr)}\fm_{c(\wh\Gamma)}(\alpha). 
$$
Our first goal is to get an analog of equation~\eqref{eq:action} for $\kkk$. 

\begin{lemma}\label{lem:action-k}
Let $\Gamma_i\in \RR_{\ell_i,g_i}$ for $i=1,2$ be two connected
trivalent ribbon graphs with extended labellings. 
Let $\xi_i\in G_i$ for $i=1,2$ and set $\xi:=\xi_1\times \xi_2$.
Then
$$
  \mu_{\can}^{-1}\xi^{-1}\mu_{\can}\kkk_{g_{210}
  ((\Gamma_1\amalg\Gamma_2)\mu_{\can})} =
  \kkk_{g_{210}((\Gamma_1\xi_1\amalg\Gamma_2\xi_2)\mu_{\can})}.
$$
\end{lemma}

\begin{proof}
The definition of $\kkk$ together with Corollary~\ref{cor:deltaeta3}
reduces the statement to 
$$
  \mu_{\can}^{-1}\xi^{-1}\mu_{\can}\fm_{(\Gamma_1\amalg\Gamma_2)\mu_{\can}}
  = \fm_{(\Gamma_1\xi_1\amalg\Gamma_2\xi_2)\mu_{\can}},
$$
which follows by three consecutive applications of
equation~\eqref{eq:action}. 
\end{proof}

Now the sign comparison takes the following form. 

\begin{lemma}\label{lm:disc}
(a) Let $\alpha$ be a decomposable tensor adapted to a connected graph
  $\Gamma$. Then for each $j$ we have 
$$
  \kkk_{g_{120}^j(\Gamma)}(\alpha)=\fm_\Gamma(\alpha)\quad\text{and}\quad
  \kkk_{g_{210}(\Gamma)}(\alpha)=\fm_\Gamma(\alpha).
$$
(b) Let $\alpha_1$ and $\alpha_2$ be decomposable tensors adapted to
  connected graphs $\Gamma_1$ and $\Gamma_2$, respectively, 
and set $\alpha:=\alpha_1\alpha_2$.
Then
$$
  \kkk_{g_{210}((\Gamma_1\amalg \Gamma_2)\mu_{\can})}(\mu_{\can}\alpha) = (\fm_{\Gamma_1}\otimes \fm_{\Gamma_2})(\alpha).
$$
In other words,  
$$
  \kkk_{g_{210}((\Gamma_1\amalg \Gamma_2)\mu_{\can})}
  = \mu_{\can}^{-1}(\fm_{\Gamma_1}\otimes \fm_{\Gamma_2}).
$$
\end{lemma}

\begin{proof}
For the first equality in (a) we consider the o-marked connected graph
$\wh\Gamma:=g_{120}^j(\Gamma)$, so that $c(\wh\Gamma)=\Gamma$.
Then using the definition of 
$\kkk_{\wh\Gamma}$ and~\eqref{eq:eta3-1} we obtain
$$
  \kkk_{g_{120}^j(\Gamma)}(\alpha)
  = (-1)^{(n-1)\bigl(\eta_3(g_{120}^j(\Gamma))-\eta_3(\Gamma)+s-1\bigr)}\fm_\Gamma(\alpha)
  =\fm_\Gamma(\alpha).
$$
The second equality in (a) follows analogously using~\eqref{eq:eta3-2}.

For (b) we consider the disconnected graph $\Gamma:=\Gamma_1\amalg\Gamma_2$
and the o-marked connected graph $\wh\Gamma:=g_{210}(\Gamma\mu_{\can})$, so
that $c(\wh\Gamma)=\Gamma\mu_{\can}$. Abbreviating
$$
  s := s_1+s_2\quad\text{and}\quad
  S:=s_{\Gamma_1}(\alpha_1)+s_{\Gamma_2}(\alpha_2)+s_\Gamma(\alpha),
$$
we compute
\begin{align*}
  (\fm_{\Gamma_1}\otimes \fm_{\Gamma_2})(\alpha)
  &\stackrel{(1)}{=}
  (-1)^{s_{\Gamma_1}(\alpha_1)+s_{\Gamma_2}(\alpha_2)}(I_{\Gamma_1}\otimes I_{\Gamma_2})(\alpha) \cr
  &\stackrel{(2)}{=}
  (-1)^{s_{\Gamma_1}(\alpha_1)+s_{\Gamma_2}(\alpha_2)+(n-1)\alpha_1e_2}I_\Gamma(\alpha) \cr
  &\stackrel{(3)}{=}
  (-1)^{S+(n-1)\alpha_1e_2}\fm_\Gamma(\alpha) \cr
  &\stackrel{(4)}{=}
  (-1)^{S+(n-1)\alpha_1e_2}\fm_{\Gamma\mu_{\can}}(\mu_{\can}\alpha) 
  \cr
  &\stackrel{(5)}{=}
  (-1)^{S+(n-1)\bigl(\alpha_1e_2+
  \eta_3(g_{210}(\Gamma\mu_{\can}))-\eta_3(\Gamma\mu_{\can})+s-1\bigr)}
  \kkk_{\wh\Gamma}(\mu_{\can}\alpha) 
  \cr
  &\stackrel{(6)}{=}
  (-1)^X\kkk_{g_{210}(\Gamma\mu_{\can})}(\mu_{\can}\alpha)
\end{align*}
with the sign exponent
$$
  X = S + (n-1)\bigl(\alpha_1e_2+\ell_1s_2+s_1s_2+1\bigr).
$$
Here for equality (1) we use the definitions of $\fm_{\Gamma_1}(\alpha_1)$ and
$\fm_{\Gamma_1}(\alpha_2)$; for (2) we use equation~\eqref{eq:2comp};
for (3) we use the definition of $\fm_\Gamma$; for (4) we use the
equivariance property~\eqref{eq:action} of $\fm_\Gamma$; 
for (5) we use the definition of $\kkk_{\wh\Gamma}=\kkk_{g_{210}(\Gamma\mu_{can})}$;
and (6) follows from
$\eta_3(\Gamma\mu_{can})=\eta_3(\Gamma_1)+\eta_3(\Gamma_2)+\eta$, $\eta=\ell_1-1$
and~\eqref{eq:eta3-3} via
\begin{align*}
  \eta_3(g_{210}(\Gamma\mu_{can}))+\eta_3(\Gamma\mu_{can})
  &= \eta_3(g_{210}(\Gamma\mu_{can}))+\eta_3(\Gamma_1)+\eta_3(\Gamma_2)+\ell_1-1\cr
  &= 1+\ell_1s_2+\ell_1+s+s_1s_2+\ell_1-1 \cr
  &= \ell_1s_2+s+s_1s_2.
\end{align*}
It remains to prove that $X=0$. By definition~\eqref{eq:sGamma} of
$s_\Gamma$, the sign exponent $S$ decomposes into four terms
$$
  S = S_1+S_2+(n-1)S_3+(n-1)S_4,
$$
where 
\begin{align*}
  S_1 &= \frac{s_1(s_1-1)}{2}+\frac{s_2(s_2-1)}{2}+\frac{s(s-1)}{2} =
  s_1s_2, \cr
  S_2 &= P(\alpha_1)+P(\alpha_2)+P(\alpha) \cr
  &= \sum_{j=1}^{s_1}(s_1-j)\alpha^j + \sum_{j=s_1+1}^{s}(s-j)\alpha^j
  + \sum_{j=1}^{s}(s-j)\alpha^j \cr
  &= \sum_{j=1}^{s_1}(s-j+s_1-j)\alpha^j
  = s_2\alpha_1, \cr
  S_3 &= se_{\Gamma_1}+se_{\Gamma_2}+se_\Gamma
  = s_1\ell_2+s_2\ell_1+1, \cr
  S_4 &= P_b(\alpha_1)+P_b(\alpha_2)+P_b(\alpha) \cr
  &= \sum_{b=1}^{\ell_1}(\ell_1-b)|\alpha^b| +
  \sum_{b=\ell_1+1}^{\ell}(\ell-b)|\alpha^b| +
  \sum_{b=1}^{\ell}(\ell-b)|\alpha^b| \cr
  &= \sum_{b=1}^{\ell_1}(\ell-b+\ell_1-b)|\alpha^b|
  = \ell_2|\alpha_1| = \ell_2(\alpha_1+s_1).
\end{align*}
Here for $S_2$ we have written $\alpha=\alpha^1\cdots\alpha^s$ with
$\alpha^j$ corresponding to the $j$-th leaf of $\Gamma$;  
for $S_3$ we have used property~\eqref{eq:seGamma-sum} of $se_\Gamma$;
and for $S_4$ we have set $\ell:=\ell_1+\ell_2$ and written
$\alpha=\alpha^1\cdots\alpha^\ell$ with $\alpha^b$ corresponding to
the $b$-th boundary component of $\Gamma$. 

Now we distinguish two cases according to the parity of $n$. 
For $n$ odd we have $\alpha_1 = s_1$ by~\eqref{eq:deg-dim-mod2},
and therefore
$$
  X = S = S_1+S_2 = s_1s_2 + s_2\alpha_1 = 0.
$$
For $n$ even we use $s_2+\ell_2=e_2$ from~\eqref{eq:Euler-mod2} to compute 
\begin{align*}
  S &= S_1+S_2+S_3+S_4 \cr
  &= s_1s_2+s_2\alpha_1+s_1\ell_2+s_2\ell_1+1+\ell_2(\alpha_1+s_1) \cr
  &= s_1s_2+e_2\alpha_1+s_2\ell_1+1.
\end{align*}
This again implies $X=0$ and concludes the proof of Lemma~\ref{lm:disc}.
\end{proof}

\subsection{Sign computations 2: assembling the full differential}

Recall from equation~\eqref{eq:dG} the formula $dG=e_ae^a$, where
$\{e_a\}$ is a basis of harmonic forms with dual basis $\{e^a\}$ and
summation over $a$ is understood. We will use this in the proof of the
Maurer-Cartan equation to relate the total differential of the
integrand to the operations $\fp_{2,1,0}$ and $\fp_{1,2,0}$, which are
defined by inserting $e_a$ and $e^a$ in suitable positions and summing
over $a$. In this subsection we work out the corresponding signs.

Recall the gluing operation $g_{210}$ from~\S\ref{ss:op-graphs},
the product $b_{210}^{12}$ from~\S\ref{ss:p210}, 
and the operations $\n_{\wh\Gamma}$ from equation~\eqref{eq:mGamma}
and $\kkk_{\wh\Gamma}$ from~\S\ref{ss:conndisc} associated to an
o-marked graph $\wh\Gamma$. 
Note that $g_{210}(\Gamma)$ is an o-marked graph, so the operations
$\n_{g_{210}(\Gamma)}$ and $\kkk_{g_{210}(\Gamma)}$ are defined for
each graph $\Gamma$.

\begin{lemma}\label{lem:b_{210}conn}
Let $\Gamma$ be a (possibly disconnected) graph such that
$g_{210}(\Gamma)$ is connected. Then 
\begin{equation}\label{eq:b210conn}
   b_{210}^{12}\kkk_{g_{210}(\Gamma)}=\n_{g_{210}(\Gamma)}.
\end{equation}
\end{lemma}

\begin{proof}
Set $\wh\Gamma:=g_{210}(\Gamma)$ and let $l$ be its o-marked edge.
We equip $\Gamma$ with an extended labelling and
$\wh\Gamma$ with the induced special extended labelling, in which $l$
has position $n(l)=1$.
In the following we will denote quantities associated to $\wh\Gamma$
with a hat, and those associated to 
$\Gamma$ without. 
Let $\ell$ be the number of boundary components of $\Gamma$, so that
$\wh\ell=\ell-1$ is the number of boundary components of $\wh\Gamma$. 

Consider a collection of decomposable tensors of harmonic forms 
$\alpha^1,\alpha^2,\dots,\alpha^\ell$ such that 
$(e_a\alpha^1)\otimes(e^a\alpha^2)\otimes\dots\otimes \alpha^\ell$ is
adapted $\Gamma$. Then 
$(\alpha^1\alpha^2)\otimes \alpha^3\otimes\dots\otimes 
\alpha^\ell$ is adapted to $\wh\Gamma$ and we will plug it into
both sides of equation~\eqref{eq:b210conn}. 
For $b=1,\dots,\ell$ let $s_b$ be the number of harmonic forms in the tensor $\alpha^b$.
Then $\wh s = s_1+\cdots+s_\ell$ and $s=\wh s+2$ are the numbers of
leaves of $\wh\Gamma$ and $\Gamma$, respectively. 
We have the following equality between forms on the configuration
space $X_{\wh\Gamma}$, where summation over $a$ is understood:
\begin{equation}\label{eq:integrands}
   (-1)^{e^a\alpha^1}(R_{\wh\Gamma}^{-1})^*R_\Gamma^*
   G^e(e_a\alpha^1\otimes e^a\alpha^2\otimes\dots\otimes\alpha^\ell)
  = G_l^{e+1}(\alpha^1\alpha^2\otimes\dots\otimes\alpha^\ell).
\end{equation}
This follows from the definitions of the reordering map $R_\Gamma$
in~\S\ref{ss:config} and the integrands $G^e(\alpha)$, $G^e_l(\alpha)$
in~\S\ref{sec:opforms} via the following computation, using
$dG=e_ae^a$ from equation~\eqref{eq:dG}:
\begin{align*}
  &\ \ \ \ (-1)^{e^a\alpha^1}(R_{\wh\Gamma}^{-1})^*R_\Gamma^*G^e
  (e_a\alpha^1\otimes e^a\alpha^2\otimes\dots
  \otimes\alpha^\ell) \cr 
  &= (-1)^{e^a\alpha^1}(R_{\wh\Gamma}^{-1})^*R_\Gamma^*(G^e\times
  e_a\alpha^1\times e^a\alpha^2\times\dots\times \alpha^\ell) \cr 
  &= e_ae^a\times G^e\times \alpha^1\times\dots\times\alpha^\ell \cr 
  &= dG\times G^e\times \alpha^1\times\dots\times \alpha^\ell \cr 
  &= G_l^{e+1}(\alpha^1\alpha^2\otimes\dots\otimes\alpha^\ell).
\end{align*}
%
%
We abbreviate the sign exponents 
$P((e_a\alpha^1)\otimes (e^a\alpha^2)\otimes\dots\otimes
\alpha^\ell)$
and $P_b((e_a\alpha^1)\otimes 
(e^a\alpha^2)\otimes\dots\otimes
\alpha^\ell)$ that correspond to the
conjugating operators $P$ and $P_b$ for the graph $\Gamma$ by $P$
and $P_b$, respectively, and the sign exponents
$P(\alpha^1\alpha^2\otimes\dots\otimes\alpha^\ell)$ and 
$P_b(\alpha^1\alpha^2\otimes\dots\otimes
\alpha^\ell)$ 
for the graph $\wh\Gamma$ by $\wh P$ and $\wh P_b$.
Multiplying both sides of~\eqref{eq:integrands} by $(-1)^\star$,
pulling back with $R_{\wh\Gamma}^*$ and integrating over $\Delta_3^k$ yields
\begin{equation}\label{eq:GGtilde}
(-1)^{e^a\alpha^1+\star} \int_{\Delta_3^k}R_\Gamma^*G^e(e_a\alpha^1\otimes e^a\alpha^2\otimes\dots\otimes\alpha^\ell)=
(-1)^{\star}\int_{\Delta_3^k}R_{\wh\Gamma}^*
G_l^{e+1}(\alpha^1\alpha^2\otimes\dots
\otimes\alpha^\ell).
\end{equation}
Here we define the sign exponent
\begin{align*}
\star := &e^a\alpha^1 + \bar R_\Gamma+(n-1)\eta_3(\Gamma)\cr
&+ n\ell+P+s(s+1)/2+(n-1)(P_b+se_\Gamma) \cr
&+ (n-1)\bigl(\eta_3(\wh\Gamma)-\eta_3(\Gamma)+s-1\bigr) \cr
&+ |e^a||\alpha^1|+|e_a|+(n-1)|e_a\alpha^1|. 
\end{align*}
Then the left hand side of~\eqref{eq:GGtilde} becomes the term in
$$
  b_{210}^{12}(\kkk_{\wh\Gamma})(\alpha^1\alpha^2\otimes\alpha^3\otimes\dots\otimes\alpha^\ell) 
  = \kkk_{\wh\Gamma}\circ c_{120}^1(\alpha^1\alpha^2\otimes\alpha^3
  \otimes\dots\otimes\alpha^\ell)
$$
corresponding to the splitting of $\alpha^1\alpha^2$ into 
$\alpha^1$ and $\alpha^2$. Indeed, the first line in the definition of
$\star$ offsets $e^a\alpha^1$ and gives the sign exponent
$R_\Gamma+(n-1)\eta_3(\Gamma)$ in the operation $I_\Gamma$;
the second line is the sign exponent~\eqref{eq:sGamma} for converting $I_\Gamma$ to $\fm_\Gamma$;
the third line is the sign exponent for converting $\fm_\Gamma$ to
$\kkk_{\wh\Gamma}$;
and the last line is the sign exponent in the definition of $c_{120}^1$
in~\S\ref{ss:p210}. 

To understand the right hand side of~\eqref{eq:GGtilde} we define
\begin{align*}
\star\star := &\star + \bar R_{\wh\Gamma} + (n-1)\eta_3(\wh\Gamma)\cr
&+ n(\wh\ell-1) +1 + (\wh s+1)(\wh s+2)/2  +\wh P +
(n-1)(\wt{se}_{\wh\Gamma} + \wh P_b).
\end{align*}
Here the first line offsets $\star$ and gives the sign exponent
$\bar R_{\wh\Gamma} + (n-1)\eta_3(\wh\Gamma)$ in the operation $I_{\wh\Gamma}$,
and the second line is the sign exponent~\eqref{eq:sGammatilde} for
converting $I_{\wh\Gamma}$ to $\n_{\wh\Gamma}$. 
So the right-hand side of~\eqref{eq:GGtilde} equals
$\n_{\wh\Gamma}(\alpha^1\alpha^2\otimes\dots\otimes\alpha^\ell)$ and the lemma follows provided 
that $\star\star = 0$ mod $2$. 

In order to prove $\star\star=0$, we regroup $\star\star$ after some
obvious cancellations as
\begin{align*}
  \star\star
  = &\bar R_\Gamma + \bar R_{\wh\Gamma} \cr
  &+ P + \wh P \cr
  &+ (n-1)(P_b + \wh P_b + |e_a\alpha^1|) \cr
  &+ (n-1)(se_\Gamma + \wt{se}_{\wh\Gamma} + s-1) \cr
  &+ e^a\alpha^1 + 1 + |e^a|\,|\alpha^1| + |e_a| \cr
  &+ s(s+1)/2 + (\wh s+1)(\wh s+2)/2 \cr
\end{align*}
and compute it line by line.
Since the first flags appear in the order $e_a\alpha^1e^a\alpha^2$ for $\Gamma$
and $e_ae^a\alpha^1\alpha^2$ for $\wh\Gamma$, their edge orders differ my moving
the flag corresponding to $e^a$ past the $s_1$ flags corresponding to
$\alpha^1$ and the first line becomes
$$
  \bar R_\Gamma + \bar R_{\wh\Gamma} = s_1.
$$
The second line $P+\wh P$ equals the sum of the sign exponents for
moving the formal degree variables $\theta$ to the left of all the
other variables in the expressions $\theta e_a\theta^{s_1}\alpha^1\theta
e^a\theta^{s_2+\dots+s_\ell}$ (for $\Gamma$) and
$\theta^{s_1}\alpha^1\theta^{s_2+\dots+s_\ell}$ (for $\wh\Gamma$). 
Moving $\theta^{s_1}$ past $e_a$ gives the sign exponent $s_1e_a$,
moving the $\theta$ to the right of $\alpha^1$ past $e_a$ and $\alpha^1$ gives
$e_a+\alpha^1$, and moving $\theta^{s_2+\dots+s_\ell}$ past $e_a$ and $e^a$
gives $(e_a+e^a)(s_2+\dots+s_\ell)=n(s_2+\dots+s_\ell)$, hence
$$
  P+\wh P = s_1e_a + e_a + \alpha^1 + n(s_2+\dots+s_\ell).
$$
The sum $P_b+\wh P_b$ equals the sum of the sign exponents for
moving the formal degree variables $\theta$ to the left of all the
other variables in the expressions $\theta e_a\alpha^1\theta
e^a\alpha^2\theta^{\ell-2}$ (for $\Gamma$) and
$\theta \alpha^1\alpha^2\theta^{\ell-2}$ (for $\wh\Gamma$). 
Moving the $\theta$ to the right of $\alpha^1$ past $e_a\alpha^1$ gives
$|e_a\alpha^1|$, and moving $\theta^{\ell-2}$ past $e_a$ and $e^a$
gives $|e_ae^a|(\ell-2)=n(\ell-2)$, hence the third line becomes
$$
  (n-1)(P_b + \wh P_b + |e_a\alpha^1|) = (n-1)\bigl(|e_a\alpha^1| + n(\ell-2) +
|e_a\alpha^1|\bigr) = 0.
$$
By definition~\eqref{eq:seGamma} and $\wh s=s-2$ we get
$se_\Gamma + \wt{se}_{\wh\Gamma} = (\ell+1)(s+1)+\wh\ell(\wh s+1)+1
= (\ell+1)(s+1)+(\ell-1)(s-1)+1 = 1$, so the fourth line becomes
$$
  (n-1)(se_\Gamma + \wt{se}_{\wh\Gamma} + s-1) = (n-1)s.
$$
Using $e^a=e_a+n$, the fifth line becomes
$$
  e^a\alpha^1+1+|e^a|\,|\alpha^1|+|e_a|
  = (e_a+n)\alpha^1+(e_a+n-1)(\alpha^1+s_1)+e_a
  = \alpha^1+e_as_1+e_a+(n-1)s_1.
$$
Using $\wh s=s-2$ the sixth line becomes
$$
  \frac{s(s+1)}{2} + \frac{(\wh s+1)(\wh s+2)}{2}
  = \frac{(s(s+1) + (s-1)s}{2}
  = \frac{2s^2}{2} = s.
$$
Summing up all the lines and using $s_1+\cdots+s_\ell\equiv s$ mod $2$
we get 
\begin{align*}
  \star\star
  &= s_1 + s_1e_a + e_a + \alpha^1 + n(s_2+\dots+s_\ell) + 0 \cr
  &\ \ \ + (n-1)s + \alpha^1+e_as_1+e_a+(n-1)s_1 + s \cr
  &= n(s_1+\dots+s_\ell) + ns = 0.
\end{align*}
This concludes the proof of Lemma~\ref{lem:b_{210}conn}. 
\end{proof}

Similarly, with the gluing operation $g_{120}^j$ from~\S\ref{ss:op-graphs}
and the coproduct $b_{120}^1$ from~\S\ref{ss:p210} we have 

\begin{lemma}\label{lem:b_{120}}
Let $\Gamma$ be a connected graph with $s_1\geq 4$ leaves on its first
boundary component. Then for every $i\in\{3,s_1-1\}$ we have 
$$
  b_{120}^1\kkk_{g_{120}^i(\Gamma)} = \sum_{j=3}^{s_1-1}\frac{1}{2}\n_{g_{120}^j(\Gamma)}.
$$
\end{lemma}

Note that, in particular, this means that the left hand side of the
equation does not depend on $i$. This makes sense because the graph
$g_{120}^i(\Gamma)$ is obtained from $\Gamma$ by gluing the
leaves $1$ and $i$ on its first boundary component, and $\kkk_{g_{120}^i(\Gamma)}$
is defined by cutting the graph $g_{120}^i(\Gamma)$ back to $\Gamma$. 
The summation over $j$ on the right hand side arises from the
definition of the operation $b_{120}^1$.

\begin{proof}[Sketch of proof]
For the coproduct we start with a 
$\alpha^1\otimes\alpha^2\otimes\dots\otimes\alpha^{\ell}$ and
split the decomposable tensor $\alpha^1$ as
$\alpha^1_{(1)}\otimes\alpha^1_{(2)}$ in all possible ways. Now with
$\alpha^1_{(1)}$ in place of $\alpha^1$, the signs are completely
analogous to those in Lemma~\ref{lem:b_{210}conn}.  
\end{proof}

\subsection{Proof of the Maurer-Cartan equation}\label{ss:maincompute}

We recall equation~\eqref{eq:MCIBL} and set 
$$
   MC:=\wh\fp_{2,1,0}\m+\frac{1}{2}\wh\fp_{2,1,0}(\m\otimes
   \m)|_{conn}+\wh\fp_{1,2,0}\m \in\wh E C,
$$
where in the second term we take the part of $\wh\fp_{2,1,0}(\m\otimes\m)$
corresponding to connected graphs. This means that when we sum over
shuffles we take only those ones that lead to a connected graph when
we apply the gluing operation $g_{210}$ 
(the proper set of shuffles will again be denoted by $Sh$). 
The $(\ell,g)$ part of $MC$ is
\begin{equation}\label{eq:MC-lg}
   MC_{\ell,g} = \wh\fp_{2,1,0}\fm_{\ell+1,g-1} + 
   \frac{1}{2}\sum_{\substack{\ell_1+\ell_2=\ell+1\\g_1+g_2=g}}\wh\fp_{2,1,0}(\fm_{\ell_1,g_1}\otimes
   \fm_{\ell_2,g_2})|_{conn} + \wh\fp_{120}\fm_{\ell-1,g}.
\end{equation}
Our goal is to prove
\begin{equation}\label{eq:MC1}
   MC_{\ell,g}=0.
\end{equation}
For this, we manipulate the $3$ summands in $MC_{\ell,g}$ one by one.

{\bf The first summand. }
We rewrite the first summand as
\begin{align}\label{eq:summand1}
   \wh\fp_{2,1,0}\fm_{\ell+1,g-1}
   &\stackrel{(1)}{=} \sum_{\eta\in Sh_{2,\ell-1}}p^{12}_{210}\eta^{-1}\fm_{\ell+1,g-1} \cr
   & \stackrel{(2)}{=} \frac{(\ell+1)\ell}{2}\,p_{210}^{12}\fm_{\ell+1,g-1} \cr
   &\stackrel{(3)}{=} \frac{1}{2(\ell-1)!}\sum_{\Gamma\in \RR_{\ell+1,g-1}}p_{210}^{12}\fm_{\Gamma}\cr
   &\stackrel{(4)}{=} \frac{1}{2(\ell-1)!}\sum_{\Gamma\in
     \RR_{\ell+1,g-1}}\sum_{\rho\in\Z_{\wt s_1}}\rho\,b_{210}^{12}\kkk_{g_{210}(\Gamma)}\cr
   &\stackrel{(5)}{=} \frac{1}{2(\ell-1)!}\sum_{\Gamma\in
     \RR_{\ell+1,g-1}}\sum_{\rho\in\Z_{\wt s_1}}\rho\,\n_{g_{210}(\Gamma)}\cr
     &\stackrel{(6)}{=} \frac{1}{2(\ell-1)!}
     \sum_{\wh\Gamma\in
     \RR_{\ell,g,1c}^{oms}}\sum_{\rho\in\Z_{\wt s_1}}\rho^{-1}\,\n_{\wh\Gamma}\cr
     &\stackrel{(7)}{=} \frac{1}{2(\ell-1)!}
     \sum_{\wh\Gamma\in
     \RR_{\ell,g,1c}^{oms}}\sum_{\rho\in\Z_{\wt s_1}}\n_{\wh\Gamma\rho}\cr
     &\stackrel{(8)}{=} \frac{1}{2(\ell-1)!}
     \left(\frac{1}{\ell}\sum_{\tau\in Sh_{1,\ell-1}}\tau^{-1}\right)
     \sum_{\wh\Gamma\in
     \RR_{\ell,g,1c}^{oms}}\sum_{\rho\in\Z_{\wt s_1}}\n_{\wh\Gamma\rho}\cr
     & \stackrel{(9)}{=}\frac{1}{2\ell!}\sum_{\wh\Gamma\in
     \RR_{\ell,g,1c}^{oms}}\sum_{\tau\in Sh_{1,\ell-1}}\sum_{\rho\in\Z_{\wt s_1}}\n_{\wh\Gamma\rho\tau}\cr
   &\stackrel{(10)}{=}\frac{1}{2\ell!}\sum_{\wh\Gamma\in R^{om}_{\ell,g,1c}}\n_{\wh\Gamma}\cr
   &\stackrel{(11)}{=}\frac{1}{\ell!}
   \sum_{\wh\Gamma\in R^{m}_{\ell,g,1c}}\n_{\wh\Gamma}.
  \end{align}
Here $(1)$ holds by definition of $\wh\fp_{2,1,0}$;
$(2)$ by invariance of $\fm_{\ell+1,g-1}$ under reordering of
the boundary components and $|Sh_{2,\ell-1}|=(\ell+1)\ell/2$;
$(3)$ by definition of $\fm_{\ell+1,g-1}$; 
$(4)$ by Lemma~\ref{lm:disc} and definition of 
$p_{210}^{12}$, where $\wt s_1$ is
the number of leaves on the first boundary component of the
glued graph $g_{210}(\Gamma)$; 
$(5)$ by Lemma~\ref{lem:b_{210}conn}; for
$(6)$ by the bijection~\eqref{eq:glue1} and 
the fact that taking the inverse induces a bijection on 
any group;
$(7)$ by equation~\eqref{eq:action};
$(8)$ because the operation in the big round brackets 
is the identity on the quotient space $\wh E_\ell C$
(recall that $|Sh_{1,\ell-1}|=\ell$)
and the expression to its right defines an element of $\wh E_\ell C$;
$(9)$ by equation~\eqref{eq:action};
$(10)$ by the bijection~\eqref{eq:slab1};
and $(11)$ by independence of $\n_{\wh\Gamma}$ of the orientation of
the marked edge.

{\bf The second summand. }
For each $\ell_1,\ell_2,g_1,g_2$ with
$\ell_1+\ell_2=\ell_1$ and $g_1+g_2=g-1$ we rewrite the corresponding
term in the second summand as
\begin{align*}
   &\frac{1}{2}\wh \fp_{2,1,0}(\fm_{\ell_1,g_1}\otimes \fm_{\ell_2,g_2})|_{conn}\cr
   &\stackrel{(1)}{=}\frac{1}{2}\sum_{\eta\in G}p^{12}_{210}\eta^{-1}(\fm_{\ell_1,g_1}\otimes \fm_{\ell_2,g_2}) \cr
   &\stackrel{(2)}{=} \frac{1}{2\ell_1!\ell_2!}\sum_{\eta\in G}\,\sum_{\Gamma_i\in
    \RR_{\ell_i,g_i}}p_{210}^{12}(\eta^{-1}\mu_{\can})
    \mu_{\can}^{-1}(\fm_{\Gamma_1}\otimes\fm_{\Gamma_2})\cr 
   &\stackrel{(3)}{=} 
   \frac{1}{2\ell_1!\ell_2!}\sum_{\xi\in G_1\times G_2}\,
   \sum_{\Gamma_i\in
    \RR_{\ell_i,g_i}}p_{210}^{12}
    (\mu_{\can}^{-1}\xi^{-1}\mu_{\can})
    \kkk_{g_{210}(\Gamma)}\cr 
   &\stackrel{(4)}{=} \frac{1}{2(\ell_1-1)!(\ell_2-1)!}\sum_{\Gamma_i\in
    \RR_{\ell_i,g_i}}p_{210}^{12}\kkk_{g_{210}(\Gamma)}\cr 
   &\stackrel{(5)}{=} \frac{1}{2(\ell_1-1)!(\ell_2-1)!}\sum_{\Gamma_i\in
    \RR_{\ell_i,g_i}}\sum_{\rho\in\Z_{\wt s_1}}\rho\,b_{210}^{12}\kkk_{g_{210}(\Gamma)}\cr 
   &\stackrel{(6)}{=} \frac{1}{2(\ell_1-1)!(\ell_2-1)!}\sum_{\Gamma_i\in
    \RR_{\ell_i,g_i}}\sum_{\rho\in\Z_{\wt s_1}}\rho\,\n_{g_{210}(\Gamma)}\cr 
    &\stackrel{(7)}{=}\frac{1}{2(\ell_1-1)!(\ell_2-1)!}
        \sum_{\wh\Gamma\in\RR^{oms}_{\ell_1,\ell_2,g_1,g_2,1dc}}
        \sum_{\rho\in\Z_{\wt s_1}}\rho^{-1}\,\n_{\wh\Gamma}\cr 
        &\stackrel{(8)}{=}\frac{1}{2(\ell_1-1)!(\ell_2-1)!}
        \sum_{\wh\Gamma\in\RR^{oms}_{\ell_1,\ell_2,g_1,g_2,1dc}}
        \sum_{\rho\in\Z_{\wt s_1}}
        \n_{\wh\Gamma\rho}\cr
    &\stackrel{(9)}{=}\frac{1}{2(\ell_1-1)!(\ell_2-1)!}
    \left(\frac{(\ell_1-1)!(\ell_2-1)!}{\ell!}\sum_{\tau\in Sh_{1,\ell_1-1,\ell_2-1}}\tau^{-1}\right)
    \sum_{\wh\Gamma\in\RR^{oms}_{\ell_1,\ell_2,g_1,g_2,1dc}}
    \sum_{\rho\in\Z_{\wt s_1}}
    \n_{\wh\Gamma\rho}\cr 
    &\stackrel{(10)}{=} \frac{1}{2\ell!}
    \sum_{\wh\Gamma\in \RR^{oms}_{\ell_1,\ell_2,g_1,g_2,1dc}} 
   \sum_{\tau\in Sh_{1,\ell_1-1,\ell_2-1}}
   \sum_{\rho\in\Z_{\wt s_1}}
   \n_{\wh\Gamma\rho\tau}\cr
   &\stackrel{(11)}{=} \frac{1}{2\ell!}
   \sum_{\wh\Gamma\in
       \RR^{om}_{\ell_1,\ell_2,g_1,g_2,1dc}}\n_{\wh\Gamma}. 
   \end{align*}
Hence, the second summand becomes
\begin{align}\label{eq:summand2}
   & \frac{1}{2}\sum_{\substack{\ell_1+\ell_2=\ell+1\\g_1+g_2=g}}
   \wh \fp_{2,1,0}(\fm_{\ell_1,g_1}\otimes \fm_{\ell_2,g_2})|_{conn} \cr
   &  =\sum_{\substack{\ell_1+\ell_2=\ell+1 \\ g_1+g_2=g}}
     \frac{1}{2\ell!}\sum_{\wh\Gamma\in \RR^{om}_{\ell_1,\ell_2,g_1,g_2,1dc}}
            \n_{\wh\Gamma}
   \stackrel{(12)}{=} \frac{1}{2\ell!}\sum_{\wh\Gamma\in R^{om}_{\ell,g,1dc}}\n_{\wh\Gamma}
   \stackrel{(13)}{=} \frac{1}{\ell!}\sum_{\wh\Gamma\in R^{m}_{\ell,g,1dc}}\n_{\wh\Gamma}.
\end{align}
Here $(1)$ holds by definition of the connected part of
$\wh\fp_{2,1,0}(\fm_{\ell_1,g_1}\otimes \fm_{\ell_2,g_2})$, recalling the subset 
$G\subset Sh_{2,\ell_1+\ell_2}$ of shuffles sending $1$ to
$\{1,\dots,\ell_1\}$ and $2$ to $\{\ell_1+1,\dots,\ell_1+\ell_2\}$;
$(2)$ by definition of $\fm_{\ell_i,g_i}$;
$(3)$ by the bijection~\eqref{eq:G1G2} and Lemma~\ref{lm:disc}, setting
$\eta=\xi\circ\mu_{can}$ and $\Gamma=(\Gamma_1\amalg \Gamma_2)\mu_{\can}$;
$(4)$ from Lemma~\ref{lem:action-k} and $|G|=\ell_1\ell_2$, adjusting
the combinatorial factors for the corresponding overcounting since
$\fm_{\Gamma_i}$ is invariant under permutations of boundary components;
$(5)$ by definition of $p_{210}^{12}$, where $\wt s_1$ is
the number of leaves on the first boundary component of the
glued graph $g_{210}(\Gamma)$; 
$(6)$ by Lemma~\ref{lem:b_{210}conn}; 
$(7)$ by the bijection~\eqref{eq:glue2} and 
the fact that taking the inverse induces a bijection on any group;
$(8)$ by equation~\eqref{eq:action};
$(9)$ because the operation in the big round brackets 
is the identity on the quotient space $\wh E_\ell C$
(recall that $|Sh_{1,\ell_1-1,\ell_2-1}|=\frac{\ell!}{(\ell_1-1)!(\ell_2-1)!}$)
and the expression to its right defines an element of $\wh E_\ell C$;
$(10)$ by equation~\eqref{eq:action};
$(11)$ by the bijection~\eqref{eq:slab2};
$(12)$ by the relation~\eqref{eq:ell12ell};
and $(13)$ by independence of 
$\n_{\wh\Gamma}$ of the orientation of
the marked edge.

{\bf The third summand. }
We rewrite the third summand as
\begin{align}\label{eq:summand3}
   \wh\fp_{1,2,0}\fm_{\ell-1,g}
   &\stackrel{(1)}{=} \sum_{\eta\in Sh_{1,\ell-2}}\fp^{1}_{120}\eta^{-1}\fm_{\ell-1,g} \cr
   &\stackrel{(2)}{=} (\ell-1)\,p_{120}^{1}\fm_{\ell-1,g} \cr
   &\stackrel{(3)}{=} \frac{1}{(\ell-2)!}\sum_{\Gamma\in \RR_{\ell-1,g}}p_{120}^{1}\fm_{\Gamma}\cr
   &\stackrel{(4)}{=}\frac{1}{(\ell-2)!}\sum_{\Gamma\in \RR_{\ell-1,g}}
   \sum_{j=3}^{s_1-1}
   \sum_{\rho\in\Z_{\wt s_1}\times\Z_{\wt s_2}}\rho\,b_{120}^{1}
   \kkk_{g_{120}^i(\Gamma)}\cr 
   &\stackrel{(5)}{=} \frac{1}{2(\ell-2)!}\sum_{\Gamma\in \RR_{\ell-1,g}}
   \sum_{j=3}^{s_1-1}\sum_{\rho\in\Z_{\wt s_1}\times\Z_{\wt
       s_2}}\rho^{-1}\,\n_{g_{120}^j(\Gamma)}\cr 
       &\stackrel{(6)}{=} \frac{1}{2(\ell-2)!}
       \sum_{\wh\Gamma\in \RR_{\ell,g,12}^{oms}}\sum_{\rho\in\Z_{\wt s_1}\times\Z_{\wt
       s_2}}\rho^{-1}\,\n_{\wh\Gamma}\cr 
      &\stackrel{(7)}{=} \frac{1}{2(\ell-2)!}
       \sum_{\wh\Gamma\in \RR_{\ell,g,12}^{oms}}\sum_{\rho\in\Z_{\wt s_1}\times\Z_{\wt
       s_2}}\,\n_{\wh\Gamma\rho}\cr 
   &\stackrel{(8)}{=} 
   \frac{1}{2(\ell-2)!}\left(\frac{1}{\ell(\ell-1)}\sum_{\tau\in Sh_{1,1,\ell-2}}\tau^{-1}\right) 
   \sum_{\wh\Gamma\in\RR_{\ell,g,12}^{oms}}
   \sum_{\rho\in\Z_{\wt s_1}\times\Z_{\wt
       s_2}}\,\n_{\wh\Gamma\rho}\cr 
   &\stackrel{(9)}{=}\frac{1}{2\ell!}
   \sum_{\tau\in Sh_{1,1,\ell-2}}
   \sum_{\wh\Gamma\in\RR_{\ell,g,12}^{oms}}
   \sum_{\rho\in\Z_{\wt s_1}\times\Z_{\wt
       s_2}}\,\n_{\wh\Gamma\rho\tau}\cr 
   &\stackrel{(10)}{=} \frac{1}{2\ell!}\sum_{\wh\Gamma\in R^{om}_{\ell,g,12}}\n_{\wh\Gamma}\cr
   &\stackrel{(11)}{=} \frac{1}{\ell!}\sum_{\wh\Gamma\in R^{m}_{\ell,g,12}}\n_{\wh\Gamma}.
\end{align}
Here equality $(1)$ holds by definition of $\wh\fp_{120}$;
$(2)$ by invariance of $\fm_{\ell-1,g}$ under reordering of
the boundary components and $|Sh_{1,\ell-2}|=\ell-1$;
and $(3)$ by definition of $\fm_{\ell-1,g}$. 
For $(4)$ we use the equality $\kkk_{g_{120}^i(\Gamma)}=\fm_{\Gamma}$
(for an arbitrary $i$) from Lemma~\ref{lm:disc} and the definition
$p_{120}^1=N^{12}\circ b_{120}^1$ from~\eqref{eq:p21012}. Here we
write out the double cyclization $N^{12}$ over the first two boundary
components of the glued graph $g_{120}^i(\Gamma)$ as a sum over
$j=3,\dots,s_1-1$ so that these boundary components have $\tilde
s_1=j-2$ and $\tilde s_2=s_1-j$ leaves, respectively, and then
cyclically permute the leaves on these components by $\rho\in\Z_{\wt
  s_1}\times\Z_{\wt s_2}$. 
Equality $(5)$ holds by Lemma~\ref{lem:b_{120}} and because 
taking the inverse induces a bijection on any group;
$(6)$ by the bijection~\eqref{eq:glue3};
$(7)$ by equation~\eqref{eq:action};
$(8)$ because the operation in the big round brackets 
is the identity on the quotient space $\wh E_\ell C$
(recall that $|Sh_{1,1,\ell-2}|=\ell(\ell-1)$)
and the expression to its right defines as an element of $\wh E_\ell C$;
$(9)$ by equation~\eqref{eq:action};
$(10)$ by the bijection~\eqref{eq:slab3};
and $(11)$ by independence of $\n_{\wh\Gamma}$ of the orientation of
the marked edge.

{\bf The sum of the three terms. }
Now we recombine the three summands in~\eqref{eq:MC-lg}, which
correspond to the three types of graphs in the decomposition
$$
   \RR_{\ell,g}^m = \RR_{\ell,g,1c}^m\amalg \RR_{\ell,g,1dc}^m\amalg \RR_{\ell,g,12}^m.
$$
Thus equations~\eqref{eq:summand1}, \eqref{eq:summand2} and~\eqref{eq:summand3}
sum up to
$$
   MC_{\ell,g} = \frac{1}{\ell!}
   \sum_{\wh\Gamma\in \RR_{\ell,g}^{m}}
   \n_{\wh\Gamma}. 
$$
We multiply this by $\ell!$ to obtain
\begin{align}\label{eq:3terms}
   \ell!MC_{\ell,g} 
   = \sum_{\wh\Gamma\in \RR_{\ell,g}^{m}}
   \n_{\wh\Gamma} 
   = \sum_{\Gamma\in \RR_{\ell,g}}
   \sum_{l\in \Edge(\Gamma)}\n_{\Gamma,l}
   = \sum_{\Gamma\in \RR_{\ell,g}}\sum_{l\in \Edge(\Gamma)}\n_{\Gamma-l}.
\end{align}
Here to see the second equality consider the projection forgetting 
the marked edge
$$
   \RR_{\ell,g}^{m}\stackrel{\pi_m}\longrightarrow \RR_{\ell,g}.
$$
Let a trivalent labelled graph $\Gamma$ represent an element 
$[\Gamma]$ in $\RR_{\ell,g}$. The natural map 
$$
  \Edge(\Gamma)\longrightarrow 
  \pi_m^{-1}([\Gamma]),\qquad l\mapsto [(\Gamma,l)].
$$
is clearly surjective. Since by Lemma~\ref{lem:no-auto1} there are
no nontrivial automorphisms of $\Gamma$, it is also injective and
yields the second equality.
The third equality follows from Corollary~\ref{cor:fulldiff}. 

Recall now the involution $\ol I$ on $\RR_{\ell,g}^{m}$ defined in
\S\ref{ss:op-graphs}. According to Lemma~\ref{lem:I-free} it has
no fixed points. 
Therefore, the terms in the last sum in~\eqref{eq:3terms} occur in pairs 
$$
   P(\Gamma,l) := \n_{\Gamma-l}+\n_{\ol I(\Gamma-l)}
$$ 
and it is enough to prove that each $P(\Gamma,l)$ equals zero.
We apply this expression to a tensor $\alpha$ of harmonic forms
corresponding to the $s$ leaves of $\Gamma$. By definition of
$\n_{\Gamma-l}$ we obtain 
$$
   (\n_{\Gamma-l}+\n_{I(\Gamma-l)})(\alpha) = (-1)^{\star}\left(\int_{\p^l\CC_\Gamma}
   \wt G^e(\alpha)+
   \int_{\p^{I(l)}\CC_{I(\Gamma)}}
   \wt G^e(\alpha)\right)
$$
with 
$$
\star:=n(\ell-1)+1+(s+1)(s+2)/2+P(\alpha)+(n-1)(\wt{se}_\Gamma+P_b(\alpha)).
$$ 
By~\cite[Lemma~12.10]{Cieliebak-Volkov-stringtop}, the last two
integrals do not depend on the extension of the labelling. Therefore, we can assume that
the extensions of labellings of $\Gamma$ and $I(\Gamma)$ are related
as in~\S\ref{ss:graph-duality}. (The sign in front of the integral
in~\cite[Lemma~12.10]{Cieliebak-Volkov-stringtop} does not
explicitly show up in the present context because 
it is built into the orientation of $\CC_\Gamma$, see~\eqref{eq:Delta3Gamma}.)
Therefore, the results of~\S\ref{ss:duality} apply
and the sum of the two integrals above vanishes because
by equation~\eqref{eq:signinvol} we have
$\p^l\CC_\Gamma=-\p^{I(l)}\CC_{I(\Gamma)}$ as oriented
manifolds. This proves $P(\Gamma,l)=0$ for each $(\Gamma,l)$, and
thus concludes the proof of the Maurer-Cartan equation $MC=0$. 

\subsection{Twisted homology equals cyclic cohomology}\label{ss:cyc-coh}

As above, we consider a harmonic subspace $\HH\subset\Om^*(M)$ of the
de Rham complex $\Om^*(M)$. We identify $\HH$ with the de Rham
cohomology $H^*(M)$ via the canonical projection. Let
$\fm=\{\fm_{\ell,g}\}_{\ell\ge 1,g\ge 0}$ be the Maurer-Cartan element
on the dual cyclic bar complex $B^{\rm cyc*}\HH[2-n]$ defined
in~\S\ref{ss:defMC}, and $\fp^\fm = \{\fp^\fm_{k,\ell,g}\}_{k,\ell\ge 1,g\ge 0}$
the corresponding twisted $\IBL_\infty$ structure. 
Its $(1,1,0)$ part defines a differential $\fp^\fm_{1,1,0}$ on $B^{\rm cyc*}\HH$. 

Let $HC_\lambda^*(A)$ denote the Connes version of cyclic cohomology
of an $A_\infty$-algebra $A$ (see~\cite{Loday} for the case of an
algebra, \cite{Cieliebak-Volkov-cyc} for the case of a DGA,
and~\cite{Cieliebak-Volkov-stringtop} for the general case).
The following proposition
completes the proof of Theorem~\ref{thm:existence-intro} from the
Introduction.  

\begin{prop}
The homology of $B^{\rm cyc*}\HH$ with respect to the differential
$\fp^\fm_{1,1,0}$ equals the Connes cyclic cohomology of the de Rham
DGA, 
$$
   H_*(B^{\rm cyc*}\HH,\fp^\fm_{1,1,0}) \cong HC_\lambda^*(\Om^*(M)).
$$
\end{prop}

\begin{proof}
By equation~\eqref{eq:dIBLtwist} and the vanishing of the differential
on $\HH$, the twisted differential is given by
$$
  \fp_{1,1,0}^\fm = \fp_{2,1,0}(\fm_{1,0},\cdot)
$$  
in terms of the $(1,0)$ part of the Maurer-Cartan element $\fm$. 
It is shown in~\cite[Proposition~12.3]{Cieliebak-Fukaya-Latschev} that
the components of 
$$
  \fm_{1,0}\in B^{\cyc*}\HH = \prod_{i\geq 1}B^{\cyc*}_i\HH
$$
define an $A_\infty$-structure on $\HH$ which we will also denote by $\fm_{1,0}$.
We need the following two facts:
\begin{enumerate}
\item The twisted differential $\fp_{1,1,0}^\fm$ on $B^{\rm cyc*}\HH$
  equals the dual of the Hochschild differential of the $A_\infty$-structure $\fm_{1,0}$. 
\item The $A_\infty$-structure $\fm_{1,0}$ on $\HH$ is obtained from the DGA
  structure on $\Om^*(M)$ by homotopy transfer as
  in~\cite{Kontsevich-Soibelman}, i.e.~there exists an $A_\infty$
  homotopy equivalence
$$
  G: (\HH,\fm_{1,0}) \to \Om^*(M). 
$$
\end{enumerate}
Fact (i) follows from equation~(12.6) in~\cite{Cieliebak-Fukaya-Latschev}.
Fact (ii) is clear on the picture level, but it involves a nontrivial
sign comparison for which we refer to~\cite{Cieliebak-Volkov-stringtop}.
Since the homology of $B^{\rm cyc*}\HH$ with respect to the dual 
Hochschild differential is the Connes version of cyclic cohomology,
Fact (i) gives 
$$
  H_*(B^{\rm cyc*}\HH,\fp^\fm_{1,1,0}) = HC_\lambda^*(\HH,\fm_{1,0}).
$$
Now it is a standard fact that an $A_\infty$ homotopy equivalence
induces an isomorphism on cyclic homology
(see~\cite{Cieliebak-Volkov-stringtop}), and therefore also on cyclic
cohomology by the universal coefficient theorem. Thus Fact (ii) gives an isomorphism
$$
  G^*: HC_\lambda^*(\Om^*(M)) \stackrel{\cong}\longrightarrow HC_\lambda^*(\HH,\fm_{1,0}),
$$
which together with the previous displayed equation proves the proposition.
\end{proof}

\section{Gauge equivalence}\label{sec:gauge}

In this section we prove Theorem~\ref{thm:uniqueness-intro} from the
Introduction.
As before, $M$ is a closed oriented manifold of dimension $n$. 
Recall from~\S\ref{ss:harmproj} the de Rham complex
$\bigl(\Om^*(M),d,\wedge)$ with its intersection pairing
$(\cdot,\cdot)$ and the cyclic pairing
$\la x,y\ra=(-1)^{\deg x}(x,y)$ (see~\ref{eq:defcycpair}). 
Note that the notion of orthogonality is the same for both
pairings. By ``dual'' we will mean dual with respect to the
cyclic pairing $\la\cdot,\cdot\ra$. 

Recall the notions of propagators $\wt G$ in the de Rham case
from $\S\ref{ss:blowupprop}$ (see e.g.~Proposition~\ref{prop:existsprop}). Recall 
from~\S\ref{sec:proofMC} that given a complement $\HH$ of $\im d$ in
$\ker d$ and a (not necessarily special) propagator 
$\wt G$ for $\HH$ we can construct a Maurer-Cartan (MC) element $\fm$ on $B^{\cyc *}\HH[2-n]$.
The goal of this section is to show that this MC element is 
independent of the choices of $\HH$ and $\wt G$ up to gauge
equivalence. Throughout this section, $I\subset \R$ will denote any
interval containing zero and $t\in I$ will denote the time
parameter. In the course of the proof we will have to consider forms
of degree $n-2$. For this to make sense in the case
$n=1$ as well, we stipulate that $\Om^{-1}(M)=0$.

In order to compare the MC elements we need to have them on the same 
footing. Therefore, we will consider them as MC elements on 
$B^{\cyc *}H^*(M)[2-n]$, where $H^*(M)$ is the de Rham 
cohomology of $M$. 

\subsection{Strategy}\label{ss:intro}

The proof of independence of choices splits into two cases. 

{\bf Case 1. }
Let $\HH_0$ and $\HH_1$ be two complements of $\im d$ in $\ker d$ as above,
and $\wt G_0$ and $\wt G_1$ respective (not necessarily special)
propagators. Assume that $\wt G_0$ and $\wt G_1$ are connected by a
smooth path $\{\wt G_t\}_{t\in [0,1]}$ of propagators (for some
possibly varying $\HH_t$) with the additional property that they solve
the Cauchy problem~\eqref{eq:diffeq} below. Then the corresponding MC
elements $\fm_0$ and $\fm_1$ are gauge equivalent.

{\bf Case 2. }
Let $\HH$ be a complement of $\im d$ in $\ker d$ as above, and $\wt G_0$ and 
$\wt G_1$ two propagators for $\HH$ such that 
\begin{equation}\label{eq:mu}
  \wt G_1-\wt G_0=d\mu
\end{equation}
for some $\mu\in \Om^{n-2}(\wt M^2)$. Then the corresponding MC
elements $\fm_0$ and $\fm_1$ are gauge equivalent. 

This section is organized as follows. In \S\ref{ss:compl} we introduce
the necessary language and notation. In \S\ref{ss:constrprim} we 
explain how to fix the necessary choices in order to take unique
primitives of exact forms on products $M^q$, $q\in\N$. In
\S\ref{ss:families} we construct families of propagators that solve
the Cauchy problem needed for Case 1. In \S\ref{ss:vanbilin}
we introduce the key property for propagators, {\em vanishing bilinear
  form}, which is a weakening of being special. It turns out that the
propagators in Case 1 are not necessarily special, but they have vanishing bilinear
form (Lemma~\ref{lem:vanishbilin}). This property is still
strong enough to imply that the difference of two propagators 
(for the same $\HH$) with vanishing bilinear form is always exact
(Lemma~\ref{lem:vanishexact}). This will be crucial for combining 
Cases 1 and 2 later in the proof of the main theorem. In
\S\ref{ss:defs} we construct the path $\bbb_t$ needed for the gauge
equivalence equation~\eqref{eq:gaugealg} in both cases. 
In \S\ref{ss:case1} we solve Case 1, and in \S\ref{ss:case2} we solve 
Case 2. In \S\ref{ss:mainres} we restate Theorem~\ref{thm:uniqueness-intro} 
in a more precise form (Theorem~\ref{thm:gauge}) and prove it. The
proof consists in joining $\fm_0$ and $\fm_1$ by a piecewise smooth path
composed of three smooth pieces; the middle one is taken care of by
Case 1, and the other two by Case 2.  

\subsection{Complements and propagators}\label{ss:compl}

Let $\{\HH_t\}_{t\in I}$ be a smooth family of subspaces of $\ker d$
complementing $\im d$ in $\ker d$. Let $C_t$ be a complement of the
space of closed forms $\ker d$ in the space of all forms $\Om^*(M)$
varying smoothly with $t$, so that
$$
  \Om^*(M) = \im d\oplus \HH_t\oplus C_t.
$$
We introduce the following notation. For $\alpha\in H^*(M)$ we denote
by $\alpha_t\in \HH_t$ the unique representative of $\alpha$ in
$\HH_t$. More generally, for $q\in\N$ and decomposable $\alpha\in
H^*(M)^{\otimes q}$ we denote by $\alpha_t\in \HH_t\otimes\dots\otimes\HH_t$ 
the unique representative of $\alpha$. Sometimes the image of 
$\alpha_t$ in $\Om^*(M^q)$ will also be denoted by $\alpha_t$ by a
slight abuse of language. For a constant family $\HH_t$ we will write
$\alpha$ for $\alpha_t$, again abusing the notation slightly.

Let $\{e_a\}$ be a basis of $H^*(M)$, and $\{e^a\}$ the dual basis of
$H^*(M)$ with respect to the pairing $\la\cdot,\cdot\ra$.
Let $\{e_{at}\}$ and $\{e^a_t\}$ be the corresponding bases of $\HH_t$.
Recall that $\wt G_t\in \Om^{n-1}(\wt M^2)$ is a propagator for $\HH_t$ iff it has
the correct symmetry with respect to the flip self-diffeomorphism 
$\wt\tau:\wt M^2\rightarrow \wt M^2$,
$$
  \wt\tau^*\wt G_t=(-1)^n\wt G_t,
$$
and is satisfies
$$
  d\wt G_t=e_{at}\times e^a_t.
$$
Recall that we use the Einstein summation convention. 
Recall also that the right hand side of the last displayed equation
depends only on the space $\HH_t$ itself and not on the basis. This
notation can be applied for a constant family $\HH_t$ as well.

\subsection{Constructing primitives}\label{ss:constrprim}

We work in the notation of \S\ref{ss:compl}. Let 
$$
  d_t^{-1}:=(d|_{C_t})^{-1}:\im d\longrightarrow C_t
$$
denote the inverse to the restriction of $d$ to $C_t$. This allows us
to take unique primitives of exact forms by requiring that the
primitives lie in $C_t$. We want to extend this to a certain class of
exact forms in $\Om^*(M^q)$ for $q\in\N$. For this, we extend 
the notion $C_t$ from a subspace of $\Om^*(M)$ to a subspace of
$\Om^*(M^q)$ as follows. Recall the cross product map
$$
  \times:\Om^*(M)^{\otimes q}\longrightarrow \Om^*(M^q).
$$
Define 
$$
  C_t^j:=\times(\HH_t \otimes\dots\otimes\HH_t\otimes C_t\otimes
  \HH_t\otimes\dots\otimes\HH_t)\subset \Om^*(M^q),\qquad j=1,\dots, q,
$$
where $C_t$ sits at position $j$. 

\begin{lemma}\label{lem:directsum}
The subspaces $\{C_t^j\}_{j=1}^q$ form a direct sum,
i.e.~$\sum_{j=1}^qx_j=0$ for $x_j\in C_t^j$ implies $x_j=0$ for all
$j$. The same is true for the family $\{dC_t^j\}_{j=1}^q$.
\end{lemma}

\begin{proof}
We start with the family $\{dC_t^j\}_{j=1}^q$ and proceed by induction on 
$q$. The case $q=1$ is trivial. Assume that the statement is true for some 
$q$, consider the product $M^{q+1}$ of $q+1$ copies of $M$, and assume that
\begin{equation}\label{eq:directsum}
  \sum_{j=1}^{q+1}x_j=0
\end{equation}
for some $x_j\in dC_t^j$.
Let $\pi:M^{q+1}\rightarrow M^q$ denote the projection forgetting the
last factor, and $\pi_{q+1}:M^{q+1}\rightarrow M$ the projection onto
the last factor. Recall that the pushforward $\pi_*$ on forms simply
means ``integrating out the last $M$ factor''. 
For each $\alpha\in\HH_t$, applying wedge with $\pi_{q+1}^*\alpha$ followed by 
$\pi_*$ to equation~\eqref{eq:directsum} yields 
\begin{equation*}
  \sum_{j=1}^{q}\pi_*(x_j\wedge \pi_{q+1}^*\alpha)=0.
\end{equation*}
Here the $j$-th term belongs to $C_t^j\subset\Om^*(M^q)$ for $1\leq
j\leq q$, and the term with $j=q+1$ is not present because $\pi_*(x_{q+1}\wedge
\pi_{q+1}^*\alpha)=0$. By induction hypothesis we conclude
$\pi_*(x_j\wedge \pi_{q+1}^*\alpha)=0$ for $j=1,\dots,q$. Since this
holds for all $\alpha\in\HH_t$, by nondegeneracy of the cup product on
cohomology of the last factor it implies $x_j=0$ for all
$j=1,\dots,q$, and therefore also $x_{q+1}=0$. 

For the family $\{C_t^j\}_{j=1}^q$, assume that $\sum_{j=1}^qx_j=0$
for some $x_j\in C_t^j$. Since the exterior
differential $d$ restricts as a monomorphism to each $C_t^j$, applying
$d$ to this equation brings us the case we just proved. 
\end{proof}

For any $q\in\N$ we define
\begin{equation}\label{eq:defCt}
  C_t:=\bigoplus_{j=1}^qC_t^j\subset \Om^*(M^q).
\end{equation}
By Lemma~\ref{lem:directsum} the exterior derivative on $\Om^*(M^q)$
restricts to a linear monomorphism
$$
  d|_{C_t}:C_t\longrightarrow\im(\times)\cap\im d
$$
and we introduce its inverse
$$
  d_t^{-1}:=(d|_{C_t})^{-1}:dC_t\longrightarrow C_t.
$$
Since $d$ commutes with the flip self-diffeomorphism $\tau:M^2\to M^2$
and restricts to $C_t$ as a monomorphism, the inverse $d_t^{-1}$ for
$q=2$ commutes with $\tau$ as well. 

\begin{lemma}\label{lem:2factors}
In the above setting, for all $a,b\in \HH_t$ and $\xi,\eta\in \im d$
we have 
$$
  \xi\times b+a\times\eta\in dC_t
$$
and 
\begin{equation}\label{eq:primprod}
d_t^{-1}(\xi\times b+a\times\eta)=d_t^{-1}(\xi)\times b+(-1)^aa\times d_t^{-1}(\eta).
\end{equation}
\end{lemma}

\begin{proof}
The right hand side of equation~\eqref{eq:primprod} belongs to $C_t$
by construction. The result follows by applying the exterior
derivative to both sides of~\eqref{eq:primprod}. 
\end{proof}

We apply this as follows.
Since the cohomology class $[e_{at}]=e_a\in H^*(M)$ is constant in
$t$, the corresponding time derivatives are exact, and similarly for $e^a_t$. 
We set
\begin{equation}\label{eq:lambda}
  \lambda_{at}:=d_t^{-1}\ddt e_{at},\qquad \lambda^a_t:=d_t^{-1}\ddt e^a_t
\end{equation}
to get
\begin{equation}\label{eq:prime}
  d_t^{-1}\ddt(e_{at}\times e^a_t)=\lambda_{at}\times e^a_t+(-1)^{e_a}e_{at}\times \lambda^a_t.
\end{equation}

\begin{lemma}\label{lem:sym}
For all $t\in I$ we have 
\begin{equation}\label{eq:symmprim}
\tau^*(\lambda_{at}\times e^a_t+(-1)^{e_a}e_{at}\times \lambda^a_t)=
(-1)^n(\lambda_{at}\times e^a_t+(-1)^{e_a}e_{at}\times \lambda^a_t).
\end{equation}
\end{lemma}

\begin{proof}
Recall from Lemma~\ref{lem:Poincare} that
$$
\tau^*(e_{at}\times e^a_t)=(-1)^ne_{at}\times e^a_t.
$$
Recall also that $d_t^{-1}$ commutes with $\tau$. Now we apply $\ddt$
and then $d_t^{-1}$ to both sides of the last displayed equation 
to get equation~\eqref{eq:symmprim}.
\end{proof}

\subsection{Constructing propagators via the Cauchy problem}\label{ss:families}

In this subsection we construct smooth families of propagators using
the Cauchy (initial value) problem. We retain the setup of
\S\ref{ss:compl} and \S\ref{ss:constrprim}.  

\begin{lemma}\label{lem:Cauchyprop}
Let $\wt K$ be a propagator for $\HH_0=\HH_t|_{t=0}$.
Define a smooth family $\{\wt G_t\}_{t\in I}$ of 
$(n-1)$-forms on the blow-up $\wt M^2$ by the following Cauchy problem:
\begin{equation}
\begin{cases}\label{eq:diffeq}
  \ddt\wt G_t=\lambda_{at}\times e^a_t+(-1)^{e_a}e_{at}\times \lambda^a_t,\\
  \wt G_0=\wt K.
\end{cases}
\end{equation}
Then $\wt G_t$ is a propagator for $\HH_t$ for all $t\in I$.
\end{lemma}

\begin{proof}
The $\tau$-symmetry of $\wt G_t$ follows from
equation~\eqref{eq:symmprim} and the $\tau$-symmetry of the initial
condition $\wt K$ by uniqueness of solutions of the Cauchy problem. To
obtain the desired condition on $d\wt G_t$ we apply $d$ to both sides
of the first equation in~\eqref{eq:diffeq} and use
equation~\eqref{eq:prime} to get 
$$
\ddt(d\wt G_t-e_{at}\times e^a_t)=0.
$$
Therefore, $d\wt G_t-e_{at}\times e^a_t$ is constant in $t$, and
since it vanishes for $t=0$ it vanishes for all $t$.
\end{proof}

\subsection{Vanishing bilinear form}\label{ss:vanbilin}

Let $\pi:\wt M^2\to M^2$ denote the natural blow-down map.
Let $\HH$ be a complement of $\im d$ in $\ker d$ and $\wt G$ an
associated propagator. We say that $\wt G$ has {\em vanishing bilinear
  form} if for every $\alpha,\beta\in \HH$ we have  
\begin{equation}\label{eq:vanish1}
  \int_{\wt M^2}\wt G\wedge\pi^*(\alpha\times\beta)=\int_{M^2}G(x,y)\wedge\alpha(x)\wedge
  \beta(y)=0.
\end{equation}
The main reason such propagators are interesting is the following result.

\begin{lemma}\label{lem:vanishexact}
(a) Each special propagator $\wt G$ has vanishing bilinear form.\\
(b) Let $\wt G_1$ and $\wt G_2$ be two propagators for $\HH$ with
vanishing bilinear forms. Then their difference is exact,
$$
  \wt G_1-\wt G_2=d\alpha
$$
for some $(n-2)$-form $\alpha$ on $\wt M^2$.
\end{lemma}

\begin{proof}
Part (a) follows directly from the orthogonality
relation~\eqref{eq:keyortho} for a special propagator.
For (b), denote $N:=\wt M^2$. Since $d\wt G_1-d\wt G_2=0$, the difference
$\eta:=\wt G_1-\wt G_2$ is closed and thus represents a cohomology
class $[\eta]\in H^{n-1}(N)$. By the Universal Coefficient Theorem it
is enough to show that the evaluation of $[\eta]$ on $H_{n-1}(N)$ vanishes. 
Poincar\'e-Lefschetz duality $H_*(N)\cong H^{2n-*}(N,\p N)$ for
$*=n-1$ allows us to replace evaluation on $H_{n-1}(N)$ by taking the
cup product with elements of $H^{n+1}(N,\p N)$ followed by 
evaluating on the fundamental class $[N]\in H_{2n}(N,\p N)$. That is, 
we have to show 
$$
  ([\eta]\cup z)[N]=0
$$
for each $z\in H^{n+1}(N,\p N)$.
Any $z\in H^{n+1}(N,\p N)$ can be represented by a closed form $\wt
\xi\in \Om^{n+1}(\wt M^2)$ vanishing near the boundary, and any such
$\wt \xi$ is the pullback under $\pi$ of a closed form on $M^2$
vanishing near the diagonal. Therefore, the last displayed equation is implied by 
\begin{equation}\label{eq:cup}
  \int_{\wt M^2}\eta\wedge \pi^*\xi=0
\end{equation}
for each closed form $\xi$ on $M^2$ that vanishes near the
diagonal. By the K\"unneth formula, for each such $\xi$ there exist
finite collections $\{\alpha_j\}\subset \HH$ and
$\{\beta^j\}\subset\HH$ and a form $\phi\in \Om^n(M^2)$
 vanishing near the diagonal such that
$$
  \xi=\alpha_j\wedge \beta^j+d\phi.
$$
Now the left hand side of equation~\eqref{eq:cup} with
$\pi^*(\alpha_j\wedge \beta^j)$ in place of $\pi^*\xi$ vanishes
because both $\wt G_1$ and $\wt G_2$ have vanishing bilinear forms; 
with $\pi^*d\phi$ in place of $\pi^*\xi$ 
it becomes
$$
  \int_{\wt M^2}\eta\wedge \pi^*d\phi = (-1)^{n-1}\int_{\p\wt
    M^2}\eta\wedge \pi^*\phi = 0
$$
because $\phi=0$ near the diagonal. 
\end{proof}

The next lemma tells us how to obtain propagators with vanishing bilinear form.

\begin{lemma}\label{lem:vanishbilin}
Let a family $\wt G_t$ of propagators be defined as in Lemma~\ref{lem:Cauchyprop}. 
Assume in addition that the initial propagator $\wt G_0=\wt K$ has vanishing bilinear form 
and that the complement $C_t$ to $\ker d$ in $\Om^*(M)$ from
\S\ref{ss:compl} is orthogonal to $\HH_t$, that is $C_t\perp \HH_t$. 
Then for each $t\in I$ the propagator $\wt G_t$ has vanishing bilinear form.
\end{lemma}

\begin{proof}
Let $t_0\in I$ be any time moment. Observe that any $\eta\in
\HH_{t_0}$ can be included in a smooth family as follows. Let
$[\eta]\in H^*(M)$ be the cohomology class represented by $\eta$, and
for every $t\in I$ let $x_t\in \HH_t$ be the unique element of $\HH_t$
with $[x_t]=[\eta]$. Then $\{x_t\}_{t\in I}$ is a smooth family with $x_{t_0}=\eta$. 

Therefore, it suffices to consider two smooth families $\alpha_t,\beta_t\in \HH_t$ 
with $[\alpha_t]=const_1\in H^*(M)$, $[\beta_t]=const_2\in H^*(M)$ and show vanishing of
$$
  B_t:=\int_{\wt M^2}\wt G_t\wedge(\alpha_t\times\beta_t)
$$
for all $t\in I$. Since $\wt G_0=\wt K$ has vanishing bilinear form,
it suffices to show vanishing of the time derivative of $B_t$. Since
the cohomology classes represented by $\alpha_t$ and $\beta_t$ are
constant in $t\in I$, their time derivatives are exact and we can apply 
Lemma~\ref{lem:2factors} to get
\begin{equation*}
  \ddt(\alpha_t\times\beta_t)=d(d_t^{-1}\dot\alpha_t\times\beta_t+(-1)^{\alpha_t}\alpha_t\times d_t^{-1}\dot\beta_t).
\end{equation*}
Using this and equation~\eqref{eq:diffeq}, we obtain 
\begin{align*}
  \ddt B_t
  &= \int_{\wt M^2}\left(\ddt\wt G_t\right)\wedge(\alpha_t\times\beta_t) +
  \int_{\wt M^2}\wt G_t\wedge\ddt(\alpha_t\times\beta_t)\cr
  &= \int_{M^2}(\lambda_{at}\times e^a_t+(-1)^{e_a}e_{at}\times \lambda^a_t)  \wedge(\alpha_t\times\beta_t)\cr
  & \ \ \ + \int_{\wt M^2}\wt G_t\wedge  
  d(d_t^{-1}\dot\alpha_t\times\beta_t+(-1)^{\alpha_t}\alpha_t\times d_t^{-1}\dot\beta_t).
\end{align*}
We continue with the last expression which is the sum of two integrals.
In the first integral there are two summands. The first summand
involves pairing $\lambda_{at}$ with $\alpha_t$, and the second
summand involves pairing $\lambda^a_t$ with $\beta_t$. Both pairings
vanish by orthogonality $\HH_t\perp C_t$. We rewrite the second
integral as
\begin{align*}
  &(-1)^{n-1}\int_{\wt M^2}\wt G_t\wedge 
  d(d_t^{-1}\dot\alpha_t\times\beta_t+(-1)^{\alpha_t}\alpha_t\times
  d_t^{-1}\dot\beta_t)\cr 
  &= \int_{\wt M^2}d(\wt G_t\wedge
  (d_t^{-1}\dot\alpha_t\times\beta_t+(-1)^{\alpha_t}\alpha_t\times
  d_t^{-1}\dot\beta_t))\cr
  &\ \  \ -\int_{\wt M^2}d\wt G_t\wedge 
  (d_t^{-1}\dot\alpha_t\times\beta_t+(-1)^{\alpha_t}\alpha_t\times d_t^{-1}\dot\beta_t)\cr
  &=: I_1+I_2.
\end{align*}
We apply Stokes and then integration over the fibre 
with $\wt G_t$ to evaluate the first integral:
\begin{align*}
  I_1
  &= \int_{\wt M^2}d\Bigl(\wt G_t\wedge
  (d_t^{-1}\dot\alpha_t\times\beta_t+(-1)^{\alpha_t}\alpha_t\times
  d_t^{-1}\dot\beta_t)\Bigr) \cr
  &=
  (-1)^n\int_M(d_t^{-1}\dot\alpha_t\wedge\beta_t+(-1)^{\alpha_t}\alpha_t\wedge
  d_t^{-1}\dot\beta_t) 
  =0,
\end{align*}
where both summands vanish because of orthogonality $\HH_t\perp C_t$ again.
We use the expression for $d\wt G_t$ to evaluate the second integral:
\begin{align*}
  -I_2
  &= \int_{\wt M^2}d\wt G_t\wedge 
  (d_t^{-1}\dot\alpha_t\times\beta_t+(-1)^{\alpha_t}\alpha_t\times d_t^{-1}\dot\beta_t)\cr
  &= \int_{M^2}(e_{at}\times e^a_t)\wedge 
  (d_t^{-1}\dot\alpha_t\times\beta_t+(-1)^{\alpha_t}\alpha_t\times d_t^{-1}\dot\beta_t)
  =0.
\end{align*}
Here the first summand involves pairing $e_{at}$ and
$d_t^{-1}\dot\alpha_t$, and the second one pairing $e^a_t$ with
$d_t^{-1}\dot\beta_t$. Both pairings vanish because of $\HH_t\perp C_t$.
\end{proof}

\subsection{Defining the basic integrals}\label{ss:defs}

Let $\Gamma\in \RR_{\ell,g}$ be a trivalent ribbon graph. We use
the terminology concerning labellings, extended labellings, 
marked edges etc from \S\ref{ss:basiccomb}. Let
as usual $s$ denote the number of leaves of $\Gamma$,
$e$ the number of edges, $k$ the number of vertices,
and $\beta\in \Om^*(M)^{\otimes s}$ a decomposable tensor. 
Let $\HH_t\subset \ker d$ be a family as in \S\ref{ss:compl}, and
$\wt G_t$ a smooth family of propagators for $\HH_t$. 
In analogy to~\S\ref{sec:opforms} we define (using slightly different notation)
\begin{itemize}
\item $G_{\Gamma t}:=G_t\times\dots\times G_t$ on $(M^2)^e$;
\item $\wt G_{\Gamma t}:=\wt G_t\times\dots\times\wt G_t$ on 
$(\widetilde M^2)^e$;
\item $G_{\Gamma t}(\beta):=G_{\Gamma t}\times\beta$ on $X_\Gamma$; 
\item $\wt G_{\Gamma t}(\beta):=\wt G_{\Gamma t}\times\beta$ on $\wt X_\Gamma$. 
\end{itemize}
A choice of an extended labelling for $\Gamma$ allows us to define the integrals
$$
  I_{\Gamma t}(\beta) := \int_{\Delta_{3\Gamma}}G_{\Gamma t}(\beta) =
  (-1)^{R_\Gamma+(n-1)\eta_3(\Gamma)}\int_{\Delta_3^k}R_\Gamma^*G_{\Gamma t}(\beta).
$$
Given a marked edge $l\in \Edge(\Gamma)$ and a form 
$\gamma\in \Om^*(\wt M^2)$, the expression 
$$
  G_t\times\dots\times \gamma\times\dots\times G_t
$$
will denote the cross product of $e$ forms, with $\gamma$ sitting at
the position that corresponds to $l$ and all the others being
$G_t$. To write down the integrals we need to pick an extension of the
labelling, but the result ultimately depends only on the labelling
itself. This can be proved in complete parallel to \S\ref{sec:opforms}.
We will skip some details in the sign computations below. 
The notation $s(\Gamma)$ will stand for a suitable sign exponent that
depends on $\Gamma$ (and possibly on the labelling/extended labelling,
interior edge etc) and may change from line to line.
The notation $s(\Gamma,\beta)$ will stand for a
sign exponent that depends on $\Gamma$ 
(with possible decorations) and $\beta$. This notation allows
us to write the MC element corresponding to $\HH_t$
and $\wt G_t$ as follows. For decomposable
$\beta\in \Om^*(M)^{\otimes s}$ we define
\begin{equation}\label{eq:mGammat}
   \fm_{\Gamma t}(\beta):=(-1)^{s(\Gamma,\beta)}I_{\Gamma t}(\beta). 
\end{equation}
The definition of $\fm_{\Gamma t}$ extends to elements of 
$C_t\subset \Om^*(M^s)$ by linearity. Now for 
$\alpha\in
H^*(M)^{\otimes s}$ we use the corresponding
$\alpha_t\in\HH_t^{\otimes s}$ to define
\begin{equation}\label{eq:m}
  \fm_{(\ell,g)t}(\alpha):=\frac{1}{\ell !}\sum_{\Gamma\in R_{\ell,g}}
  \fm_{\Gamma t}(\alpha_t). 
\end{equation}
By~\S\ref{sec:proofMC}, this defines a MC element on $B^{\cyc
  *}H^*(M)[2-n]$ for all $t\in I$. We now define the smooth family
$\bbb_t$ needed for the gauge equivalence equation~\eqref{eq:gaugeIBL} 
in the two cases which will be treated in \S\ref{ss:case1} and \S\ref{ss:case2}.

{\bf Case 1 (Cauchy problem). }
In this case we assume that $\wt G_t$ solves the Cauchy
problem~\eqref{eq:diffeq} and define
\begin{equation}\label{eq:bspeccond}
  \bbb_{\Gamma t}(\alpha):=\fm_{\Gamma t}(d_t^{-1}\ddt\alpha_t).
\end{equation}
Note that $C_t$ implicitly enters the definition via $d_t^{-1}$.

{\bf Case 2 (exact difference). }
In this case we assume that the family $\HH_t$ is
constant, pick some $\mu\in \Om^{n-2}(M)$ satisfying~\eqref{eq:mu}, 
and define 
\begin{equation}\label{eq:exactdiff}
\bbb_{\Gamma t}(\alpha):=
\sum_{l\in \Edge(\Gamma)}(-1)^{s(\Gamma,\alpha)}\int_{\Delta_{3\Gamma}}
G_t\times\dots\times\mu\times\dots\times G_t\times\alpha.
\end{equation}
In both cases $\bbb_{\Gamma t}$ enjoys an equivariance property with
respect to relabellings analogous to equation~\eqref{eq:action}: 
Let $\Gamma$ be a labelled graph with $s$ exterior flags, and $l$ an
edge of $\Gamma$. Let $\eta\in S(\bs)$ describe a change of
labelling (altering items (i) and (ii) in Definition~\ref{def:labelling}).
Then
\begin{equation}\label{eq:actiongauge}
  \bbb_{\Gamma\eta t}=\eta^{-1}\bbb_{\Gamma t}.
\end{equation}
In both cases the element $\bbb_{(\ell,g)t}$ is defined 
as a sum of over isomorphism classes of labelled graphs,
\begin{equation}\label{eq:b}
\bbb_{(\ell,g)t}:=\frac{1}{\ell !}\sum_{\Gamma\in R_{\ell,g}}\bbb_{\Gamma t}.
\end{equation}

\subsection{Case 1 (Cauchy problem)}\label{ss:case1}

Recall that $\HH_t\subset \ker d$ is a smooth family of subspaces
complementing $\im d$ in $\ker d$, and $C_t\subset\Om^*(M^q)$ is a
smooth family of subspaces defined by~\eqref{eq:defCt} for some $q$
which will always be clear from the context.
Recall the correspondence between $\alpha\in H^*(M)$ and
$\alpha_t\in\HH_t$, and in particular the bases $e_a,e^a$ of $H^*(M)$
and $e_{at},e^a_t$ of $\HH_t$.

Let $\wt G_t$ be a smooth family of propagators for $\HH_t$ solving
the Cauchy problem~\eqref{eq:diffeq}. Define the family
$\fm_{(\ell,g)t}$ of MC elements by formulas~\eqref{eq:mGammat}
and~\eqref{eq:m}, and the family $\bbb_{(\ell,g)t}$ by formulas~\eqref{eq:bspeccond} and~\eqref{eq:b}. 
The proof of the following result will occupy the rest of this subsection.

\begin{lemma}\label{lem:case1}
In the above setup the families $\fm_{(\ell,g)t}$ and $\bbb_{(\ell,g)t}$
satisfy equation~\eqref{eq:gaugealg}. In particular, all the MC elements
$\fm_t$ are gauge equivalent to each other.
\end{lemma}

{\bf Useful identities.}
We collect some useful integral identities for the proof
of Lemma~\ref{lem:case1}. First, we compute
\begin{align*}
&\sum_{\Gamma\in \RR_{\ell,g}}\int_{\Delta_{3\Gamma}}d(G_{\Gamma t}\times d_t^{-1}\ddt\alpha_t)\cr
=&\sum_{\Gamma\in \RR_{\ell,g}}\int_{\CC_\Gamma}d(\wt G_{\Gamma t}\times d_t^{-1}\ddt\alpha_t)\cr
=&\sum_{\Gamma\in \RR_{\ell,g}}\int_{\p\CC_\Gamma}\wt G_{\Gamma t}\times d_t^{-1}\ddt\alpha_t\cr
=&\sum_{\Gamma\in \RR_{\ell,g}}\int_{\p^{\rm main}\CC_\Gamma}\wt G_{\Gamma t}\times d_t^{-1}\ddt\alpha_t+
\sum_{\Gamma\in \RR_{\ell,g}}\int_{\p^{\rm hidden}\CC_\Gamma}\wt G_{\Gamma t}\times d_t^{-1}\ddt\alpha_t\cr
=&0
\end{align*}
Here in the last line the first sum vanishes by the duality argument at the end of 
\S\ref{ss:maincompute}, and each integral in the second sum vanishes by
Theorem~\ref{thm:key-vanish}. This implies
\begin{equation}\label{eq:vanish}
\begin{aligned}
&\sum_{\Gamma\in \RR_{\ell,g}}\int_{\Delta_{3\Gamma}}G_{\Gamma t}\times d\circ d_t^{-1}\ddt\alpha_t\\
&=(-1)^{\star}\sum_{\Gamma\in \RR_{\ell,g}}\sum_{l\in \Edge(\Gamma)} 
\int_{\Delta_{3\Gamma}}G_t\times\dots\times dG_t\times\dots\times G_t\times d_t^{-1}\ddt\alpha_t
\end{aligned}
\end{equation}
with the sign exponent
$$
  \star=(n-1)(n(l)-e)+n.
$$ 
Here $e$ is the number of edges of $\Gamma$ and $n(l)$ is the number
of the marked edge in the order of edges. To see this, we apply the
Leibniz rule to $d(G_{\Gamma t}\times d_t^{-1}\ddt\alpha_t)$ and note
that the term with $dG_t$ at the $n(l)$-th position comes with the
sign $(-1)^{(n(l)-1)(n-1)}$, while the term $G_{\Gamma t}\times d\circ
d_t^{-1}\ddt\alpha_t$ comes with the sign $(-1)^{e(n-1)}$.

We will need one more identity.

\begin{lemma}
For any $\alpha\in H^*(M)^{\otimes q}$ we have
\begin{equation}\label{eq:fulldifft}
\left(\ddt G_t\right)\times \alpha_t+(-1)^ndG_t\times d_t^{-1}\ddt \alpha_t=d_t^{-1}\ddt
(e_{at}\times e^a_t\times\alpha_t).
\end{equation}
\end{lemma}

\begin{proof}
The explicit expressions for $\ddt G_t$ from
equation~\eqref{eq:diffeq} and $dG_t$ show that the left hand side
belongs to $C_t$. Therefore, it is enough to apply $d$ to the left
hand side and compute  
\begin{align*}
d\left(\ddt G_t\times \alpha_t+(-1)^ndG_t\times d_t^{-1}\ddt \alpha_t\right)
{=}
&\ddt dG_t\times \alpha_t+dG_t\times\ddt\alpha_t\cr
{=}
&\ddt (dG_t\times \alpha_t)\cr
{=}
&\ddt (e_{at}\times e^a_t\times \alpha_t).
\end{align*}
Since the last expression is the image under $d$ of the right hand
side of equation~\eqref{eq:fulldifft}, this proves the lemma.
\end{proof}


{\bf Homotopy computation.}
In the following computation we use the notation from~\S\ref{ss:op-graphs}.
For a marked graph $\wh\Gamma$, we denote by $c(\wh\Gamma)$ the result of
cutting $\wh\Gamma$ at its marked edge $l$.
For decomposable $\alpha\in H^*(M)^{\otimes s}$ we compute
\begin{align*}
&\ddt\fm_{(\ell,g)t}(\alpha)\cr
&\stackrel{(1)}
{=}\frac{1}{\ell!}\sum_{\Gamma\in \RR_{\ell,g}
}(-1)^{s(\Gamma,\alpha)}
\left(\int_{\Delta_{3\Gamma}}\ddt G_{\Gamma t}\times\alpha_t+
\int_{\Delta_{3\Gamma}}G_{\Gamma t}\times d\circ d_t^{-1}\ddt\alpha_t\right)\cr
&\stackrel{(2)}
{=}\frac{1}{\ell!}\sum_{\Gamma\in \RR_{\ell,g}}\sum_{l\in \Edge(\Gamma)}
\bigg((-1)^{s(\Gamma,\alpha)}\int_{\Delta_{3\Gamma}}G_t\times\dots\times\ddt G_t\times\dots\times G_t\times\alpha_t\cr
&\ \ \ +(-1)^{s(\Gamma,\alpha)+(n-1)(n(l)-e)+n}\int_{\Delta_{3\Gamma}}G_t\times\dots\times dG_t\times\dots\times G_t\times d_t^{-1}\ddt\alpha_t\bigg)\cr
&\stackrel{(3)}
{=}\frac{1}{2\ell!}\sum_{\wh\Gamma\in 
\RR_{\ell,g}^{om}}
\bigg((-1)^{s(\wh\Gamma,\alpha)}  \int_{\Delta_{3\wh\Gamma}}G_t\times\dots\times\ddt G_t\times\dots\times G_t\times\alpha_t\cr
&\ \ \ +(-1)^{s(\wh\Gamma,\alpha)+(n-1)(n(l)-e)+n}\int_{\Delta_{3\wh\Gamma}}G_t\times\dots\times dG_t\times\dots\times G_t\times d_t^{-1}\ddt\alpha_t\bigg).
\end{align*}
Here
$(1)$ follows from the definition of $\fm_{(\ell,g)t}$, the Leibniz
rule for $\ddt$, and the definition of $d_t^{-1}$; 
$(2)$ uses the Leibniz rule for the first integral and
equation~\eqref{eq:vanish} for the second one; 
and for $(3)$ we unite the double sum into one sum over o-marked graphs,
where the factor of $2$ compensates for passing from ``marked'' to ``o-marked''.

Our next goal is to manipulate $2\ell !$ times the last expression, namely
\begin{equation}\label{eq:zeta}
\zeta(\alpha):=\sum_{\wh\Gamma\in \RR_{\ell,g}^{om}}X_{\wh\Gamma}(\alpha),
\end{equation}
where 
\begin{align*}
  X_{\wh\Gamma}(\alpha):=
  &(-1)^{s(\wh\Gamma,\alpha)}\int_{\Delta_{3\wh\Gamma}}G_t\times\dots\times\ddt G_t\times\dots\times G_t\times\alpha_t\cr
  &+(-1)^{s(\wh\Gamma,\alpha)+
  (n-1)(n(l)-e)+n}\int_{\Delta_{3\wh\Gamma}}G_t\times\dots\times dG_t\times\dots\times G_t\times d_t^{-1}\ddt\alpha_t.
\end{align*}
This enjoys an equivariance property
analogous to~\eqref{eq:action} with respect to $\eta\in S(\bs)$,
\begin{equation}\label{eq:X}
  X_{\wh\Gamma\eta}=\eta^{-1}X_{\wh\Gamma}.
\end{equation}
Recall from~\eqref{eq:R-disjoint-union} and~\eqref{eq:ell12ell} the
disjoint union decomposition  
\begin{equation}\label{eq:disj}
\RR_{\ell,g}^{om} =
\RR_{\ell,g,1c}^{om}\amalg
   \coprod_{\substack{\ell_1+\ell_2=\ell+1 \\ g_1+g_2=g}}\RR_{\ell_1,\ell_2,g_1,g_2,1dc}^{om} 
     \amalg \RR_{\ell,g,12}^{om}.
\end{equation}
To $\wh\Gamma\in \RR_{\ell,g}^{om}$ we associate the set of shuffles
\begin{equation*}
Sh(\wh\Gamma) := \begin{cases}
Sh_{1,\ell-1},\,\, &\wh\Gamma\in \RR_{\ell,g,1c}^{oms}
\,\, 
\\
Sh_{1,\ell_1-1,\ell_2-1} ,\,\, &\wh\Gamma\in 
\RR_{\ell_1,\ell_2,g_1,g_2,1dc}^{oms}  
\\ 
Sh_{1,1,\ell-2},\,\, &\wh\Gamma\in 
\RR_{\ell,g,12}^{oms}
\end{cases}
\end{equation*}
and the group of cyclic permutations
\begin{equation*}
G_{cyc}(\wh\Gamma) := \begin{cases}
\Z_{s_1},\,\, &\wh\Gamma\in \RR_{\ell,g,1c}^{oms}
\,\, 
\\
\Z_{s_1} ,\,\, &\wh\Gamma\in 
\RR_{\ell_1,\ell_2,g_1,g_2,1dc}^{oms}  
\\ 
\Z_{s_1}\times\Z_{s_2},\,\, &\wh\Gamma\in 
\RR_{\ell,g,12}^{oms}
\end{cases}
\end{equation*}
Here $s_1$ and $s_2$ are the numbers of leaves on the first two
boundary components of $\wh\Gamma$. Note that these sets of permutations
depend on the graph $\wh\Gamma$ only through its ``type'' with respect to
the decomposition~\eqref{eq:disj}. The cardinalities of the shuffle
sets are 
\begin{equation}\label{eq:cardshuffle}
\begin{cases}
|Sh_{1,\ell-1}|=\ell,
\\
|Sh_{1,\ell_1-1,\ell_2-1}|=\frac{\ell !}
{(\ell_1-1)!(\ell_2-1)!},
\\ 
|Sh_{1,1,\ell-2}|=\ell(\ell-1).
\end{cases}
\end{equation}

Now we can rewrite $\zeta$ from~\eqref{eq:zeta} as
\begin{align*}
  \zeta \stackrel{(1)}{=}& \sum_{\substack{\wh\Gamma\in
      \RR_{\ell,g}^{oms}\\ (\tau,\kappa)\in Sh(\wh\Gamma)\times G_{cyc}(\wh\Gamma)}}
  X_{(\wh\Gamma\kappa^{-1})\tau}\cr
  \stackrel{(2)}{=}&
  \sum_{\wh\Gamma\in \RR_{\ell,g,1c}^{oms}}\Bigl(\sum_{\tau\in Sh(\wh\Gamma)}\tau^{-1}\Bigr)\sum_{\kappa \in G_{cyc}(\wh\Gamma)}
  X_{\wh\Gamma\kappa^{-1}} +
  \sum_{\wh\Gamma\in \RR_{\ell,g,12}^{oms}}\Bigl(\sum_{\tau\in Sh(\wh\Gamma)}\tau^{-1}\Bigr)\sum_{\kappa \in G_{cyc}(\wh\Gamma)}
  X_{\wh\Gamma\kappa^{-1}}  \cr
&
  + \sum_{\substack{\ell_1+\ell_2=\ell+1 \\ g_1+g_2=g}} \sum_{\wh\Gamma\in 
  \RR_{\ell_1,\ell_2,g_1,g_2,1dc}^{oms}} \Bigl(\sum_{\tau\in
    Sh(\wh\Gamma)}\tau^{-1}\Bigr) \sum_{\kappa \in G_{cyc}(\wh\Gamma)} X_{\wh\Gamma\kappa^{-1}}
\cr
\stackrel{(3)}{=}&
\sum_{\substack{\wh\Gamma\in \RR_{\ell,g}^{oms}
\\\kappa\in G_{cyc}(\wh\Gamma)}}|Sh(\wh\Gamma)|\,
X_{\wh\Gamma\kappa^{-1}}.
\end{align*}
Here $(1)$ follows from the bijections~\eqref{eq:slab1},\,\eqref{eq:slab2} 
and~\eqref{eq:slab3} applied to the first, second and third set in the
disjoint union~\eqref{eq:disj};
for $(2)$ we split the sum according to~\eqref{eq:disj} and apply
equation~\eqref{eq:X} with $\eta:=\tau$;
and for $(3)$ we regard each sum following the expression in big round
brackets as an element of the quotient space $\wh E_\ell C$, observe
that the corresponding expression in big round brackets acts as
multiplication by the cardinality of the corresponding set of shuffles,
and assemble the three summands back according to~\eqref{eq:disj}.
Using equation~\eqref{eq:X} with $\eta:=\kappa^{-1}$, we therefore get 
\begin{equation}\label{eq:key}
  \zeta(\alpha) = \sum_{\substack{\wh\Gamma\in \RR_{\ell,g}^{oms}
  \\\kappa\in G_{cyc}(\wh\Gamma)}}|Sh(\wh\Gamma)|\,X_{\wh\Gamma}(\kappa\alpha).
\end{equation}
This allows us to continue
\begin{align*}
&\ddt\fm_{(\ell,g)t}(\alpha)\cr
&\stackrel{(KEY)}{=} \sum_{\substack{\wh\Gamma\in \RR_{\ell,g}^{oms}
\\\kappa\in G_{cyc}}} \frac{|Sh(\wh\Gamma)|}{2\ell!}
\bigg((-1)^{s(\wh\Gamma,\kappa\alpha)}\int_{\Delta_{3\wh\Gamma}}G_t\times\dots\times\ddt G_t\times\dots\times G_t\times\kappa\alpha_t\cr
&\ \ \ \ \ \ \ \ 
+(-1)^{s(\wh\Gamma,\kappa\alpha)+(n-1)(n(l)-e)+n}\int_{\Delta_{3\wh\Gamma}}G_t\times\dots\times dG_t\times\dots\times G_t\times d_t^{-1}\ddt\kappa\alpha_t\bigg)\cr
&\stackrel{(1)}{=}\frac{1}{2\ell!}
\sum_{\substack{\wh\Gamma\in \RR_{\ell,g}^{oms}\\\kappa\in G_{cyc}}}
(-1)^{s(\wh\Gamma,\kappa\alpha)
+(n-1)(n(l)-e)}
|Sh(\wh\Gamma)|
\int_{\Delta_{3c(\wh\Gamma)}}G_t\times\dots\times G_t \\
&\ \ \ \ \ \ \ \ \ \ \  \ \ \ \ \ \ \times\left(\ddt G_t\times\kappa\alpha_t+(-1)^ndG_t\times d_t^{-1}\ddt\kappa\alpha_t\right)\cr
&\stackrel{(2)}{=}\frac{1}{2\ell!}
\sum_{\substack{\wh\Gamma\in \RR_{\ell,g}^{oms}
\\\kappa\in G_{cyc}}}
(-1)^{s(\wh\Gamma,\kappa\alpha)}|Sh(\wh\Gamma)|
\int_{\Delta_{3c(\wh\Gamma)}}G_t\times\dots\times G_t\times d_t^{-1}\ddt \left(e_{at}\times e_t^a\times\kappa\alpha_t\right)
\cr
&\stackrel{(3)}{=}\frac{1}{2\ell!}
\sum_{\substack{\wh\Gamma\in \RR_{\ell,g}^{oms}
\\\kappa\in G_{cyc}}}
(-1)^{s(\wh\Gamma,\kappa\alpha)}|Sh(\wh\Gamma)|
\int_{\Delta_{3c(\wh\Gamma)}}G_t\times\dots\times G_t\times d_t^{-1}\ddt 
\left(e_{at}\times \kappa\alpha_{1t}\times e_t^a\times \kappa\alpha_{2t}\right)\cr
&\stackrel{(4)}{=}
\frac{1}{2\ell!}\frac{\ell!}{(\ell_1-1)!(\ell_2-1)!}
\sum_{\substack{(\Gamma_1,\Gamma_2)\in \RR_{g_1,\ell_1}\times \RR_{g_2,\ell_2}\\
g_1+g_2=g\\ \ell_1+\ell_2=\ell+1}}  
p_{210}^{12}(\bbb_{(\Gamma_1\amalg\Gamma_2)t})(\alpha)\cr
&+\frac{1}{2\ell!}\ell\sum_{\Gamma\in \RR_{g-1,\ell+1}}p_{210}^{12}
(\bbb_{\Gamma t})(\alpha)
+\frac{1}{\ell!}\ell(\ell-1)\sum_{\Gamma\in \RR_{g,\ell-1}}p_{120}^{1}
(\bbb_{\Gamma t})(\alpha).
\end{align*}
Here equality $(KEY)$ follows by inserting the definitions of $\zeta$
and $X_{\wh\Gamma}$ in~\eqref{eq:key}.
For $(1)$ we replace the domain of integration $\Delta_{3\wh\Gamma}$ by
$\Delta_{3c(\wh\Gamma)}$, using an appropriate composition of reordering maps
(note that the number of integration variables is the same in both
cases, since it equals the number of vertices of $\wh\Gamma$), and 
move $\ddt G_t$ and $dG_t$ to the right past all the other propagators, 
the additional sign $(-1)^{(n-1)(n(l)-e)}$ resulting from the move of $\ddt G_t$; 
Equality $(2)$ follows from equation~\eqref{eq:fulldifft}, 
absorbing the common sign exponent $(n-1)(n(l)-e)$ in $s(\wh\Gamma,\kappa\alpha)$.
For $(3)$ we apply a shuffle permutation to move $e_t^a$ to the right
past the first tensor factor $\alpha_t^1$ of
$\alpha_t=\alpha_t^1\otimes\dots\otimes\alpha_t^s$ if 
$\wh\Gamma$ belongs to the last set in the disjoint union~\eqref{eq:disj}
(the case of $p_{120}^1$), or past those elements of the
first tensor factor $\alpha_t^1$ of $\alpha_t$ that get cut as a result of cutting the graph 
$\wh\Gamma$ along the marked edge if 
$\wh\Gamma$ belongs to either of the first two
sets in~\eqref{eq:disj} 
(the case of $p_{210}^{12}$). 
For $(4)$ we use the definition~\eqref{eq:bspeccond} of $\bbb_{\Gamma t}$,
the definitions of the operations $p_{210}^{21}$ and $p_{210}^1$, as well as 
equation~\eqref{eq:cardshuffle}. Here the first summand corresponds
to $c(\wh\Gamma)$ being disconnected, and the other two correspond  
to $c(\wh\Gamma)$ being connected. Note that the cyclic permutations
needed in the definitions of $p_{210}^1$ and $p_{120}^1$ are taken care of 
by the group $G_{\cyc}$ and equation~\eqref{eq:actiongauge}. 
We recap and continue
\begin{align*}
&\ddt\fm_{(\ell,g)t}
=\frac{1}{2\ell!}\frac{\ell!}{(\ell_1-1)!(\ell_2-1)!}
\sum_{\substack{(\Gamma_1,\Gamma_2)\in \RR_{g_1,\ell_1}\times \RR_{g_2,\ell_2}\\
g_1+g_2=g\\ \ell_1+\ell_2=\ell+1}}  
p_{210}^{12}(\bbb_{(\Gamma_1\amalg\Gamma_2)t})\cr
&\ \ \ +\frac{1}{2\ell!}\ell\sum_{\Gamma\in \RR_{g-1,\ell+1}}p_{210}^{12}
(\bbb_{\Gamma t})
+\frac{1}{\ell!}\ell(\ell-1)\sum_{\Gamma\in \RR_{g,\ell-1}}p_{120}^{1}
(\bbb_{\Gamma t})\cr
&\stackrel{(1)}{=}
\frac{1}{2\ell!}\frac{\ell!}{(\ell_1-1)!(\ell_2-1)!}\frac{1}{\ell_1\ell_2}
\sum_{\substack{(\Gamma_1,\Gamma_2)\in \RR_{g_1,\ell_1}\times \RR_{g_2,\ell_2}\\
g_1+g_2=g\\ \ell_1+\ell_2=\ell+1}}  
\wh \fp_{2,1,0}(\bbb_{(\Gamma_1\amalg\Gamma_2)t})\cr
&\ \ \ +\frac{1}{2\ell!}\ell\frac{2}{\ell(\ell+1)}\sum_{\Gamma\in \RR_{g-1,\ell+1}}
\wh \fp_{2,1,0}(\bbb_{\Gamma t})
+\frac{1}{\ell!}\ell(\ell-1)\frac{1}{\ell-1}\sum_{\Gamma\in \RR_{g,\ell-1}}\wh \fp_{1,2,0}(\bbb_{\Gamma t})\cr
&\stackrel{(2)}{=}\frac{1}{2}
\sum_{\substack{g_1+g_2=g\\ \ell_1+\ell_2=\ell+1}}\Bigl(\wh\fp_{2,1,0}(\fm_{(g_1,\ell_1)t},\bbb_{(g_2,\ell_2)t})+
\wh \fp_{2,1,0}(\bbb_{(g_1,\ell_1)t},\fm_{(g_2,\ell_2)t})\Bigr)\cr
&\ \ \ +\wh \fp_{2,1,0}(\bbb_{(g-1,\ell+1)t})+
\wh \fp_{1,2,0}(\bbb_{(g,l-1)t})\cr
&\stackrel{(3)}{=}
\sum_{\substack{g_1+g_2=g\\ \ell_1+\ell_2=\ell+1}}\wh \fp_{2,1,0}(\fm_{(g_1,\ell_1)t},\bbb_{(g_2,\ell_2)t})
+\wh \fp_{2,1,0}(\bbb_{(g-1,\ell+1)t})+ 
\wh \fp_{1,2,0}(\bbb_{(g,l-1)t}).
\end{align*}
Here equality $(1)$ follows from the definition~\eqref{eq:210120hat}
of $\wh \fp_{2,1,0}$ and $\wh \fp_{1,2,0}$ in terms of $p_{210}^{12}$ and $p_{120}^1$
using shuffles, together with equivariance~\eqref{eq:actiongauge} of 
$\bbb$ under shuffling of boundary components; 
$(2)$ follows from the definition of $\fm_{(\ell,g)t}$ and
$\bbb_{(\ell,g)t}$ and the obvious generalization of
equation~\eqref{eq:primprod} from $M^2$ to higher powers of $M$;
and for $(3)$ we rewrite the first summand using 
the symmetry of $\wh \fp_{2,1,0}$.

\subsection{Case 2 (exact difference)}\label{ss:case2}

Let $\HH\subset \ker d$ be a subspace complementing $\im d$ in $\ker
d$, and $\wt G_0$ an associated propagator. Let $\mu$ be an
$(n-2)$-form on $\wt M^2$ with $\tau^*d\mu=(-1)^nd\mu$ and  
consider the smooth family   
$$
  \wt G_t:=\wt G_0+td\mu.
$$ 
Then each $\wt G_t$ is a propagator for $\HH$. Define the family
$\fm_{(\ell,g)t}$ of MC elements and the family $\bbb_{(\ell,g)t}$ as
in \S\ref{ss:defs} by the formulas~\eqref{eq:mGammat},  
\eqref{eq:m}, \eqref{eq:exactdiff} and~\eqref{eq:b}.
The proof of the following result will occupy the rest of this subsection.

\begin{lemma}\label{lem:case2}
In the above setup, the families $\fm_{(\ell,g)t}$ and
$\bbb_{(\ell,g)t}$ satisfy equation~\eqref{eq:gaugealg}. In
particular, all the MC elements $\fm_t$ are gauge equivalent.
\end{lemma}

To prove this, we compute for $\alpha\in\HH^{\otimes s}$:
\begin{align*}
&\ddt\fm_{(\ell,g)t}(\alpha)\cr
&\stackrel{(1)}{=}\frac{1}{\ell !}\sum_{\Gamma\in \RR_{\ell,g}}\sum_{l_1\in \Edge(\Gamma)}(-1)^{s(\Gamma,\alpha)}
\int_{\Delta_{3\Gamma}}G_t\times\dots\times d\mu\times\dots\times G_t\times\alpha\cr
&\stackrel{(2)}{=}
\frac{1}{\ell !}\sum_{\Gamma\in\RR_{\ell,g}}\sum_{l_1\in \Edge(\Gamma)}(-1)^{s(\Gamma,\alpha)+(n(l_1)-1)(n-1)}
\int_{\Delta_{3\Gamma}}d(G_t\times\dots\times \mu\times\dots\times G_t\times\alpha)\cr
& +\frac{1}{\ell !}
\sum_{\Gamma\in \RR_{\ell,g}}
\sum_{l_1\ne l_2\in \Edge(\Gamma)}(-1)^{s(\Gamma,\alpha)}\cr
&\int_{\Delta_{3\Gamma}}
G_t\times\dots\times G_t\times dG_t\times G_t\times\dots\times G_t\times\mu\times G_t\times\dots\times G_t\times\alpha\cr
&\stackrel{(3)}{=}\frac{1}{2\ell !}
\sum_{(\Gamma,l_2)\in \RR_{\ell,g}^{om}}
\sum_{l_1\ne l_2\in \Edge(\Gamma)}(-1)^{s(\Gamma,\alpha)}\cr
&\int_{\Delta_{3\Gamma}}
G_t\times\dots\times G_t\times dG_t\times G_t\times\dots\times G_t\times\mu\times G_t\times\dots\times G_t\times\alpha\cr
&\stackrel{(4)}{=}
\sum_{\substack{(\Gamma,l_2)\in \RR_{\ell,g}^{oms}\\\kappa\in G_{cyc}(\Gamma,l_2)}}\frac{|Sh(\Gamma,l_2)|}{2\ell !}
\sum_{l_1\ne l_2\in \Edge(\Gamma)}
(-1)^{s(\Gamma,\kappa\alpha)}\cr
&\int_{\Delta_{3\Gamma}}
G_t\times\dots\times G_t\times dG_t\times G_t\times\dots\times G_t\times\mu\times G_t\times\dots\times G_t\times\kappa\alpha\cr
&\stackrel{(5)}{=}\frac{1}{2\ell !}\sum_{(\Gamma,l_2)\in\RR_{\ell,g}^{om}}\sum_{l_1\in \Edge(\Gamma)\setminus
\{l_2\}}(-1)^{s(\Gamma,\alpha)}|Sh(\Gamma,l_2)|\cr
&\int_{\Delta_{3c(\Gamma)}}G_t\times\dots\times\mu\times\dots\times G_t\times
\left(e_a\times\kappa\alpha_1\times e^a\times\kappa\alpha_2\right)\cr
&\stackrel{(6)}{=}\frac{1}{2\ell}
\sum_{\substack{(\Gamma_1,\Gamma_2)\in \RR_{g_1,\ell_1}\times \RR_{g_2,\ell_2}\\
g_1+g_2=g,\ell_1+\ell_2=\ell+1}}\frac{\ell !}{(\ell_1-1)!(\ell_2-1)!}
p_{210}^{12}(\bbb_{(\Gamma_1\amalg\Gamma_2)t})(\alpha)\cr
&+\frac{1}{2\ell !}\ell \sum_{\Gamma\in \RR_{g-1,\ell+1}}p_{210}^{12}(\bbb_{\Gamma t})(\alpha)
+\frac{1}{\ell !}\ell(\ell-1) \sum_{\Gamma\in \RR_{g,\ell-1}}p_{120}^1(\bbb_{\Gamma t})(\alpha)\cr
&\stackrel{(7)}{=}\frac{1}{2\ell}
\sum_{\substack{(\Gamma_1,\Gamma_2)\in \RR_{g_1,\ell_1}\times \RR_{g_2,\ell_2}\\
g_1+g_2=g,\ell_1+\ell_2=\ell+1}}\frac{\ell !}{(\ell_1-1)!(\ell_2-1)!}\frac{1}{\ell_1\ell_2}
\wh \fp_{2,1,0}(\bbb_{(\Gamma_1\amalg\Gamma_2)t})(\alpha)\cr
&+\frac{1}{2\ell !}\ell\frac{2}{\ell(\ell+1)}\sum_{\Gamma\in \RR_{g-1,\ell+1}}\wh \fp_{2,1,0}
(\bbb_{\Gamma t})(\alpha)
+\frac{1}{\ell !}\ell(\ell-1)\frac{1}{\ell-1}\sum_{\Gamma\in \RR_{g,\ell-1}}\wh \fp_{1,2,0}(\bbb_{\Gamma t})(\alpha)\cr
&\stackrel{(8)}{=}
\sum_{\substack{g_1+g_2=g,\\ \ell_1+\ell_2=\ell+1}}
\wh \fp_{2,1,0}(\fm_{(g_1,\ell_1)t},\bbb_{(g_2,\ell_2)t})(\alpha)
+\wh \fp_{2,1,0}(\bbb_{(g-1,\ell+1)t})(\alpha)+
\wh \fp_{1,2,0}(\bbb_{(g,l-1)t})(\alpha).
\end{align*}
Here $(1)$ follows from the definition of $\fm_{(\ell,g)t}$ and the
Leibniz rule for $\ddt$; $(2)$ follows from the Leibniz rule for $d$ (note
that $l_1$ is responsible for the position $n(l_1)$ of $\mu$, while $l_2$ is
responsible for the position $n(l_2)$ of $dG$). For $(3)$ we apply
Lemma~\ref{lem:vanish2} below to show vanishing of the first integral,   
view $(\Gamma,l_2)$ as a marked graph, and sum over both possible
orientations of $l_2$ at the cost of additional factor of $1/2$ in the
second integral.  
Here the sign $(-1)^{(n(l_2)-1)(n-1)}$ in front of the second integral is absorbed  
in $s(\Gamma,\alpha)$, so that the integral together with the sign prefactor does not 
depend on the position of $l_2$ which is part of the extension of the labelling for 
$(\Gamma,l_2)$, cf.~Lemma~\ref{lm:exten}.
The argument to get $(4)$ is analogous to the one giving us equality
$(KEY)$ in Case 1. We briefly indicate the main points, 
retaining the notation $Sh(\wh\Gamma)$ and $G_{cyc}(\wh\Gamma)$ from
Case 1 applied to the marked graph $\wh\Gamma=(\Gamma,l_2)$. We
replace the summation over $\RR_{\ell,g}^{om}$ with the one over $\RR_{\ell,g}^{oms}$, 
at the cost of introducing additional summations over $\kappa$ belonging to $G_{cyc}(\wh\Gamma)$
and $\tau$ belonging to a suitable set of shuffles 
using the bijections~\eqref{eq:slab1},\,\eqref{eq:slab2}
and~\eqref{eq:slab3}. The sum over $\tau$ is converted
to multiplication with $Sh(\wh\Gamma)$ by means of passing 
to the quotient space $\wh E_\ell C$.
For $(5)$ we we replace the domain of integration $\Delta_{3\Gamma}$
by $\Delta_{3c(\Gamma)}$, using an appropriate composition of reordering maps, move 
$dG_t=e_a\times e^a$ to the right past all other 
propagators, and then move $e^a$ to the right past the first tensor factor of $\alpha_t$
if $\Gamma$ belongs to the last set in the disjoint
union~\eqref{eq:disj} (the case of $p_{120}^1$), or past the first
several elements of the first tensor factor
of $\alpha_t$
if $\Gamma$ belongs to
either of the first two sets of the disjoint union~\eqref{eq:disj}
(the case of $p_{210}^{12}$).
This is analogous to combining steps (1) and (3) in the homotopy computation 
after equation~\eqref{eq:key}.
For $(6)$ we use the definitions of $\bbb_\Gamma$ and the operations
$p_{210}^{21}$ and $p_{210}^1$, as well as
equation~\eqref{eq:cardshuffle}. Here the first summand corresponds to
$c(\Gamma)$ being disconnected, and the other two correspond  
to $c(\Gamma)$ being connected. Note that the cyclic permutations
needed in the definitions of $p_{210}^1$ and $p_{120}^1$ are taken
care of by the group $G_{\cyc}$ and equation~\eqref{eq:actiongauge}.
Equality $(7)$ follows from definition~\eqref{eq:210120hat} of $\wh
\fp_{2,1,0}$ and $\wh \fp_{1,2,0}$ in terms of $p_{210}^{12}$ and $p_{120}^1$
using shuffles, together with equivariance~\eqref{eq:actiongauge} of 
$\bbb$ under shuffling of boundary components; $(8)$ follows from the
definitions of $\fm_{(l,g)t}$ and $\bbb_{(l,g)t}$ and the symmetry of
$\wh p_{210}$ for the first summand. 
This concludes the proof of Lemma~\ref{lem:case2} modulo the following
vanishing result.

\begin{lemma}\label{lem:vanish2}
For $\alpha,\mu$ as above and each $\Gamma\in \RR_{\ell,g}$ we have
$$
\sum_{\Gamma\in \RR_{\ell,g}}\sum_{l\in
  \Edge(\Gamma)}(-1)^{s(\Gamma,\alpha)+(n(l)-1)(n-1)}
\int_{\Delta_{3\Gamma}}d(G_t\times\dots\times \mu\times\dots\times G_t\times\alpha)=0.
$$
\end{lemma}

\begin{proof}
As usually, we apply Stokes' theorem and then split the boundary into
the main and hidden parts: 
\begin{align*}
 &\sum_{\Gamma\in \RR_{\ell,g}}\sum_{l\in \Edge(\Gamma)}(-1)^{s(\Gamma,\alpha)
 +(n(l)-1)(n-1)}
\int_{\Delta_{3\Gamma}}d(G_t\times\dots\times \mu\times\dots\times G_t\times\alpha)\cr
=&\sum_{\Gamma\in \RR_{\ell,g}}\sum_{l\in \Edge(\Gamma)}(-1)^{s(\Gamma,\alpha)+(n(l)-1)(n-1)}
\int_{\CC_{\Gamma}}d(\wt G_t\times\dots\times \mu\times\dots\times \wt G_t\times\alpha)\cr
=&\sum_{\Gamma\in \RR_{\ell,g}}\sum_{l\in \Edge(\Gamma)}(-1)^{s(\Gamma,\alpha)+(n(l)-1)(n-1)}
\int_{\p\CC_{\Gamma}^{\rm main}}\wt G_t\times\dots\times \mu\times\dots\times \wt G_t\times\alpha\cr
&+\sum_{\Gamma\in \RR_{\ell,g}}\sum_{l\in \Edge(\Gamma)}(-1)^{s(\Gamma,\alpha)+(n(l)-1)(n-1)}
\int_{\p\CC_{\Gamma}^{\rm hidden}}\wt G_t\times\dots\times \mu\times\dots\times \wt G_t\times\alpha\cr
=&0
\end{align*}
Here in the last expression the first summand vanishes by the usual
duality argument as in~\S\ref{ss:maincompute},
see equation~\eqref{eq:gauge-vanish}.
For the second summand we need to consider hidden faces. Unlike in the
computation for the Maurer-Cartan  
relation (see Theorem~\ref{thm:key-vanish}) or in~\S\ref{ss:case1}
above, here the individual integrals over a hidden face may not
vanish. We prove the vanishing of the sum in
Corollary~\ref{cor:gauge-vanish}.  
\end{proof}


\subsection{Main result}\label{ss:mainres}

Now we state and prove the main result of this section, which is a
more precise version of Theorem~\ref{thm:uniqueness-intro} in the Introduction.

\begin{theorem}\label{thm:gauge}
Let $M$ be a closed oriented manifold.
Let $\HH_0$ and $\HH_1$ be two complements of $\im d$ in $\ker d$. 
Define Maurer-Cartan elements elements $\fm_0$ and $\fm_1$ on $\dIBL(H_{dR}^*(M))$
using {\em special} propagators $\wt G_0$ for $\HH_0$ and $\wt G_1$
for $\HH_1$ (whose existence follows from Proposition~\ref{prop:existsprop}).
Then $\fm_0$ is gauge equivalent to $\fm_1$.
\end{theorem}

\begin{proof}
The proof has 3 steps.

{\bf Step 1. }
Let $\{e_a\}$ be a basis of $H^*(M)$. This uniquely defines bases
$\{e_{a0}\}$ and $\{e_{a1}\}$ of $\HH_0$ and $\HH_1$,
respectively. Define the linear interpolation
$$
  e_{at}:=(1-t)e_{a0}+te_{a1},\qquad t\in [0,1].
$$
Note that for each $t\in [0,1]$ the element $e_{at}$ projects to
$e_a\in H^*(M)$. Therefore, the span
$$
  \HH_t:={\rm span}\{e_{at}\}
$$
complements $\im d$ in $\ker d$ and $\{e_{at}\}$ is a basis of
$\HH_t$. Following the convention from~\S\ref{ss:compl}, we denote by
$e^a_t\in \HH_t$ the unique representative of $e^a\in H^*(M)$ in $\HH_t$.

{\bf Step 2. }
We pick a smooth family of propagators 
$\{\wt G_t^1\}_{t\in [0,1]}$ with
$$
d\wt G_t^1=e_{at}\times e^a_t.
$$
There are many ways to do so and the precise choice is irrelevant. One
way is by using the Cauchy problem~\eqref{eq:diffeq}. Here the complement
for defining $d_t^{-1}$ is irrelevant and can be chosen to be constant 
for simplicity.
Another way to fix $\wt G_t^1$ uniquely is to impose Neumann
conditions at the boundary (contraction of the form with any vector
normal to the boundary is zero), see equation~(1) in~\cite{Cappell-DeTurck-Gluck-Miller}. 
Let $\wt G_t^2$ be the special propagator associated to $\wt G_t^1$,
see Lemma~\ref{lem:propagator} and Remark~\ref{rem:spec}.
We now take the subspace
$$
  C_t:=\im\left(\Om^*(M)\ni\gamma\mapsto\int_{x\in M}G_t^2(x,y)\gamma(x)\right)
$$
as the complement of $\ker d$ in $\Om^*(M)$ to define $d_t^{-1}$.
Since $\wt G_t^2$ is special, equation~\eqref{eq:keyortho} yields the key orthogonality
$$
  \HH_t\perp C_t,
$$
used below in Lemma~\ref{lem:vanishbilin}.

{\bf Step 3. }
We define $\lambda_{at}$ and $\lambda^a_t$ by equation~\eqref{eq:lambda} with $d_t^{-1}$
from Step 2. Let $\wt G_t^3$ be the unique solution of the Cauchy problem~\eqref{eq:diffeq} with 
$\wt K:=\wt G^2_0$. Lemma~\ref{lem:case1} shows that the MC elements $\fm_0^3$ and 
$\fm_1^3$ defined by $\wt G_0^3$ respectively $\wt G_1^3$ are gauge
equivalent. We indicate this as  
$$
  \fm_0^3\sim \fm_1^3.
$$
Both propagators $\wt G_0^3$ and $\wt G_0$ are special, in particular
their bilinear forms vanish by Lemma~\ref{lem:vanishexact}(a). Therefore, by 
Lemma~\ref{lem:vanishexact}(b) the difference $\wt G_0^3-\wt G_0$ is exact. Hence, by 
Lemma~\ref{lem:case2} the MC elements $\fm_0^3$ and $\fm_0$ are gauge equivalent,
$$
  \fm_0\sim\fm_0^3.
$$
By Lemma~\ref{lem:vanishbilin} the bilinear form of each $\wt G_t^3$ vanishes. 
In particular, this is true for $\wt G_1^3$. Since the propagator $\wt
G_1$ has vanishing bilinear form as well, by
Lemma~\ref{lem:vanishexact}(b) the difference $\wt G_1^3-\wt G_1$ is
exact. Hence, by Lemma~\ref{lem:case2} the MC elements $\fm_1^3$ and
$\fm_1$ are gauge equivalent. By transitivity of gauge equivalence we
conclude
$$
  \fm_0\sim\fm_1.
$$
\end{proof}

\section{Relation to physical Chern-Simons theory}\label{sec:physics}















\bigskip

In this section we discuss how the results in this paper relate to
perturbative Chern-Simons theory in the physics literature, see for
example~\cite{Witten-jones,Bar-Natan91,Bar-Natan95,Axelrod-Singer-I,
  Axelrod-Singer-II,Kontsevich-feynman,Bott-Taubes,Bott-Cattaneo98,Cattaneo-Mnev}.

We start by sketching this theory, mostly following the exposition in~\cite{Sawon}. 
Let $M$ be a closed oriented $3$-manifold, and $G$ a compact connected
Lie group with an $\Ad$-invariant inner product $\Tr$ on its Lie
algebra $\g$. The Lie group $G$ is usually taken to be $U(N)$ or $SU(N)$.
Let $\AA$ be the space of connections on the trivial prinipal
$G$-bundle $M\times G\to M$. We identify $\AA$ with $\Om^1(M,\g)$ by
writing a connection as $d+A$ with $A\in\Om^1(M,\g)$. The {\em
  Chern-Simons action} $S:\AA\to\R$ is defined by
$$
   S(A) := \frac{1}{4\pi}\int_M\Tr(A\wedge dA+\frac23 A\wedge A\wedge A).
$$
The action of the gauge group $\GG = \Om^0(M,G)$ given by $A\mapsto
g^{-1}dg+g^{-1}Ag$ changes $S(A)$ by integer multiples of $2\pi$, so
the expression $e^{ikS(A)}$ is gauge invariant for $k\in\Z$. Now one
formally considers the {\em partition function} 
$$
  Z_k(M) := \int_\AA e^{ikS(A)}\DD A, \qquad k\in\N.
$$
Since this expression does not depend on any auxiliary data such as a
metric, it is expected to give rise to invariants of the $3$-manifold $M$.

One approach to construct such invariants was pioneered by Witten in
his seminal paper~\cite{Witten-jones}. Using formal properties of path
integrals, he derives a surgery formula for $Z_k(M)$ which makes it
computable from $Z_k(S^3)$ by iterated surgeries. When including a
link in $M$ as a Wilson line, similar arguments lead to the skein
relation for the Jones polynomial. 

The other approach, which is more relevant for us, considers the
``semiclassical limit'' of $Z_k(M)$ as $k\to\infty$ (where $1/k$
plays the role of Planck's constant). By stationary phase
approximation, the integral over $\AA$ localizes at critical points of
$S$, i.e., at flat connections $A$ satisfying
$$
  F_A = dA + A\wedge A = 0.
$$
Writing connnections near a flat connection $A$ as $A+\alpha$ with
$\alpha\in\Om^1(M,\g)$ we get
$$
   S(A+\alpha) - S(A) = \frac{1}{4\pi}\int_M\Tr(\alpha\wedge
   d_A\alpha+\frac23 \alpha\wedge \alpha\wedge \alpha) =: S_A(\alpha), 
$$
where $d_A$ is the second map in the twisted de Rham complex 
$$
   \Om^0(M,\g)\stackrel{d_A}\longrightarrow
   \Om^1(M,\g)\stackrel{d_A}\longrightarrow
   \Om^2(M,\g)\stackrel{d_A}\longrightarrow
   \Om^3(M,\g).
$$
So the contribution of the flat connection $A$ to $Z_k(M)$ is given by
$$
  e^{ikS(A)}\int_\AA e^{ikS_A(\alpha)}\DD \alpha.
$$
Note that the function $S_A(\alpha)$ in the exponent is the sum of a
quadratic and a cubic term. If the linear operator $d_A$ defining the
quadratic term were invertible, then Wick expansion would give a
perturbative expansion of the integral in powers of $1/k$ in terms of 
configuration space integrals associated to trivalent ribbon graphs with 
edges labelled by the inverse of $d_A$. However, this is never the case. 
One way to remedy this is the BRST gauge fixing procedure, which
introduces a gauge fixing boson and two fermionic ghost fields.
As explained in~\cite{Axelrod-Singer-I,Cattaneo-Mnev},
identifying these fields with differential forms of degrees $0$, $2$
and $3$ one arrives at an expansion of the integral in powers of $1/k$
in terms of configuration space integrals associated to trivalent
ribbon graphs with leaves, where the edges are labelled by a
propagator in the sense of~\S\ref{sec:prop} and the leaves are
labelled by $\Delta_A$-harmonic forms.  

For the trivial connection $A=0$, the perturbative expansion of
physical Chern-Simons theory is thus described by the same
configuration space integrals that enter the definition of our
Maurer-Cartan element in~\S\ref{ss:defMC}. However, there are many
noteworthy differences between the two theories. 
\begin{enumerate}
\item Physical Chern-Simons theory is defined for $3$-manifolds,
  while our theory is defined for manifolds of any (odd or even) dimension.
\item Physical Chern-Simons theory requires the choice of a framing of
  the $3$-manifold,
  while our manifolds need not be parallelizable.
\item Physical Chern-Simons theory includes trivalent ribbon graphs
  without leaves, while in our theory these are ruled out by
  condition~\eqref{eq:one-boundary-vertex}. 
\item Physical Chern-Simons theory is derived from the partition
  function for the Chern-Simons action, while our theory is not
  derived in this way from an action functional.
\item Physical Chern-Simons theory has perturbative expansions at
  every flat connection,
  while our theory expands only at the trivial connection.
\item Physical Chern-Simons theory can be naturally enhanced by the
  inclusion of Wilson loops, while such an enhancement has not been
  developed for our theory.
\item Physical Chern-Simons theory couples the configuration space
  integrals to the Lie algebra of a Lie group $G$, 
  while no Lie group enters in our theory.
\end{enumerate}
Items (i)--(iii) say that our theory is more general than
physical Chern-Simons theory in terms of the underlying manifolds,
but more restrictive in terms of the allowable graphs. It would be
interesting to develop a ``curved'' version of our theory including
also graphs without leaves and see how it relates to the physical theory.

A first step towards (iv) may be the description of the canonical
Maurer-Cartan element on the dual cyclic bar complex in terms of a BV
action in~\cite[\S8]{Hajek-thesis}. 

We expect that our theory admits generalizations incorporating
items (v) and (vi). Expanding it around a nontrivial flat connection
should correspond to string topology with twisted coefficients, while
incorporating boundary conditions on submanifolds should lead to
relative string topology in the sense
of~\cite{Basu-thesis,Basu-McGibbon-Sullivan-Sullivan}. 

Item (vii) is still the most perplexing one for us. From experts in
Chern-Simons theory we have obtained at least three different
explanations for this discrepancy: we are doing the theory for the
gauge group $U(1)$; we are doing the ``universal'' part of the theory
which involves only configuration space integrals and can be coupled
to any Lie group; we are doing the large $N$ limit of the theory with
gauge group $U(N)$. We currently find the last explanation most
convincing in view of recent work by Ginot, Gwilliam, Hamilton
and Zeinalian~\cite{Ginot-Gwilliam-Hamilton-Zeinalian} relating the
large $N$ limit of $U(N)$ Chern-Simons theory to loop space homology
on the classical level. 

Further open questions concern the relation of our theory
to Kontsevich's graph complexes~\cite{Kontsevich-formal},
to Witten's work relating Chern-Simons theory to string
theory~\cite{Witten95} (see also~\cite{Ekholm-strings}),
and to Costello's work on topological conformal field
theories~\cite{Costello-TCFT-gauge,Costello-TCFT-CY},


\appendix

\section{Blow-ups and compactifications}\label{sec:blow-up}

In this section we recall background on manifolds with
corners, blow-ups, and compactified configuration spaces
from~\cite{Cieliebak-Volkov-stringtop}. 
Throughout this section, $X$ denotes a manifold without boundary. All
submanifolds of $X$ will have no boundary as well.  

\subsection{Manifolds with corners}\label{ss:corners}

Let $Y$ be  a {\em manifold with
corners}. Recall that is is defined like a manifold,
with open subsets of $\R^n$ replaced by open subsets of $[0,\infty)^n$;
see e.g.~\cite{Joyce-corners}. 
For $k\geq 0$ we denote by $\p_kY$ its codimension $k$ stratum (where
exactly $k$ of the local coordinates are zero), and by $\p_{\ge k}Y$
the (closed) union of strata of codimension at least $k$. 
The interior $\p_0Y$ will also sometimes be denoted by $Y_0$. By a
``closed submanifold'' we will mean a submanifold which is closed 
as a subset (such as $\R\subset\C$). 

{\bf Transverse collections. }
Consider a finite collection $\{C_a\}_{a\in\AA}$ of
submanifolds of a manifold $X$. Then two submanifolds $C_1$ and
$C_2$ intersect {\em transversely} if $T_pC_1+T_pC_2=T_pX$ for
all $p\in C_1\cap C_2$.
Assume now inductively that transversality has been
defined for any collection of $m-1$ submanifolds, for some $m\ge 3$.
Then we call a collection of $m$ submanifolds {\em transverse}
if any subcollection of $m-1$ submanifolds is
transverse and its intersection is transverse
to the remaining submanifold. We could extend this notion to manifolds
with corners, but we will not need it. In the following, by a {\em
  manifold with a transverse collection}  
$$
  (X,C) = (X, \{C_a\}_{a\in \AA})
$$
we will mean a manifold (without boundary) $X$ with a transverse
collection of submanifolds $C_a\subset X$ indexed by a finite set $\AA$.

\begin{remark}\label{rem:straighten}
For a transverse collection $(X,C)$, near every point $p$ of $X$ there
is a chart straightening all $C_a$'s passing through $p$.  
\end{remark}

\subsection{Blow-ups}\label{ss:blow-up}

In this subsection we recall some basic facts about oriented real
blow-ups. 
Let $(X, C=\{C_a\}_{a\in \AA})$ be manifold with a transverse
collection as above. We stratify the union $\bigcup_{a\in \AA}C_a$ as 
$$
  \bigcup_{a\in \AA}C_a=\coprod_{\varnothing\ne J\subset \AA}
  X_J,
$$
where for a nonempty subset $J\subset \AA$ we set
\begin{equation}\label{eq:basicstrat0}
  X_J:=\bigcap_{a\in J}C_a\setminus\bigcup
  _{a\in\AA\setminus J}C_a.
\end{equation}  
For $a\in \AA$ we denote by $N_{C_a}:=TX|_{C_a}/TC_a$ the normal bundle
to $C_a$ and introduce its oriented projectivization
$$
   P^+N_{C_a}:=(N_{C_a}\setminus C_a)/\sim,
$$
where $v_1\sim v_2$ if and only if $v_1=t v_2$ for some $t>0$.
This yields a sphere bundle
$$
  P^+N_{C_a}\longrightarrow C_a.
$$
We define the bundle 
$$
  \pi_J:P^+N_J\longrightarrow X_J
$$
(with fibre a product of spheres) as the pullback of the product
bundle 
$$
\prod_{a\in J}P^+N_{C_a}\longrightarrow 
\prod_{a\in J}C_a
$$
under the natural inclusion $X_J\into \prod_{a\in J}C_a$. 
We define (as a set)
$$
   \Bl(X,C) := (X\setminus \bigcup_{a\in \AA}C_a)\amalg
   \coprod_{\varnothing\neq J\subset \AA}P^+N_J.
$$
By a slight abuse of notation, we will often write
$$
  C = \bigcup_{a\in \AA}C_a.
$$
The natural projection 
$$
  \pi: \Bl(X,C)\longrightarrow X
$$
is given by the identity on $X\setminus C$ and by $\pi_J$ on
$P^+N_J$. When there is no risk of confusion we will identify
$X\setminus C$ with its preimage under $\pi$.

\begin{lemma}\label{lem:blow-up}
The set $\Bl(X,C)$ carries the natural structure of a manifold with
corners such that $P^+N_J$ becomes part of the codimension $|J|$ boundary.
\end{lemma}

\begin{proof}
Let us first describe the blow-up $\wt\R^d:=\Bl(\R^d,0)$ of $\R^d$ at the
origin. It is defined semi-algebraically using the incidence relation
$$
  \wt\R^d=\{(l,x)\in P^+\R^d\times \R^d\mid x\in l\}.
$$
Consider the homeomorphism
$$
   \Phi:[0,\infty)\times S^{d-1}\stackrel{\cong}\longrightarrow
     \wt\R^d,\qquad (r,v)\mapsto ([v],rv),
$$
where $S^{d-1}\subset\R^d$ is the unit sphere and $[v]\in P^+\R^d$ is
the ray defined by $v$. Note that its inverse if given by
$\Phi^{-1}(l,x)=(|x|,v)$ for the unique representative $v\in S^{d-1}$
of $l$. We make $\wt\R^d$ a manifold with boundary by declaring $\Phi$
to be a global chart. 

For a linear subspace $E\subset \R^d$, the blow-up 
$\Bl(\R^d,E)$ is
diffeomorphic to the product $\wt\R^{d-\dim E}\times E$ using the
orthogonal splitting $\R^d=E^\perp\oplus E$. For a transverse
collection $C$ of linear subspaces in $\R^d$, the blow-up is
therefore diffeomorphic to the product of several $\wt\R^{d_a}$ and a
linear space. In view of Remark~\ref{rem:straighten}, this gives us
local manifold-with-corner charts for $\Bl(X,C)$. 
The statement about $P^+N_J$ follows immediately from this
description. 
\end{proof}

We call the manifold with corners 
$\Bl(X,C)$ together with the map $\pi$ the 
{\em (oriented real) blow-up of $X$ along $C$}.
The map $\pi$ is called the {\em blow-down map}.


{\bf Proper transforms. }
Let $(X,C)$ be a manifold with a transverse collection and 
$$
  \pi:\Bl(X,C)\rightarrow X
$$ 
the corresponding blow-up. The {\em proper transform} of a subset
$Z\subset X$ is the closure of $Z\setminus C$ in $\Bl(X,C)$,
$$
  PT(Z) := Closure({\pi}^{-1}(Z\setminus C))\subset \Bl(X,C).
$$

\begin{remark}\label{rem:closedness}
In our applications the subset $Z$ will usually be closed in $X$. In
this case we can identify the part of $PT(Z)$ which lies in the
complement of $\pi^{-1}(C)$ with
$Z\setminus C$ by means of $\pi$. 
\end{remark} 

\subsection{A general setup for Stokes' theorem}\label{ss:Stokes-gen}

In this subsection we introduce a general setup for Stokes' theorem. 

{\bf Manifolds with quasi-regular boundary. }
We begin with some definitions from Pawlucki's
article~\cite{Pawlucki}.

\begin{definition}\label{def:gen-q-reg}
Let $L$ be a topological space and $\p^{\rm q-reg} L\subset L$ a closed subset. 
We say that $(L,\p^{\rm q-reg} L)$ (or simply $L$) is {\em a manifold with 
quasi-regular boundary} if the following holds: the difference 
$L\setminus \p^{\rm q-reg} L$ is an oriented manifold; for each point 
$p\in \p^{\rm q-reg} L$
there exists an open neighbourhood $U\subset L$ of $p$ such that 
$U\setminus \p^{\rm q-reg} L$ consists of $m\ge 1$ connected components
$\{U^j_0\}_{j=1}^m$ and each
$U^j:=U_0^j\amalg (\p^{\rm q-reg} L\cap U)$ is a $C^1$-manifold with
boundary, with interior $U_0^j$ and boundary $(\p^{\rm q-reg} L\cap U)$. 
The $U^j$ are called the {\em local regular components} at $p$. The
multiplicity $m$ can depend on $p$ but must be locally
constant. The open subset of $\p^{\rm q-reg} L$ defined by the equation 
$m=1$ is denoted by $\p^{reg} L$ and called the {\em regular boundary}.
If $m=1$ constantly, then we get the well-known notion of an oriented
manifold with boundary. 
\end{definition}


An {\em odd $k$-form} on a manifold
$L$ is a $k$-form $\alpha$ on its orientation double cover $\wt L$ with
$\tau^*\alpha=-\alpha$ for the canonical involution $\tau:\wt L\to\wt L$. 
If $k=\dim L$ and $\alpha$ has compact support, then it has a
well-defined integral $\int_L\alpha$.  
In this terminology, each local regular component $U^j$ at $p\in
\p^{\rm q-reg}L$ induces an odd $0$-form $\eps^j$ on $\p^{\rm q-reg}L$ near $p$
whose value is $1$ on the boundary orientation. The sums
$\eps^1+\dots+\eps^m$ at all $p$ fit together to a $\Z$-valued odd
$0$-form $\eps$ on $\p^{\rm q-reg}L$.

\begin{definition}\label{def:emb-q-reg}
Let $(L,\p^{\rm q-reg} L)$ be a manifold with quasi-regular boundary and $(N,\p N)$
be a manifold with corners. We say that a map 
$\iota:(L,\p^{\rm q-reg} L)\rightarrow (N,\p N)$ is an {\em embedding 
(of a manifold with quasi-regular boundary into a manifold with corners)}
if $\iota$ is injective, it restricts to an embedding between the interiors
$L\setminus \p^{\rm q-reg} L\rightarrow N\setminus\p N$, and it restricts to an embedding 
of a manifold with boundary into a manifold with corners on each local
regular component of $L$. 
\end{definition}


{\bf Pairs and Stokes' theorem. }
Now we introduce our general setup for Stokes' theorem. 

\begin{definition}\label{def:pair}
A {\em pair} $(\YY,\XX)$ consists of a (not necessarily compact)
manifold-with-corners $\YY$ and a closed subset $\XX\subset\YY$
with a decomposition
\begin{equation}\label{eq:basicstrat}
  \XX = \XX_0\amalg\p \XX, \quad\text{where}\quad
  \p\XX := \p\YY\cap \XX \quad\text{and}\quad \XX_0 := \YY_0\cap \XX,
\end{equation}
such that $\XX_0$ is an oriented $d$-dimensional submanifold
of $\YY_0$ whose closure equals $\XX$.
\end{definition}


\begin{definition}\label{def:bdry}
Let $(\YY,\XX)$ be a pair. 
Consider the collection of all open subsets $\XX_1$ of $\p\XX$
such that the natural inclusion $\XX_0\amalg\XX_1\into \YY$
is an embedding of a manifold with quasi-regular boundary.
{\em The quasi-regular boundary 
$\p^{\rm q-reg}\XX$} of $\XX$ is the subset of $\p\XX$ maximal with respect
to this property. The {\em regular boundary} $\p^{\rm reg}\XX$ of $\XX$ is defined by requiring that the multiplicity $m$ be equal to $1$. 
\end{definition}

In the setup of Definition~\ref{def:bdry}, note that
$
\XX_0\amalg \p^{\rm q-reg}\XX
$
is a quasi-regular submanifold of $\YY$, and as such 
its boundary $\p^{\rm q-reg}\XX$ carries an odd $0$-form $\eps$.

Recall the codimension $k$ boundary $\p_k\YY$ of $\YY$ and the
``codimension at least $k$'' part $\p_{\ge k}\YY$ of the boundary of
$\YY$.

\begin{definition}\label{def:hidden}
We define the {\em main} (or {\em primary}) and 
{\em hidden} parts of 
$\p^{\rm q-reg}\XX$ as
$$
\p^{\main}\XX:=\p^{\rm q-reg}\XX\cap\p_1\YY,
\qquad 
\p^{\hidden}\XX:=\p^{\rm q-reg}\XX\cap\p_{\ge 2}\YY.
$$
\end{definition}

\begin{definition}\label{def:Stokes}
We say that {\em Stokes' theorem holds} 
for a pair $(\YY,\XX)$ if for every $d$-form $\alpha\in \Om^d(\YY)$ such that
$supp\,\alpha\cap\XX$ is compact the integral $\int_{\XX_0} \alpha$
exists, and for every $(d-1)$-form $\beta\in \Om^{d-1}(\YY)$ 
such that $supp\,\beta\cap\XX$ is compact we have
\begin{equation*}
  \int_{\XX_0} d\beta=\int_{\p^{\rm q-reg}\XX}\eps\beta.
\end{equation*}
Here $\int_{\p^{\rm q-reg}\XX}\eps\beta$ is understood as the integral of an odd
$(d-1)$-form.
\end{definition}

\subsection{Stokes' theorem for $\CC_\Gamma$}\label{ss:CGamma}

In this subsection we derive from~\cite{Cieliebak-Volkov-stringtop}
that Stokes' theorem holds for the compactified configuration
space $\CC_\Gamma$ from~\S\ref{ss:config}. We begin by recalling its
construction (ignoring orientation issues). 
Throughout this subsection, $\Gamma$ denotes a trivalent ribbon graph of
signature $(k,\ell,g)$
with a chosen extended labelling in the sense of
Definition~\ref{def:labelling}. We will use the notation from \S\ref{sec:graphs}. 
In particular, $s_b$ is the number of leaves on the $b$-th boundary component, 
$s=s_1+\cdots+s_\ell$, $k$ is the number of vertices determined
by~\eqref{eq:k}, $e$ is the number of edges, and 
$$
\bar R_\Gamma:=O_v^{-1}\circ O_e:
\{1,\dots,|\Flag(\Gamma)|\}\longrightarrow \{1,\dots,|\Flag(\Gamma)|\}
$$
is the reordering permutation from~\eqref{eq:reod}. 
In addition, we fix an $n$-dimensional oriented closed manifold
$M$. Recall the contravariant functor $S\to M^S$ from the category of
finite sets to the category of manifolds.  
In particular, the reordering permutation $\bar R_\Gamma$ induces a
diffeomorphism 
$$
   R_\Gamma=M^{\bar R_\Gamma}:M^{3k}\stackrel{\cong}\longrightarrow M^{3k}.
$$ 
We denote the domain and target of the map $R_\Gamma$ by different
symbols $Y_\Gamma$ and $X_\Gamma$, so that we have
$$
   R_\Gamma:Y_\Gamma\longrightarrow X_\Gamma.
$$
Let $\Delta_3\subset M^3$ be the triple diagonal and define the (slim)
{\em vertex diagonal}
$$
   \Delta_3^k\subset Y_\Gamma
$$
with its natural orientation, as well as its image
\begin{equation*}
   \Delta_{3\Gamma}:=
     R_\Gamma(\Delta_3^k)\subset X_\Gamma.
\end{equation*}
Let now $l$ be an edge of $\Gamma$. Let  $\Delta_2\subset M^2$
be the diagonal with its canonical orientation and define the double
diagonal corresponding to the edge $l$,
$$
 \Delta_2^l:=(M^2\times\cdots\times M^2\times \Delta_2\times
 M^2\times\cdots\times M^2)\times M^s\subset X_\Gamma.
$$
Here $\Delta_2$ comes at the position corresponding to the edge $l$ in the
numbering (v) in Definition~\ref{def:labelling}.
The family 
$$
\Delta_{2\Gamma}:=\{\Delta_2^l\}_{l\in \Edge(\Gamma)}
$$
is a transverse collection in $X_\Gamma$. By a slight abuse of
language, we also denote by $\Delta_{2\Gamma}$ the union of all
members of this collection and call it the (fat) {\em edge diagonal}. 
Let
$$
   \widetilde X_\Gamma:=(\widetilde M^2)^e\times M^s
$$
denote the blow-up of $X_\Gamma$ along $\Delta_{2\Gamma}$. 

\begin{proposition}[{\cite[\S9]
{Cieliebak-Volkov-stringtop}}]\label{prop:Stokes}
Stokes' theorem holds for the pair
$$
  \bigl(\widetilde X_\Gamma,\,\CC_\Gamma:=PT(\Delta_{3\Gamma})\bigr).  
$$
\end{proposition}

{\bf Description of $\p^{\rm main}\CC_\Gamma$.}
We denote
$$
   \Delta_2^{\{l\}} := \Delta_2^l\setminus \bigcup_{k\in\Edge(\Gamma)\setminus \{l\}}\Delta_2^k
$$
and let
$$
   \widetilde X_{\Gamma,l}:=(M^2\times\cdots\times M^2\times \widetilde
   M^2\times M^2\times\cdots\times M^2)\times M^s
$$
be the (oriented real) blow-up of $X_\Gamma$ along 
$\Delta_2^l$. According to Definition~\ref{def:hidden} the 
main part of the codimension $1$ boundary of 
$\CC_\Gamma$ is exactly the intersection of 
the quasi-regular boundary of $\CC_\Gamma$ with 
the codimension $1$ stratum of the ambient manifold 
with corners $\wt X_\Gamma$.
Thus
\begin{equation}\label{eq:descr-prim}
  \p^{\rm main}\CC_\Gamma = \coprod_{l\in\Edge(\Gamma)}\p^l\CC_\Gamma, \qquad 
  \p^l\CC_\Gamma:=\p\widetilde X_{\Gamma,l}|_{\Delta_{3\Gamma}\cap\Delta_2^{\{l\}}},
\end{equation}
where $\p^l\CC_\Gamma$ is the restriction of the $S^{n-1}$-bundle $\widetilde
X_{\Gamma,l}\to \Delta_2^l$ to $\Delta_{3\Gamma}\cap\Delta_2^{\{l\}}$. 

\subsection{Vanishing of integrals over hidden faces}

In this subsection we derive the two vanishing results for hidden
faces used in the main body of this paper. We retain the notation from
the previous subsection. 

Consider two adjacent edges $(A,B)$ and $(B,C)$ of the graph $\Gamma$,
where the vertices $A$ and $C$ can be equal but must be different from
$B$, see Figure~\ref{fig:hidden}. Let 
\begin{equation}\label{eq:key-invol}
\tau:\wt X_\Gamma\longrightarrow \wt X_\Gamma\
\end{equation}
be the involution that swaps the $\wt M^2$ factors in $\wt X_\Gamma$
corresponding to the edges $(A,B)$ and $(B,C)$.

\begin{definition}\label{def:symm}
We call a form $\om\in \Om^*(\wt X_\Gamma)$ {\em invariant} if 
\begin{equation}\label{eq:key-symm}
  \tau^*\om=(-1)^{n-1}\om
\end{equation} 
for every pair of edges $(A,B)$ and $(B,C)$ as above.
\end{definition}

\begin{proposition}[{\cite[\S9]
{Cieliebak-Volkov-stringtop}}]\label{prop:cancel}
If $\om\in \Om^*(\widetilde X_\Gamma)$ is invariant, then 
\begin{equation*}
  \int_{\p^\hidden\CC_\Gamma}\om=0.
\end{equation*}
\end{proposition}

\begin{figure}
\begin{center}
\includegraphics[width=\textwidth]{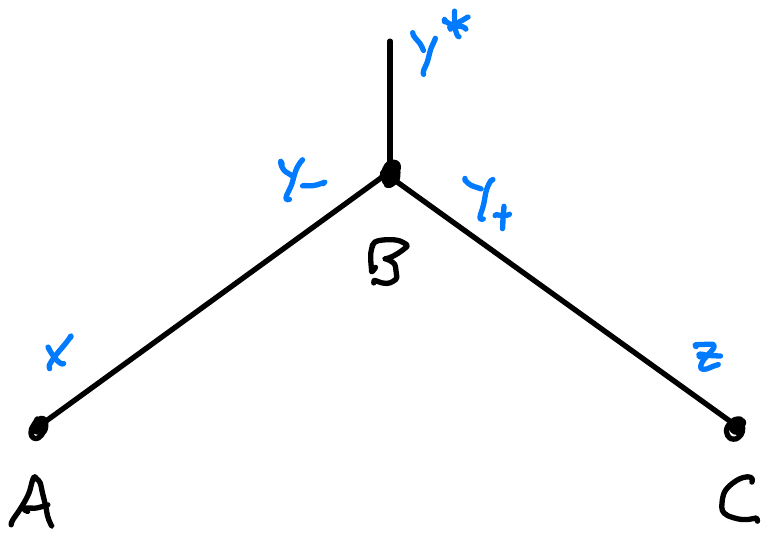}
\vspace{-4cm}
\caption{The involution on hidden faces}
\label{fig:hidden} 
\end{center}
\end{figure}

For the purposes of this paper, we will use two applications of this
result. For the first one, recall from Section~\S\ref{sec:opforms} 
the form $\wt G^e(\alpha)$ on $\wt X_\Gamma$
defined as the cross product of $e$ copies of $\wt G$
times the form $\alpha$ on $M^s$. 

\begin{corollary}\label{cor:Galpha-invar} 
The form $\wt G^e(\alpha)$ on $\wt X_\Gamma$ is invariant, hence
\begin{equation*}
  \int_{\p^\hidden\CC_\Gamma}\wt G^e(\alpha) = 0.
\end{equation*}
\end{corollary}

\begin{proof}
It suffices to consider the map $\tau$ on $\wt M^2\times \wt M^2$
swapping the two factors, ignoring the other factors in $\wt
X_\Gamma$. Let $p_1,p_2$ denote the projections onto the first
respectively second $\wt M^2$ factor. Then
$$
  \tau^*(p_1^*\wt G\wedge p_2^*\wt G) = p_2^*\wt G\wedge p_1^*\wt G =
  (-1)^{n-1}p_1^*\wt G\wedge p_2^*\wt G, 
$$ 
where the first equality follows from $p_1\circ\tau=p_2$ and vice
versa, and the second equality holds because $\wt G$ has degree $n-1$.
\end{proof}


For the second application, 
recall the setup of~\S\ref{ss:case2}: $\HH\subset \ker d$ is a
subspace complementing $\im d$ in $\ker d$ and $\wt
G_0\in\Om^{n-1}(\wt M^2)$ an
associated propagator; $\mu$ is an $(n-2)$-form on $\wt M^2$ and $\wt
G_t=\wt G_0+td\mu$; $\Gamma\in\RR_{\ell,g}$ is a graph with $s$ leaves
and $\alpha\in\HH^{\otimes s}$.  

\begin{corollary}\label{cor:gauge-vanish}
In the setting above,
\begin{equation}\label{eq:gauge-vanish}
  \sum_{l\in \Edge(\Gamma)}(-1)^{s(\Gamma,\alpha)+(n(l)-1)(n-1)}
  \int_{\p\CC_{\Gamma}^{\rm hidden}}\wt G_t\times\dots\times
  \mu\times\dots\times \wt G_t\times\alpha=0.
\end{equation}
Here $(-1)^{s(\Gamma,\alpha)}$ denotes a suitable sign exponent
depending on $\Gamma$ and $\alpha$, and $\mu$ is inserted at the
position $n(l)$ of $l$ in a chosen extension of the labelling of $\Gamma$. 
\end{corollary}

\begin{proof}
In view of Proposition~\ref{prop:cancel}, it suffices to show that the 
form 
\begin{equation*}
  \om:=\sum_{l\in \Edge(\Gamma)}(-1)^{s(\Gamma,\alpha)+(n(l)-1)(n-1)}\wt G_t\times\dots\times
  \mu\times\dots\times \wt G_t\times\alpha
\end{equation*}
is invariant in the sense of Definition~\ref{def:symm}.
Consider two adjacent edges $l_1=(A,B)$ and $l_2=(B,C)$. 
We write
$$
  \om=\om_1+\om_2
$$
with 
$$
  \om_1:=\sum_{l\in \Edge(\Gamma)\setminus\{l_1,l_2\}}(-1)^{s(\Gamma,\alpha)+(n(l)-1)(n-1)}
  \wt G_t\times\dots\times \mu\times\dots\times \wt G_t\times\alpha
$$
and 
$$
  \om_2:=\sum_{l\in\{l_1,l_2\}}(-1)^{s(\Gamma,\alpha)+(n(l)-1)(n-1)}
  \wt G_t\times\dots\times \mu\times\dots\times \wt G_t\times\alpha.
$$
In each summand of $\om_1$ the forms corresponding to $l_1$ and $l_2$
are $\wt G_t$, so the argument in the proof of 
Corollary~\ref{cor:Galpha-invar} goes through 
verbatim to conclude that all summands of $\om_1$
are invariant.
The two summands of $\om_2$ combine to
$$
  \om_2 = (-1)^{s(\Gamma,\alpha)}\nu\times \eta,\qquad
  \nu := \mu\times \wt G_t+(-1)^{n-1}\wt G_t\times \mu,
$$
where $\nu$ corresponds to the edges $l_1$
and $l_2$ (assuming without loss of generality that they have positions
$1$ and $2$ in the extended labelling) and $\eta$ combines the remaining forms.  
As the diffeomorphism $\tau$ preserves $\eta$, it remains to consider
its effect on $\nu$.
Denoting by $p_1,p_2$ the projections onto the two factors, we write
this form more explicitly as
$$
  \nu = p_1^*\mu\wedge p_2^*\wt G_t+(-1)^{n-1}p_1^*\wt G_t\times p_2^*\mu
$$
and compute
\begin{align*}
  \tau^*\nu
  = p_2^*\mu\wedge p_1^*\wt G_t+(-1)^{n-1}p_2^*\wt G_t\times p_1^*\mu 
  = (-1)^{n-1}\nu.
\end{align*} 
Here for the last equality we use the fact that the degrees of $\wt
G_t$ and $\mu$ differ by $1$, so we can swap their order in the wedge 
product without a sign. 
This shows that $\om_2$ is invariant. Combined with the invariance of
$\om_1$, this gives us the desired invariance of $\om$.
\end{proof}

\subsection{An example}\label{ss:ex}

In this subsection we discuss
a simple example of a hidden face that exhibits many of the features
occurring in the general case.  

Consider the graph given by a triangle with $3$ vertices, $3$ edges
and $3$ leaves shown in Figure~\ref{fig:triangle}.
We associate to the vertices the variables $v_1,v_2,v_3\in\R^n$, and
we wish to blow up in $(\R^n)^3$ the three diagonals 
$$
   \Delta_1=\{v_2=v_3\},\quad \Delta_2=\{v_3=v_1\},\quad
   \Delta_3=\{v_1=v_2\}.  
$$
Unfortunately, the collection 
$\{\Delta_j\}_{j=1,2,3}$ does not intersect transversely. To get around
this difficulty, we make the linear coordinate change
$$
   \R^n\times\R^n\times\R^n\stackrel{\cong}{\longrightarrow}
   \{(x_0,x_1,x_2,x_3)\in(\R^n)^4\mid x_1+x_2+x_3=0\}
$$
given by the equations
$$
   x_0=v_1,\quad x_1=v_2-v_3,\quad x_2=v_3-v_1,\quad x_3=v_1-v_2.  
$$
\begin{figure}
\begin{center}
\includegraphics[width=\textwidth]{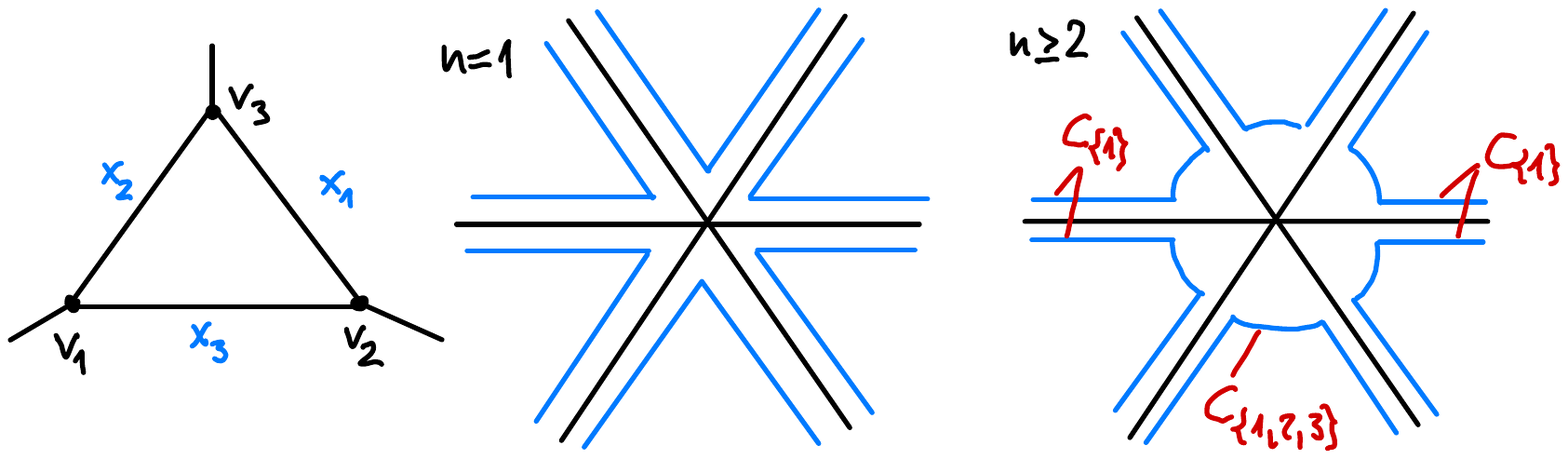}
\vspace{-3.5cm}
\caption{The triangle graph and its compactified configuration space}
\label{fig:triangle} 
\end{center}
\end{figure}
Ignoring the irrelevant variable $x_0$, we have transformed the
problem to blowing up the subspaces $\{x_i=0\}$ in the space
$\{(x_1,x_2,x_3)\in(\R^n)^3\mid x_1+x_2+x_3=0\}$. 
Let 
$$
   \wt X = \Bl(\R^n,0)\times \Bl(\R^n,0)\times
   \Bl(\R^n,0)\longrightarrow X=\R^n\times\R^n\times\R^n
$$
be the oriented real blow-up of $X$ at the three ``diagonals''
$$
   \Delta_i=\{x=(x_1,x_2,x_3)\in X\mid x_i=0\},\qquad i=1,2,3.
$$ 
We define the desired blow-up as the proper transform
$$
  C := PT(B)\subset \wt X,\qquad B:=\{x\in X\mid x_1+x_2+x_3=0\}.
$$
The projection $\wt X\to X$ restricts to a projection $\pi:C\to B$.  
For a (possibly empty) subset $I\subset\{1,2,3\}$ we set 
$$
   C_I := \pi^{-1}\Bigl(\{x\in B\mid x_i=0\Leftrightarrow i\in I\}\Bigr).
$$
These sets define a stratification $C=\amalg_IC_I$ which we will now
describe in detail. Figure~\ref{fig:triangle} depicts the true picture
for $n=1$ and a schematic picture for $n\geq 2$.
First note that $\pi$ restricts to a diffeomorphism on
the open dense stratum
$$
   C_\varnothing = \{x\in (\R^n\setminus 0)\times (\R^n\setminus
   0)\times (\R^n\setminus 0) \mid x_1+x_2+x_3=0\} 
   \cong (\R^n\setminus 0)\times (\R^n\setminus 0). 
$$
We write $S^{n-1}=\p\Bl(\R^n,0)$ for the space of rays at the origin,
and we write $x\in l$ if $x$ is a {\em nonzero} vector in the ray
$l\in S^{n-1}$. Then
$$
   C_{\{1\}} = \{(l_1,x_2,x_3)\in S^{n-1}\times (\R^n\setminus
   0)\times (\R^n\setminus 0) \mid x_2+x_3=0\} 
   \cong S^{n-1} \times (\R^n\setminus 0)
$$
is a codimension $1$ boundary stratum of $C$, and
similarly for $C_{\{2\}}$ and $C_{\{3\}}$. Finally, we have
$$
   C_{\{1,2\}} = C_{\{2,3\}} = C_{\{1,3\}} = \varnothing
$$
and
\begin{equation}\label{eq:123}
\begin{aligned}
  C_{\{1,2,3\}} = &\{(l_1,l_2,l_3)\in S^{n-1}\times S^{n-1} \times
  S^{n-1} \mid t_1x_1+t_2x_2+t_3x_3=0 \\
  &\ \ \text{ for some } x_i\in l_i,\,t_i\geq 0\text{ with }t_1+t_2+t_3>0\}. 
\end{aligned}
\end{equation}
Note that $C_{\{1,2,3\}}\subset C$ is closed. The closure of $C_{\{1\}}$ is the
union of $C_{\{1\}}$ with
$$
   \ol{C_{\{1\}}}\cap C_{\{1,2,3\}} = \{(l_1,l_2,l_3)\in S^{n-1}\times S^{n-1} \times
   S^{n-1} \mid l_2=-l_3\} \cong S^{n-1}\times S^{n-1}, 
$$
and similarly for $C_{\{2\}}$ and $C_{\{3\}}$.
Recall that the codimension $2$ part of $C$ is by definition the
saturation of the union of these sets with respect to the natural map 
sending rays to lines. Observe, however, that this union is already
saturated. To see this, consider a point $(l_1,l_2,-l_2)\in
\ol{C_{\{1\}}}\cap C_{\{1,2,3\}}$ and assume that it does not belong
to any of the other two sets, i.e.~$l_1\ne \pm l_2$. Let 
$(\bar l_1,\bar l_2,\bar l_2)$ denote the corresponding image in the  
unoriented blow up. Then the only points
we have to add to get the full preimage 
of $(\bar l_1,\bar l_2,\bar l_2)$ are the 
points $(-l_1,l_2,-l_2),(\pm l_1,-l_2,l_2)\in \ol{C_{\{1\}}}\cap C_{\{1,2,3\}}$.
It remains to consider a  ``completely degenerate'' point
$(l,l,-l)\in \ol{C_{\{1\}}}\cap \ol{C_{\{2\}}}\cap C_{\{1,2,3\}}$.
In this case, to get the full preimage 
we have to add $(l,-l,l), (-l,l,-l)\in \ol{C_{\{1\}}}\cap \ol{C_{\{3\}}}\cap C_{\{1,2,3\}}$
and $(-l,l,l), (l,-l,-l)\in \ol{C_{\{2\}}}\cap \ol{C_{\{3\}}}\cap C_{\{1,2,3\}}$
as well as $(-l,-l,l)\in \ol{C_{\{1\}}}\cap \ol{C_{\{2\}}}\cap C_{\{1,2,3\}}$.

The complement of these
sets in $C_{\{1,2,3\}}$ is empty for $n=1$, while for $n\geq 2$ it is
given by
\begin{equation}\label{eq:123reg}
\begin{aligned}
   C_{\{1,2,3\}}^\reg 
   &:= C_{\{1,2,3\}}\setminus(\ol{C_{\{1\}}}\cup \ol{C_{\{2\}}}\cup
   \ol{C_{\{3\}}}) \cr 
   &= \{(l_1,l_2,l_3)\in C_{\{1,2,3\}} \mid l_i\neq -l_j \text{ for
     all }i,j\} \cr
   &\cong\{(l_1,l_2)\in S^{n-1}\times S^{n-1}\mid l_1\neq\pm l_2\}\times(0,1). 
\end{aligned}
\end{equation}
Here for the last isomorphism note that $C_{\{1,2,3\}}^\reg$ is the
manifold of $2$-dimensional (nondegenerate) fans in $\R^n$, i.e., triples of rays
$(l_1,l_2,l_3)$ whose positive linear combinations span a plane. Given
$(l_1,l_2)$ with $l_1\neq \pm l_2$, the possible rays $l_3$ form an
open interval whose boundary points correspond to $l_3=-l_1$ and
$l_3=-l_2$. 

The manifold $C_{\{1,2,3\}}^\reg$ is part of the codimension $1$ boundary of
$C$. A tubular neighbourhood is given by
$$
   [0,\infty)\times C_{\{1,2,3\}}^\reg\to C,\qquad
  (r,l_1,l_2,l_3)\mapsto (l_1,rx_1,l_2,rx_2,l_3,rx_3),
$$
where $x_i\in l_i$ are the unique vectors with $x_1+x_2+x_3=0$ and
$|x_1|^2+|x_2|^2+|x_3|^2=1$, and we write points in $\Bl(\R^n,0)$ as
$(l,x)$ with $x\in l$. Altogether,
$$
   C_1 := C_{\{1\}}\cup C_{\{2\}}\cup C_{\{3\}}\cup C_{\{1,2,3\}}^\reg
$$
is the codimension $1$ boundary of $C_0\cup C_1$, where
$C_0:=C_\varnothing$, and the remaining set
$$
   C_2 := C\setminus(C_0\cup C_1) = \bigcup_{i=1}^3(\ol{C_{\{i\}}}\cap
   C_{\{1,2,3\}}) 
$$
has codimension at least $2$. The sets in the last union are not
disjoint, but intersect pairwise in 
$$
   \ol{C_{\{1\}}}\cap \ol{C_{\{2\}}} = \{(l_1,l_2,l_3)\in
   (S^{n-1})^3\mid l_1=l_2=-l_3\}\cong S^{n-1}\subset C_{\{1,2,3\}},
$$
and similarly for the other two pairs. The triple intersection
$\ol{C_{\{1\}}}\cap \ol{C_{\{2\}}}\cap \ol{C_{\{3\}}}$ is empty. The
set
$$
   (\ol{C_{\{1\}}}\cap C_{\{1,2,3\}})\setminus(\ol{C_{\{2\}}}\cup \ol{C_{\{3\}}})
$$
is actually a codimension $2$ corner at which the codimension $1$
boundary strata $C_{\{1\}}$ and $C_{\{1,2,3\}}^\reg$ meet, and
similarly for the other two combinations. However, points of
$\ol{C_{\{i\}}}\cap \ol{C_{\{j\}}}$ are not corners of the appropriate
codimension. Therefore, it is not immediate that Stokes' theorem holds
on spaces of this kind. Nonetheless, if we perform our constructions
with a closed manifold in place of $\R^n$ (to gain compactness), then
the space we obtain in place of $C$ is sufficiently nice that Stokes'
theorem does indeed hold, see Proposition~\ref{prop:Stokes}. The main point
is that $C_2$ has codimension at least $2$. 

Finally, let us have a look at the partial blow-ups
$$
   \wt X\stackrel{p_1}{\longrightarrow}\wt X_1=\Bl(\R^n,0) \times \R^n
   \times \R^n \stackrel{q_1}{\longrightarrow} X.  
$$
Here the map $p_1$ blows down all the 
diagonals except the first one, and $q_1$ blows down the remaining first 
diagonal. Denote by $p_1C$ the proper transform of $\{x_1+x_2+x_3=0\}$ 
in $\wt X_1$. This is a manifold with boundary given by
$$
   \p p_1C = \{(l_1,x_2,x_3)\in S^{n-1}\times\R^n\times\R^n\mid
   x_2+x_3=0\} \cong S^{n-1}\times\R^n. 
$$
The preimage of the boundary is
$$
   p_1^{-1}(\p p_1C) = C_{\{1\}}\cup C_{\{1,2,3\}},
$$
where 
\begin{equation}\label{eq:partbup}
\begin{aligned}
  p_1|_{C_{\{1\}}}: C_{\{1\}} \to  
   &\{(l_1,x_2,x_3)\in S^{n-1}\times(\R^n\setminus
   0)\times(\R^n\setminus 0)\mid x_2+x_3=0\} \cr
   &\cong S^{n-1}\times(\R^n\setminus 0)
\end{aligned}
\end{equation}
is the identity diffeomorphism onto a full measure subset of $\p p_1C$ and 
$$
   p_1|_{C_{\{1,2,3\}}}: C_{\{1,2,3\}}\to\{(l_1,0,0)\mid l_1\in
   S^{n-1}\} \cong S^{n-1}
$$
is surjective. Analogous properties hold for $p_2$ and $p_3$ defined
in the same way. We see that 
$$
   C_{\{1,2,3\}} = p_1^{-1}(\p p_1C)\cap p_2^{-1}(\p p_2C)\cap p_3^{-1}(\p p_3C)
$$
is the ``overlap'' of the preimages of the boundaries under the three
projections. These overlaps (called {\em hidden faces}) present a
major difficulty in our proof of the Maurer-Cartan
equation~\eqref{eq:MCIBL}. Indeed, integrals over primary parts like
$C_{\{1\}}$ can be rewritten as integrals over single blow-ups 
using~\eqref{eq:partbup} and then incorporated in our algebraic
formalism. Integrals over the hidden parts like $C_{\{1,2,3\}}^{\reg}$  
cannot be incorporated in our formalism in any 
sensible way, so the only way to prove 
equation~\eqref{eq:MCIBL} is to show that all such integrals
vanish. The main idea is to produce a certain  
self-diffeomorphism and use invariance of integration. In our 
example above the self-diffeomorphism 
of $\wt X$ swapping any two $\Bl(\R^n,0)$ factors of $\wt X$ 
preserves $C_{\{1,2,3\}}$ and 
$C_{\{1,2,3\}}^{\reg}$ since the defining equations~\eqref{eq:123} 
and~\eqref{eq:123reg} are invariant under such a swap. Moreover, one
can see explicitly that the restriction of this swap to
$C_{\{1,2,3\}}^{\reg}$ preserves orientation if $n$ is even, and reverses
orientation if $n$ is odd.


\bibliographystyle{abbrv}
\bibliography{./000_chen}

\end{document}